\documentclass[12pt,centertags,oneside]{amsart}
\usepackage{amsmath,amstext,amsthm,amscd,typearea}
\usepackage{amssymb}
\usepackage{fancyhdr}
\usepackage[mathscr]{eucal}
\usepackage{mathrsfs}
\usepackage{concmath}
\usepackage{charter}
\usepackage{dsfont}
\usepackage{color}

\usepackage[a4paper,width=16.2cm,top=3cm,bottom=3cm]{geometry}

\numberwithin{equation}{section}

\newcommand{\Complex}{\mathbb C}
\newcommand{\Real}{\mathbb R}
\newcommand{\ddbar}{\overline\partial}
\newcommand{\pr}{\partial}
\newcommand{\ol}{\overline}
\newcommand{\Td}{\widetilde}
\newcommand{\norm}[1]{\left\Vert#1\right\Vert}
\newcommand{\abs}[1]{\left\vert#1\right\vert}
\newcommand{\rabs}[1]{\left.#1\right\vert}
\newcommand{\set}[1]{\left\{#1\right\}}
\newcommand{\To}{\rightarrow}
\newcommand{\ov}{\overline}
\DeclareMathOperator{\Ker}{Ker}
\DeclareMathOperator{\End}{End}
\DeclareMathOperator{\Id}{Id}

\newcommand{\imat}{\sqrt{-1}}
\newcommand{\field}[1]{\mathbb{#1}}

\newcommand{\R}{\field{R}}
\newcommand{\C}{\field{C}}
\newcommand{\N}{\field{N}}

\newcommand{\cali}[1]{\mathscr{#1}}
\newcommand{\cO}{\cali{O}} \newcommand{\cI}{\cali{I}}
\newcommand{\cC}{\cali{C}} \newcommand{\cH}{\cali{H}}
\newcommand{\cE}{\cali{E}}

\theoremstyle{plain}
\newtheorem{thm}{Theorem}[section]
\newtheorem{cor}[thm]{Corollary}
\newtheorem{lem}[thm]{Lemma}
\newtheorem{prop}[thm]{Proposition}

\theoremstyle{definition}
\newtheorem{defn}[thm]{Definition}
\newtheorem{setup}[thm]{Setup}

\theoremstyle{remark}
\newtheorem{rem}[thm]{Remark}

\numberwithin{equation}{section}

\begin{document}
\title[]{Asymptotics of spectral function of lower energy forms and Bergman kernel of semi-positive and big line bundles}
\author[]{Chin-Yu Hsiao}
\address{Universit{\"a}t zu K{\"o}ln,  Mathematisches Institut,
    Weyertal 86-90,   50931 K{\"o}ln, Germany}
\thanks{First-named author is supported by the DFG funded project MA 2469/2-1}
\email{chsiao@math.uni-koeln.de}
\author[]{George Marinescu}
\address{Universit{\"a}t zu K{\"o}ln,  Mathematisches Institut,
   Weyertal 86-90,   50931 K{\"o}ln, Germany\\
    \& Institute of Mathematics `Simion Stoilow', Romanian Academy,
Bucharest, Romania}
\thanks{Second-named author is partially supported by the SFB/TR 12}
\email{gmarines@math.uni-koeln.de}
\date{\today}

\setlength{\headheight}{14pt}
\pagestyle{fancy}
\lhead{\itshape{Chin-Yu Hsiao \& George Marinescu}}
\rhead{\itshape{Asymptotics of spectral function of lower energy forms}}
\cfoot{\thepage}

\begin{abstract}
In this paper we study the asymptotic behaviour of the spectral function corresponding to the lower part of the spectrum of the Kodaira Laplacian on high tensor powers of a holomorphic line bundle.
This implies a full asymptotic expansion of this function on the set where the curvature of the line bundle is non-degenerate.
As application we obtain the Bergman kernel asymptotics for adjoint semi-positive line bundles over complete K\"ahler manifolds, on the set where the curvature is positive. We also prove the asymptotics for big line bundles endowed with singular Hermitian metrics with strictly positive curvature current. In this case the full asymptotics holds outside the singular locus of the metric.
\end{abstract}
\noindent

\maketitle
\tableofcontents 

\section{Introduction and statement of the main results}

Let $L$ be a holomorphic line bundle over a complex manifold $M$ and let $L^k$ be the $k$-th tensor power of $L$.
The Bergman kernel is the smooth kernel of the orthogonal projection onto the space of $L^2$-integrable holomorphic sections of $L^k$.
The study of the large $k$ behaviour of the Bergman kernel is an active research subject in complex geometry and is closely related to topics like the structure of algebraic manifolds (e.\,g.\ \cite{De:85}, \cite{Si1:84}), the existence of canonical K\"ahler metrics (e.\,g.\ \cite{CDS1}, \cite{CDS2},\cite{CDS3}, \cite{Do:01}, \cite{Fine08}, \cite{Fine10}, \cite{Tian}, \cite{Tian2}), Berezin-Toeplitz quantization (e.\,g.\ \cite{BMS94}, \cite{Eng02}, \cite{MM08b}, \cite{MM11}), equidistribution of zeros of holomorphic sections (e.\,g.\ \cite{CM11}, \cite{DMS11}, \cite{ShZ99}, \cite{SZ02}),
quantum chaos and mathematical physics \cite{DouKle10}.
We refer the reader to the book \cite{MM07}
for a comprehensive study of the Bergman kernel and its applications and also to the survey \cite{Ma10}.

In the case of a positive line bundle $L$ over a compact base manifold $M$, 
D.~Catlin~\cite{Cat97} and S.~Zelditch~\cite{Zel98} established the asymptotic expansion of the Bergman kernel on the diagonal by using a fundamental result by Boutet de Monvel-Sj\"{o}strand~\cite{BouSj76} about the asymptotics of the Szeg\"{o} kernel on a strictly pseudoconvex boundary. 
It was already pointed out by Boutet de Monvel-Guillemin~\cite{BouGu81} that the Bergman kernel of $L^k$ is linked to the Szeg\"o kernel on the boundary of the unit disc bundle (Grauert tube), which is strictly pseudoconvex if $L$ is positive.

X. Dai, K. Liu and X. Ma \cite{DLM04a}, \cite {DLM06} obtained the full off-diagonal asymptotic expansion and Agmon estimates of the Bergman kernel for a high power of positive line bundle on a compact complex manifold by using the heat kernel method. Their result holds actually for the more general
Bergman kernel of the spin$^{c}$ Dirac operator associated to a positive line bundle on a
compact symplectic manifold.
In \cite{MM04}, \cite{MM07}, \cite{MM08a}, X. Ma and the second-named author proved the asymptotic expansion for yet another generalization of the Kodaira Laplacian, namely the renormalized Bochner-Laplacian on a symplectic manifold and also showed the existence of the estimate on a large class of non-compact manifolds. 
The main analytic tool in \cite{DLM04a}, \cite {DLM06}, \cite{MM04}, \cite{MM07}, \cite{MM08a} is the analytic localization technique in local index theory developed by Bismut-Lebeau~\cite{MM07}.

Another proof of the existence of the full asymptotic expansion for the Bergman kernel for a high power of a positive line bundle on a compact complex manifold was obtained by B. Berndtsson, R. Berman and
J. Sj\"{o}strand \cite{BBS04}.

A natural generalization is the asymptotic of the kernel of the projection on the harmonic forms in the case of a line bundle with non-degenerate curvature. 
R. Berman and J. Sj\"{o}strand~\cite{BS05} obtained these asymptotics building on the heat equation method of Menikoff-Sj\"{o}strand~\cite{MS78}.
More generally, the expansion in the non-degenerate case was obtained independently by X. Ma and the second-named author~\cite{MM06} for the kernel of the projection on the kernel of the spin$^{c}$ Dirac operator on symplectic manifolds.
The asymptotics of the Szeg\"o kernel for forms on a compact abstract CR manifold with non-degenerate Levi form were established by the first-named author \cite{Hsiao08}.

For a singular Hermitian metric $h^L$ on $L$ with strictly positive curvature current, X.\ Ma and the second-named author used the generalized Poincar\'e metric on the regular locus of $h^L$ and a modified fiber metric on $L$ to obtain a full asymptotic expansion for the associated Bergman kernel~\cite{MM07}, \cite{MM08a}. As a corollary, they could reprove the Shiffman conjecture, asserting that Moishezon
manifolds can be characterized in terms of integral K\"ahler currents.


Witten~\cite{Wit82} suggested that the subcomplex of eigenforms of the Witten Laplacian, correponding to the lower part of the spectrum, is isomorphic to the Thom-Smale complex. This was first made rigorous by Helffer-Sj\"ostrand \cite{HfSj85} by means of microlocal analysis. Inspired by ~\cite{Wit82}, Demailly \cite{De:85} used the subcomplex of eigenforms of the Kodaira Laplacian on $L^k$ in order to prove the holomorphic Morse inequalities (see also Bismut \cite{Bi87}). In this paper we give the first microlocal study of the latter complex.

The first main result of this paper is a local asymptotic expansion
of the spectral function of the Kodaira Laplacian on $L^k$ on a not necessarily compact Hermitian manifold $M$ for states of energy less than $k^{-N_0}$, for $N_0\in\N$ fixed, on the
non-degenerate locus of $L$, cf.\ Theorem \ref{s1-main1}. 
This is a very general result since it holds without global assumptions on the manifold or the line bundle. However, the estimates obtained do not apply directly to the Bergman kernel, which is obtained by formally letting
$N_0\to\infty$ in \eqref{s1-e1main}. 
We then impose a very mild semiclassical local condition on the Kodaira Laplacian, namely the $O(k^{-n_0})$ small spectral gap on an open set $D\Subset M$ (see Definition~\ref{s1-d2bis}). We prove that the Bergman kernel admits an asymptotic expansion on $D$ if 
the Kodaira Laplacian has $O(k^{-n_0})$ small spectral gap on $D$, cf.\ Theorem \ref{s1-main2}.

The distinctive feature of these asymptotics is that they work under minimal hypotheses. This allows us to apply them in situations which were up to now out of reach. We illustrate this in the study of the Bergman kernel of semi-positive or positive but singular
Hermitian line bundles.
We prove that if $M$ is a complete K\"{a}hler manifold and $L$ is semi-positive on $M$, then the Bergman kernel of $L^k\otimes K_M$ admits a full asymptotic
expansion on the non-degenerate locus of $L$, cf.\ Theorem \ref{tmain-semi1}.
Moreover, we show in Theorem \ref{s1-sing-semi-main} that if $M$ is any compact complex manifold and $L$ is semi-positive and positive at some point, then the Bergman kernel of $L^{k}$ admits a full asymptotic expansion on the set where $L$ is positive, with the possible exception of a proper analytic variety $\Sigma\subset M$.

We also consider the case of a singular Hermitian fiber metric on $L$. The holomorphic sections which are $L^2$ with respect to the singular metric turn out to be sections of $L$ twisted with a multiplier ideal sheaf. One can naturally define the orthogonal projection on this space of sections and consider its kernel on the regular locus of the metric. We show that this kernel
has an asymptotic expansion on the regular locus, if the curvature current is strictly positive and smooth outside a proper analytic set (Theorem \ref{sing-main}). This yields yet another proof of the Shiffman conjecture.

We further give formulas for the first top leading terms of the asymptotic expansion of the spectral function and recover the top leading coefficients of the Bergman kernel expansion. These coefficients recently attracted a lot of attention, see the comments after Theorem \ref{s1-main2}.

Other applications of the main results are local and global holomorphic Morse inequalities, a local form of the expansion of the Bergman kernel on forms, examples of manifolds having very small spectral gap, Tian's convergence theorem and equidistribution of zeros of holomorphic sections in the case of singular metrics. 

We now formulate the main results. We refer to Section \ref{s:prelim} for some standard notations and terminology used here.

\subsection{Statement of main results}
Let $(M,\Theta,J)$ be a Hermitian manifold of
complex dimension $n$, where $\Theta$ is a smooth positive $(1,1)$-form and $J$ is the complex structure. Let $g^{TM}_\Theta(\cdot,\cdot)=\Theta(\cdot,J\cdot)$ be the Riemannian metric on $TM$ induced by $\Theta$ and $J$ and let $\langle\,\cdot\,,\cdot\,\rangle$ be the Hermitian
metric on $\Complex TM$
induced by $g^{TM}_\Theta$. 
The Riemannian volume
form $dv_M$ of $(M,\Theta)$ has the form $dv_M= \Theta^n/n!$\,.

Let $(L, h^L)$ be a holomorphic Hermitian line bundle on $M$ and set $L^{k}:=L^{\otimes k}$.
Let $\nabla ^L$ be the holomorphic Hermitian
connection on $(L,h^L)$  with curvature $R^L=(\nabla^L)^2$.
We will identify the curvature form $R^L$ with the Hermitian matrix
$\dot{R}^L \in \cC^\infty(M,\End(T^{(1,0)}M))$ satisfying
for every $U, V\in T^{(1,0)}_xM$, $x\in M$,
\begin{equation}\label{s0-e0}
\langle R^L(x), U\wedge\ov{V}\rangle = \langle \dot{R}^L(x)U, V\rangle.
\end{equation} 
Let $\det\dot{R}^L(x):=\mu_1(x)\ldots\mu_n(x)$, where $\{\mu_j(x)\}_{j=1}^n$, are the eigenvalues of $\dot R^L$ with respect to 
$\langle\,\cdot\,,\cdot\,\rangle$.
For $q\in\{0,1,\dots,n\}$, let 
\begin{equation} \label{s0-e2}
\begin{split}
M(q)=\big\{x\in M;\, &\mbox{$\dot{R}^L(x)\in \End(T^{(1,0)}_x M)$ is non--degenerate}\\
&\quad\mbox{and has exactly $q$ negative eigenvalues}\big\}.
\end{split}
\end{equation}
We denote by $W$ the subbundle of rank $q$ of $T^{(1,0)}M|_{M(q)}$ generated by the eigenvectors corresponding to negative eigenvalues of $\dot{R}^L$. Then $\det\ov{W}^{\,*}:=\Lambda^q\ov{W}^{\,*}\subset\Lambda^qT^{*(0,1)}M|_{M(q)}$ is a rank one subbundle, where $\Lambda^qT^{*(0,1)}M$ is the bundle of $(0,q)$ forms, $\ov{W}^{\,*}$ is the dual bundle of the complex conjugate bundle of $W$ and $\Lambda^q\ov{W}^{\,*}$ is the vector space of all finite sums of $v_1\wedge\cdots\wedge v_q$, 
$v_1,\ldots,v_q\in\ov{W}^{\,*}$. We denote by $I_{\det\ov{W}^{\,*}}\in\End(\Lambda^qT^{*(0,1)}M)$ the orthogonal projection from $\Lambda^qT^{*(0,1)}M$ onto $\det\ov{W}^{\,*}$.

Let $(L^k,h^k)$ be the $k$-th tensor power of $(L,h^L)$, where $h^k:=(h^L)^{\otimes k}$. Let $(\cdot\,,\cdot)_k$ be the inner product 
on $\Omega^{0,q}_0(M,L^k)$ induced by $g^{TM}_{\Theta}$ and $h^k$ (see \eqref{toe2.2}). 
Let $\norm{\cdot}$  be the corresponding norm and let $L^2_{(0,q)}(M,L^k)$ be the completion of $\Omega^{0,q}_0(M,L^k)$ with respect to $\norm{\cdot}$. For $q=0$, we write $L^2(M,L^k):=L^2_{(0,0)}(M,L^k)$.

Let $\Box_{k}^{(q)}$ be the Kodaira Laplacian acting on $(0,q)$--forms with values in $L^k$, cf.\ \eqref{e:Kod_Lap}. We denote by the same symbol $\Box_{k}^{(q)}$ the Gaffney extension of the Kodaira Laplacian, cf.\ \eqref{Gaf1}. It is well-known that $\Box^{(q)}_k$ is self-adjoint and the spectrum of $\Box^{(q)}_k$ is contained in $\ol\Real_+$ (see \cite[Prop.\,3.1.2]{MM07}). 
For a Borel set $B\subset\Real$ we denote by $E(B)$ the spectral projection of $\Box^{(q)}_k$ corresponding to the set $B$, where $E$ is the spectral measure of $\Box^{(q)}_k$ (see Section 2 in Davies~\cite{Dav95}) and for $\lambda\in\Real$ we set $E_\lambda=E\big((-\infty,\lambda]\big)$ and
\begin{equation} \label{s1-specsp}
\cE^q_\lambda(M, L^k)=\operatorname{Range}E_\lambda\subset L^2_{(0,q)}(M,L^k)\,.
\end{equation}
If $\lambda=0$, then 
$\cE^q_0(M, L^k)=\Ker\Box_{k}^{(q)}=:\cH^q(M,L^k)$
is the space of global harmonic sections. For a holomorphic vector bundle over $M$ we have 
\[
H^0_{(2)}(M,E):=\big\{s\in L^2(M,E);\,\ddbar_Es=0\big\}=\Ker\Box^E\,,
\]
where $\ddbar_E$ is the Cauchy-Riemann operator with values in $E$ and $\Box^E$ is the Kodaira Laplacian with values in $E$ (see Section \ref{gaffney}).
The \emph{spectral projection} of $\Box_{k}^{(q)}$ is the orthogonal projection
\begin{equation}\label{e:orto_proj}
P^{(q)}_{k,\lambda}:L^2_{(0,q)}(M,L^k)\to\cE^q_\lambda(M, L^k)\,.
\end{equation}
The \emph{spectral function} $P^{(q)}_{k,\lambda}(\,\cdot\,,\cdot)=P^{(q)}_{k}(\,\cdot\,,\cdot,\lambda)$ is the Schwartz kernel of $P^{(q)}_{k,\lambda}$, see \eqref{specfdef} and \eqref{s1-e2}. Since $\Box^{(q)}_k$ is elliptic, it is not difficult to see that 
\[P^{(q)}_{k}(\,\cdot\,,\cdot\,,\lambda)\in\cC^\infty\big(M\times M,L^k\otimes(\Lambda^qT^{*(0,1)}M\boxtimes\Lambda^qT^{*(0,1)}M)\otimes(L^k)^*\big)\]
and $\cE^q_\lambda(M,L^k)\subset\Omega^{0,q}(M,L^k)$. 
Since $L^k_x\otimes(L^k_x)^*\cong\Complex$ we can identify $P^{(q)}_{k}(x,x,\lambda)$ to an element of $\End(\Lambda^{q}T^{*(0,1)}_xM)$.
Then
\begin{equation}\label{e:state}
X\ni x\longmapsto P^{(q)}_{k}(x,x,\lambda)=P^{(q)}_{k,\lambda}(x,x)\in\End(\Lambda^{q}T^{*(0,1)}_xM)
\end{equation}
is a smooth section of $\End(\Lambda^qT^{*(0,1)}M)$, called \emph{local density of states} of $\cE^q_\lambda(M, L^k)$. 
The trace of $P^{(q)}_k(x,x,\lambda)$ is given by
\[{\rm Tr\,}P^{(q)}_{k,\lambda}(x,x)=
{\rm Tr\,}P^{(q)}_k(x,x,\lambda):=\sum^d_{j=1}\big\langle\,P^{(q)}_k(x,x,\lambda)\,e_{J_j}(x)\,,e_{J_j}(x)\,\big\rangle,\]
where $e_{J_1},\ldots,e_{J_d}$ is a local orthonormal basis of $\Lambda^qT^{*(0,1)}M$ with respect to $\langle\cdot,\cdot\rangle$.  
The projection 
\begin{equation}\label{berg_proj}
P^{(q)}_{k}:=P^{(q)}_{k,0}:L^2_{(0,q)}(M,L^k)\to\Ker\Box^{(q)}_k
\end{equation}
on the lowest energy level $\lambda=0$ is called the \emph{Bergman projection}, its kernel $P^{(q)}_{k}(\,\cdot\,,\cdot)$ is called the \emph{Bergman kernel}. The restriction to the diagonal of $P^{(q)}_{k}(\,\cdot\,,\cdot)$ is denoted $P^{(q)}_{k}(\cdot)$ and is called the \emph{Bergman kernel form}. The function ${\rm Tr\,}P^{(q)}_{k,0}(x,x):={\rm Tr\,}P^{(q)}_k(x)$ is called the \emph{Bergman kernel function}. We notice that ${\rm Tr\,}P^{(0)}_k(x)=P^{(0)}_k(x)$.

We introduce now the notion of asymptotic expansion (see Definition~\ref{s1-d1}).
Let $D\subset M$ be an open set and $a(x,k),a_j(x)\in\cC^\infty(M,\End(\Lambda^qT^{*(0,1)}M))$, $j=0,1,\ldots$ and $m\in\mathbb Z$. We say that
$a(x,k)$ has an asymptotic expansion
\[a(x,k)\sim\sum^{\infty}_{j=0}a_j(x)k^{m-j}\ \ \mbox{locally uniformly on $D$},\]
if for every $N\in\mathbb N_0$, $\ell\in\N_0$ and every compact set $K\subset D$, there exists a constant $C_{N,\ell,K}>0$ independent of $k$, such that  for $k$ sufficiently large
\[\Big|a(x,k)-\sum^N_{j=0}k^{m-j}a_j(x)\Big|_{\cC^\ell(K)}\leq C_{N,\ell,K}\,k^{m-N-1}\,.\]
We say that $a(x,k)=O(k^{-\infty})$ locally uniformly on $D$ if $a(x,k)\sim0$ locally uniformly on $D$. 

The following theorem is one of the main results. It expresses the fact that the Kodaira Laplacian acting on $\Omega^{\bullet,\bullet}(M,L^k)$ admits a \emph{local} semi-classical Hodge decomposition. Note that there are neither global assumptions on the positivity of the bundle nor on the base manifold.
\begin{thm} \label{s1-main1}
Let $(M,\Theta)$ be a Hermitian manifold, $(L,h^L)$ be a holomorphic Hermitian line bundle on $M$.
Fix $q\in\{0,1,\dots,n\}$ and $N_0\geqslant1$. Then for every $m\in\{0,1,\dots,n\}$ there exists a $k$-dependent section $b^{(m)}(x,k)\in\cC^\infty(M(q),\End(\Lambda^{m}T^{*(0,1)}M))$ 
with the following properties: for every $D\Subset M(q)$, $\ell\in\N_0$, there exists a 
constant $C_{D,\ell}>0$ independent of $k$ with
\begin{equation} \label{s1-e1main}
\abs{P^{(m)}_{k}(x,x,k^{-\!N_0})-b^{(m)}(x,k)}_{\cC^\ell(D)}\leqslant C_{D,\ell}\,k^{3n+\ell-N_0}\,,
\end{equation}
$b^{(m)}(x,k)=0$ for $m\neq q$ and $b^{(q)}(x,k)$ has an asymptotic expansion
\begin{equation} \label{s1-e2main}
b^{(q)}(x,k)\sim\sum^\infty_{j=0}b^{(q)}_j(x)k^{n-j}\ \ \mbox{locally uniformly on $M(q)$},
\end{equation}
for some $b^{(q)}_j\in\cC^\infty(M(q),\End(\Lambda^qT^{*(0,1)}M))$, $j=0,1,2,\ldots$\,.
On $M(q)$ we have
\begin{equation}\label{s1-e21main}
b^{(q)}_0=(2\pi)^{-n}\big|\det\dot{R}^L\big|I_{\det\ov{W}^{\,*}}\,.
\end{equation} 
\end{thm}
We wish to give formulas for the top coefficients of the expansion in the case $q=0$. We introduce the geometric objects used in Theorem~\ref{s1-main12} and Theorem~\ref{tmain-semi1} below. 
Put 
\begin{equation} \label{coeI} 
\omega:=\frac{\sqrt{-1}}{2\pi}R^L. 
\end{equation}
On the set $M(0)$ the $(1,1)$-form $\omega$ is positive and induces a Riemannian metric $g^{TM}_\omega(\cdot,\cdot)=\omega(\cdot,J\cdot)$. Let $\nabla^{TM}_\omega$ be the Levi-Civita connection on $(M,g^{TM}_\omega)$, $R^{TM}_\omega=(\nabla^{TM}_\omega)^2$ its curvature (cf.\ \eqref{sa1-e13}), ${\rm Ric}$ its Ricci curvature and $r$ the scalar curvature of $g^{TM}_\omega$ (see \eqref{sa1-e11}). We denote by ${\rm Ric}_{\omega}={\rm Ric\,}(J\cdot,\cdot)$ the $(1,1)$-form associated to ${\rm Ric}$ (cf.\ \eqref{sa1-e15}) and by $\triangle_\omega$ be the complex Laplacian with respect to $\omega$ (see \eqref{sa1-e9}).
We also denote by $\langle\,\cdot\,,\cdot\,\rangle_{\omega}$ the pointwise Hermitian metric induced by $g^{TM}_\omega$ on $(p,q)$-forms on $M$ and by $\abs{\,\cdot\,}_{\omega}$ the corresponding norm.
 
Let $R^{\det}_\Theta$ denote the curvature of the canonical line bundle $K_M=\det T^{*(1,0)}M$ with respect to the metric induced by $\Theta$ (see \eqref{sa1-e12}). 
Put 
\[\widehat r=\triangle_{\omega}\log V_\Theta\,,\quad V_\Theta=\det\left(\Theta_{j,k}\right)^n_{j,k=1}\] 
where $\Theta=\sqrt{-1}\sum^n_{j,k=1}\Theta_{j,k}dz_j\wedge d\ol z_k$ in local holomorphic coordinates $z=(z_1,\ldots,z_n)$.
\begin{thm} \label{s1-main12}
Let $(M,\Theta)$ and $(L, h^L)$ be as in Theorem~\ref{s1-main1}.
The coefficients $b^{(0)}_1$ and $b^{(0)}_2$ in the expansion\eqref{s1-e2main} for 
$q=0$ have the following form:
\begin{equation} \label{coeII} 
\begin{split}
b^{(0)}_1 &= (2\pi)^{-n}\det\dot{R}^L
\Big(\frac{1}{8\pi}r -\frac{1}{4\pi}\Delta_\omega \log\det \dot{R}^L\Big)
\\&=(2\pi)^{-n}\det\dot{R}^L\Bigr(\frac{1}{4\pi}\widehat r-\frac{1}{8\pi} r\Bigr),
\end{split}
\end{equation} 
\begin{equation}\label{coeIII}
\begin{split}
b^{(0)}_2&=(2\pi)^{-n}\det\dot{R}^L\Bigr(\frac{1}{128\pi^2}r^2-\frac{1}{32\pi^2}r\widehat r+\frac{1}{32\pi^2}(\widehat r)^2
-\frac{1}{32\pi^2}\triangle_{\omega}\widehat r
-\frac{1}{8\pi^2}\abs{R^{\det}_\Theta}^2_{\omega}\\
&\quad+\frac{1}{8\pi^2}\langle\,{\rm Ric}_{\omega}\,,R^{\det}_\Theta\,\rangle_{\omega}+\frac{1}{96\pi^2}\triangle_{\omega}r-
\frac{1}{24\pi^2}\abs{{\rm Ric}_{\omega}}^2_{\omega}+\frac{1}{96\pi^2}\abs{R^{TM}_{\omega}}^2_{\omega}\Bigr),
\end{split}
\end{equation} 
where $\abs{R^{TM}_{\omega}}^2_{\omega}$ is given by $\eqref{sa1-e14}$.
\end{thm}
On the set where the curvature of $L$ is degenerate we have the following behaviour.
\begin{thm} \label{s1-maindege}
Let $(M,\Theta)$ and $(L, h^L)$
be as in Theorem~\ref{s1-main1}. 
Set 
\[M_{\mathrm{deg}}=\set{x\in M;\, \mbox{$\dot{R^L}$ is degenerate at $x\in M$}}.\]
Then for every $x_0\in M_{\mathrm{deg}}$, $\varepsilon>0$ and every $m\in\{0,1,\dots,n\}$ , there 
exist a neighborhood $U$ of $x_0$ and $k_0>0$, such that for all $k\geq k_0$ we have
\begin{equation} \label{s1-e3main}
\operatorname{Tr}P^{(m)}_{k}(x,x,k^{-N_0})\leq \varepsilon k^n,\ \ x\in U.
\end{equation}
\end{thm} 

As a Corollary of Theorem~\ref{s1-main1}, Theorem~\ref{s1-main12} and Theorem~\ref{s1-maindege}, we obtain  
\begin{cor} [Local holomorphic Morse inequalities]\label{localmorse}
Let $(M,\Theta)$ and $(L, h^L)$
be as in Theorem~\ref{s1-main1}. Let $N_0\geq 2n+1$. Then the spectral function of the Kodaira Laplacian
has the following asymptotic bahaviour:
\begin{equation} \label{s0-e4mainmore}
{\rm Tr\,}P^{(q)}_{k}(x,x,k^{-N_0})
=k^n(2\pi)^{-n}\abs{{\rm det\,}\dot{R}^L(x)}+O(k^{n-1})\,,\:k\to\infty,
\end{equation} 
locally uniformly on $M(q)$, and if $\mathds{1}_{M(q)}$ denotes the characteristic function of $M(q)$,
\begin{equation} \label{s0-e4main}
\lim_{k\To\infty}k^{-n}{\rm Tr\,}P^{(q)}_{k}(x,x,k^{-N_0})
=(2\pi)^{-n}\abs{{\rm det\,}\dot{R}^L(x)}\mathds{1}_{M(q)}(x),\ \ x\in M.
\end{equation}
Moreover, for every $\varepsilon>0$, every $D\Subset M$, there exists a $k_0>0$, such that for all $k\geq k_0$, we have
\begin{equation} \label{s0-e5main}
{\rm Tr\,}P^{(q)}_{k}(x,x,k^{-N_0})\leq\Bigr(\varepsilon+(2\pi)^{-n}\abs{{\rm det\,}\dot{R}^L(x)}\mathds{1}_{M(q)}(x)\Bigr)k^n,\ \ x\in D,
\end{equation} 
and for $q=0$ and $N_0\geq 2n+3$, we have as $k\to\infty$
\begin{equation} \label{s0-e5mainmore}
\begin{split}
P^{(0)}_{k}(x,x,k^{-N_0})&\leq k^n(2\pi)^{-n}\det\dot{R}^L(x)
+k^{n-1}b^{(0)}_1(x)+k^{n-2}b^{(0)}_2(x)+O(k^{n-3})\,, 
\end{split}
\end{equation} 
locally uniformly on $M(0)$, where $b^{(0)}_1(x)$ and $b^{(0)}_2(x)$ are as in \eqref{coeII} and \eqref{coeIII} 
respectively.
\end{cor} 
The term local holomorphic Morse inequalities is motivated by the fact that when $M$ is compact, integration of the inequalities from Corollary \ref{localmorse} yields the holomorphic Morse inequalities of Demailly, see Section \ref{s:holo_morse}.
Berman~\cite{Be04} proved that 
\[
\limsup_{k\To\infty}k^{-n}{\rm Tr\,}P^{(q)}_k(x)\leq(2\pi)^{-n}\abs{\det\dot R^L(x)}\mathds{1}_{M(q)}(x),\ \ x\in M,\]
and when $M$ is compact, there exists a sequence $\mu_k\To0$, as $k\To\infty$, such that
\[\lim_{k\To\infty}k^{-n}{\rm Tr\,}P^{(q)}_{k}(x,x,\mu_k)
=(2\pi)^{-n}\abs{{\rm det\,}\dot{R}^L(x)}\mathds{1}_{M(q)}(x),\ \  x\in M.\]
Corollary~\ref{localmorse} refines and generalizes Berman's results.

In order to obtain precise asymptotics we combine the local asymptotics from Theorem \ref{s1-main1} with a mild condition on the semiclassical behaviour of the spectrum of the Kodaira Laplacian $\Box^{(q)}_k$ for $k$ large, which we call (local) $O(k^{-n_0})$ small spectral gap.
\begin{defn} \label{s1-d2bis}
Let $D\subset M$. We say that $\Box^{(q)}_k$ has \emph{$O(k^{-n_0})$ small spectral gap on $D$} if there exist constants $C_D>0$,  $n_0\in\mathbb N$, $k_0\in\mathbb N$, such that for all $k\geq k_0$ and $u\in\Omega^{0,q}_0(D,L^k)$, we have  
\[\norm{(I-P^{(q)}_k)u}\leq C_D\,k^{n_0}\norm{\Box^{(q)}_ku}.\]
\end{defn}
To explain this condition, assume that $M$ is a complete Hermitian manifold. Then the operator $\Box^{(q)}_k$ is essentially self-adjoint and $\Omega^{0,q}_0(D,L^k)$ is dense with respect to the graph-norm in the domain of the quadratic form of $\Box^{(q)}_k$ (see e.\,g.\ \cite[\S\,3.3]{MM07}). If $\Box^{(q)}_k$ has $O(k^{-n_0})$ small spectral gap on $M$ then $\inf\set{\lambda\in\operatorname{Spec}(\Box^{(q)}_k);\, \lambda\neq0}\geq Ck^{-n_0}$, for some $n_0\in\mathbb N$ and $C>0$ independent of $k$.

From Theorem~\ref{s1-main1}, Definition~\ref{s1-d2bis} and 
some simple arguments (see Section \ref{s:asymp_BK}), we deduce:
\begin{thm} \label{s1-main2}
Let $(M,\Theta)$ be a Hermitian manifold, $(L, h^L)$ be a holomorphic Hermitian line bundle on $M$.
Fix $q\in\{0,1,\dots,n\}$ and $N_0\geqslant1$. Let $D\subset M(q)$.
If $\Box^{(q)}_k$ has $O(k^{-n_0})$ small spectral gap on $D$, then for every $D'\Subset D$, $\ell\in\N_0$, there exists a constant $C_{D',\ell}>0$ independent of $k$ with
\[\abs{P^{(q)}_{k}(x,x,k^{-\!N_0})-P^{(q)}_k(x)}_{\cC^\ell(D')}\leqslant C_{D',\ell}\,k^{3n+\ell-N_0}\,.\] 
In particular, 
\begin{equation} \label{s1-e3mainbis}
P^{(q)}_k(x)\sim\sum^\infty_{j=0}b^{(q)}_j(x)k^{n-j}\ \ \text{locally uniformly on $D$},
\end{equation}
where $b^{(q)}_j\in\cC^\infty(D,\End(\Lambda^qT^{*(0,1)}M))$, $j=0,1,2,\ldots$\,, are as in \eqref{s1-e2main} and $b^{(q)}_0$, $b^{(0)}_1$, $b^{(0)}_2$ are given by \eqref{s1-e21main}, \eqref{coeII}, \eqref{coeIII}.
\end{thm} 
Note that if $L$ is a positive line bundle on a compact manifold $M$, or more generally $L$ is uniformly positive on a complete manifold $(M,\Theta)$ with $\sqrt{-1}R^{K^*_M}$ and $\partial\Theta$ bounded below, then the Kodaira Laplacian $\Box^{(0)}_k$ has a ``large" spectral gap on $M$, i.e.\ there exists a constant $C>0$ such that for all $k$
we have \[\inf\set{\lambda\in\operatorname{Spec}(\Box^{(0)}_k);\, \lambda\neq0}\geq Ck\,,\]
(see \cite[Th.\,1.5.5]{MM07}, \cite[Th.\,6.1.1, (6.1.8)]{MM07}).
Therefore the Bergman kernel $P^{(q)}_{k}$ has the asymptotic expansion \eqref{s1-e3mainbis}
and we recover from Theorem \ref{s1-main2} the asymptotic expansion of the Bergman kernel for: 
\begin{itemize}
\item[(i)] compact manifolds for $q=0$, \cite{Cat97}, \cite{Zel98} (cf.\ also \cite[Th.\,4.1.1]{MM07}), 
\item[(ii)] compact manifolds for arbitrary $q$, \cite{BS05}, \cite{MM06} \cite[Th.\,8.2.4]{MM07},
\item[(iii)] for complete manifolds \cite[Th.\,3.11]{MM08a}, \cite[Th.\,6.1.1]{MM07}. 
\end{itemize}
In the case $q=0$ the precise formulas \eqref{coeII}, \eqref{coeIII} for the coefficients of the Bergman kernel expansion \eqref{s1-e3mainbis} play an important role in the investigations about the relation between canonical metrics in K\"ahler geometry and stability in algebraic geometry see  e.g.\ \cite{Do:01}, \cite{Fine08}, \cite{Fine10}, \cite{GFR11}, \cite{Tian}, \cite{Wangl03}, \cite{Wang05} (cf.\ also \cite[{\S} 5.2]{MM07}). 

The coefficients $b^{(0)}_{1}$, $b^{(0)}_{2}$ were computed by
Z.~Lu \cite{Lu00}, L.~Wang \cite{Wangl03},  X.~Wang \cite{Wang05}, in various degrees of generality. 
The method of these authors is to construct appropriate peak sections 
as in \cite{Tian}, using 
H{\"o}rmander's $L^2$ $\overline\partial$-method.

In \cite[\S 5.1]{DLM04a}, Dai-Liu-Ma computed  $b^{(0)}_{1}$
by using the heat  kernel, and in \cite[\S 2]{MM08a}, \cite[\S 2]{MM06}
(cf.\ also \cite[\S 4.1.8, \S 8.3.4]{MM07}), $b^{(0)}_{1}$ was
computed in the symplectic case.
The coefficient $b^{(0)}_{2}$ was calculated in \cite[Th.\,0.1]{MM10}
(these results include a twisting Hermitian vector bundle $E$).
Recently, a combinatorial formula for the coefficients $b^{(0)}_{j}$ was obtained in \cite{Xu11} and the formula for $b^{(0)}_{2}$ was rederived in \cite{GFR11}.
In the above mentioned results it was supposed that the curvature $\omega=\frac{\sqrt{-1}}{2\pi}R^L$ equals the underlying metric $\Theta$. If $\omega\neq\Theta$ formulas for $b^{(0)}_{1}$, $b^{(0)}_{2}$ were given in \cite[Th.\,4.1.3]{MM07}, \cite[Remark\,0.5]{MM10}, \cite[Th.\,1.4]{Hsiao09}. We notice that all the results mentioned above concern the coefficients 
of the Bergman kernel expansion and our results \eqref{coeII}, \eqref{coeIII} could recover the first three coefficients of the Bergman kernel expansion in the complex case.

Let $M$ be a compact complex manifold and $L\to M$ be a holomorphic line bundle with non-degenerate curvature
of signature $q\in\{0,1,\ldots,n\}$, i.\,e.\ $M(q)=M$. The coefficient $b^{(q)}_0$ given by \eqref{s1-e21main} appeared in
\cite[Th.\,1.3]{MM06} (the manifold $M$ there is supposed to be symplectic).
The coefficient $b^{(q)}_1$ was calculated recently by Wen Lu \cite{Lu_w13} and by Hsiao \cite{Hsiao12}
for the trivial line bundle with mixed curvature over $\C^n$ endowed with the Euclidean metric.
\par Since we allow a local $O(k^{-n_0})$ small spectral gap, we can obtain the Bergman kernel expansion under weak conditions, such as semi-positivity of the line bundle. 
In this case we have to twist $L^k$ with the canonical line bundle $K_M$, which we endow with the natural Hermitian metric induced by $\Theta$.
We denote by $P_{k,K_M}$ the orthogonal projection from $L^2(M,L^k\otimes K_M)$ on $H^0_{(2)}(M,L^k\otimes K_M)=\cH^0(M,L^k\otimes K_M)$.
\begin{thm} \label{tmain-semi1}
Let $(M,\Theta)$ be a complete K\"ahler manifold and $(L,h^L)$ be a semi-positive line bundle over $M$. 
Then the Bergman kernel function $P_{k,K_M}(\,\cdot\,)$ of $H^0_{(2)}(M,L^k\otimes K_M)$ 
has the asymptotic expansion
\begin{equation}\label{asymcanonical}
P_{k,K_M}(x)\sim\sum^\infty_{j=0}k^{n-j}b^{(0)}_{j,K_M}(x)\ \ \mbox{locally uniformly on $M(0)$},
\end{equation}
where $b^{(0)}_{j,K_M}\in\cC^\infty(M(0),\End(K_M))$, $j=0,1,2,\ldots$\,, are given by
\begin{equation}\label{coeIV}\begin{split}
&b^{(0)}_{0,K_M}=(2\pi)^{-n}{\rm det\,}\dot{R}^L\Id_{K_M}\,,\\
&b^{(0)}_{1,K_M}=(2\pi)^{-n}\det\dot{R}^L\Bigr(-\frac{1}{8\pi} r\Bigr)\Id_{K_M}\,,\\
&b^{(0)}_{2,K_M}=(2\pi)^{-n}\det\dot{R}^L\Bigr(\frac{1}{128\pi^2}r^2
+\frac{1}{96\pi^2}\triangle_{\omega}r-
\frac{1}{24\pi^2}\abs{{\rm Ric\,}_{\omega}}^2_{\omega}+\frac{1}{96\pi^2}\abs{R^{TM}_{\omega}}^2_{\omega}\Bigr)
\Id_{K_M},
\end{split}\end{equation}
where $\abs{R^{TM}_\omega}^2_\omega$ is given by \eqref{sa1-e14} and $\Id_{K_M}$ is the identity map on $K_M$.
\end{thm}
In \cite[Th.1.8]{Be07} the expansion is proved on $M(0)\setminus B_+(L)$ for $M$ compact, where $B_+(L)$ is the augmented base locus of $L$. Note that $L$ is ample if and only if $B_+(L)$ is empty.

Let us consider now a singular Hermitian holomorphic line bundle $(L,h^L)\to M$ (see e.\,g.\ \cite[Def.\,2.3.1]{MM07}). We assume that $h^L$ is smooth outside a proper analytic set $\Sigma$ and the curvature current of $h^L$ is strictly positive.
The metric $h=h^L$ induces singular Hermitian
metrics $h^k$ on $L^k$. We denote by $\cI(h^k)$ the Nadel multiplier ideal sheaf associated to $h^k$ and by $H^0(M,L^k\otimes\cI(h^k))\subset H^0(M,L^k)$ the space of global sections of the sheaf $\cO(L^k)\otimes\cI(h^k)$ (see \eqref{l2:mult}), where $H^0(M,L^k):=\set{u\in\cC^\infty(X,L^k);\, \ddbar_ku=0}$.
We denote by $(\cdot\,,\cdot)_k$ the natural inner products on $\cC^\infty(M,L^k\otimes\cI(h^k))$ induced by $h^L$ and the volume form $dv_M$ on $M$ (see \eqref{l2sing} and see also \eqref{pcsing} for the precise meaning of $\cC^\infty(M,L^k\otimes\cI(h^k))$\,).
Let $\{S^k_j\}^{d_k}_{j=1}$ be an orthonormal basis of $H^0(X,L^k\otimes\cI(h^k))$ with respect to the inner product induced $(\cdot\,,\cdot)_k$.
The (multiplier ideal) Bergman kernel function of $H^0(M,L^k\otimes\cI(h^k))$ is defined by
\begin{equation} \label{sing-e0-2}
P^{(0)}_{k,\cI}(x):=\sum^{d_k}_{j=1}\abs{S^k_j(x)}^2_{h^k},\ \ x\in M\setminus\Sigma\,.
\end{equation}
\begin{thm} \label{sing-main}
Let $(L,h^L)$ be a singular Hermitian holomorphic line bundle over a compact Hermitian manifold $(M,\Theta)$. We assume that $h^L$ is smooth outside a proper analytic set $\Sigma$ and the curvature current of $h^L$ is strictly positive. Then the Bergman kernel function $P_{k,\cI}(\,\cdot\,)$ of $H^0(M,L^k\otimes\cI(h^k))$ has the asymptotic expansion
\begin{equation}\label{s1-e3mainterz}
P^{(0)}_{k,\cI}(x)\sim\sum^\infty_{j=0}k^{n-j}b^{(0)}_j(x)\ \ \text{locally uniformly on $M\setminus\Sigma$},
\end{equation}
where $b^{(0)}_j\in\cC^\infty(M\setminus\Sigma)$, $j=0,1,2,\ldots$\,, $b^{(0)}_0=(2\pi)^{-n}\det\dot{R}^L$ and $b^{(0)}_1$ and $b^{(0)}_2$ are given by \eqref{coeII} and \eqref{coeIII}, respectively.
\end{thm}
\noindent
We obtain in this way another proof of the Shiffman-Ji-Bonavero-Takayama criterion (cf.\ \cite[Th.\,2.3.28,\,2.3.30]{MM07}).
\begin{cor} \label{sing-c1}
Under the assumptions in Theorem~\ref{sing-main}, we have
\[\dim H^0(M,L^k\otimes\cI(h^k))\geq ck^n\]
for $k$ large, where $c>0$ is independent of $k$. Therefore, $L$ is big and $M$ is Moishezon.
\end{cor}

We assume that $(M,\Theta)$ is compact and we set
\[
\operatorname{Herm}(L)=\big\{\text{singular Hermitian metrics on $L$} \big\}\,,
\]
\[\begin{split}
\mathcal{M}(L)=&\big\{h^L\in\operatorname{Herm}(L);\,\text{$h^L$ is smooth outside a proper analytic set}\\
\quad&\text{and the curvature current of $h^L$ is strictly positive}\big\}\,.
\end{split}\] 
Note that by Siu's criterion \cite[Th.\,2.2.27]{MM07}, $L$ is big under the hypotheses of Theorem \ref{s1-sing-semi-main} below. By \cite[Lemma\,2.3.6]{MM07}, $\mathcal{M}(L)\neq\emptyset$.
Set
\begin{equation}\label{s1-sing-semi-set2}
M':=\set{p\in M\,;\, \mbox{$\exists$ $h^L\in\mathcal{M}(L)$ with $h^L$ smooth near $p$}}.
\end{equation}
\begin{thm} \label{s1-sing-semi-main}
Let $(M,\Theta)$ be a compact Hermitian manifold.
Let $(L,h^L)\to M$ be a Hermitian holomorphic line bundle with smooth Hermitian metric $h^L$ having semi-positive curvature and with $M(0)\neq\emptyset$.
Then the Bergman kernel function $P_{k}(\,\cdot\,)$ has the asymptotic expansion
\[P_{k}(x)\sim\sum^\infty_{j=0}k^{n-j}b^{(0)}_j(x)\ \ \mbox{locally uniformly on $M(0)\cap M'$}, \]
where $b^{(0)}_j\in\cC^\infty(M(0))$, $j=0,1,2,\ldots$\,, $b^{(0)}_0=(2\pi)^{-n}{\det\,}\dot{R}^L$ and $b^{(0)}_1$ and $b^{(0)}_2$ are given in \eqref{coeII} and \eqref{coeIII}, respectively. 
\end{thm}
The existence of the asymptotic expansion from Theorem~\ref{s1-sing-semi-main} was obtained by Berman~\cite{Be07} in the case of a projective manifold $M$.

\begin{rem} \label{s1-r1bis}
(I) In Theorems \ref{s1-main1}, \ref{s1-main2}, we obtain the 
diagonal expansion of the kernels $P^{(q)}_{k, k^{-\!N_0}}(\cdot,\cdot)$ . We will prove actually more, namely the off-diagonal asymptotic expansion for $P^{(q)}_{k, k^{-\!N_0}}(x, y)$ on the non-degenerate part of $L$, see Theorem~\ref{tII}, Theorem~\ref{tIII} and Theorem~\ref{t:localspectral} for the details. In the same vein, the diagonal expansions of the Bergman kernels from Theorems \ref{s1-main2}, \ref{tmain-semi1}, \ref{sing-main}, \ref{s1-sing-semi-main} have off-diagonal counterparts. See Theorem~\ref{t-semi1}, 
Theorem~\ref{sing-t5} and Theorem~\ref{t-sing-semi} for the details.
\\[2pt]
\
(II) Let $E$ be a holomorphic vector bundle over $M$. 
Theorem~\ref{s1-main1}, Theorem~\ref{s1-maindege}, Theorem~\ref{s1-main2}, Theorem~\ref{sing-main} and Theorem~\ref{s1-sing-semi-main} and their off-diagonal counterparts can be generalized to the situation when $L^k$ is replaced by $L^k\otimes E$. See Remark~\ref{rabis2} and the discussions in the end of Section 4.4 and Section 5, for the details.
\end{rem} 

The layout of this paper is as follows. In Section \ref{s:prelim} we collect some notations, definitions and statements we use throughout (geometric set-up, self-adjoint extension of the Kodaira Laplacian, Schwartz kernel theorem). In Section \ref{s:sem_hodge} we exhibit a microlocal Hodge decomposition for the Kohn Laplacian on a non-degenerate CR manifold and apply this to obtain the semiclassical Hodge decomposition for the Kodaira Laplacian on a complex manifold.
In Section \ref{s:exp_spec_fct} we prove the existence of the asymptotic expansion of the spectral function associated to forms of energy less that $k^{-N_0}$. As a consequence we obtain the expansion of the Bergman kernel if the local $O(k^{-n_0})$ spectral gap exists. In Section \ref{s:deg_set} we get an asymptotic upper bound near the degeneracy set of the curvature of $L$. In Section \ref{adjoint_semipos} we prove the expansion of the Bergman kernel on the positivity set of an adjoint semi-positive line bundle over a complete K\"ahler manifold.
In Section \ref{s:sing_l2_est} we prove an $L^2$-estimate for the $\ddbar$ for singular metrics.
We use this estimate in Sections \ref{s:exp_semipos} and \ref{S:mibk} to prove the existence of the Bergman kernel expansion for semi-positive line bundles and bundles endowed with a strictly positively-curved singular Hermitian metric. In Section \ref{s:appl} we apply the previous methods to obtain miscellaneous results, such as Bergman kernel expansion under various conditions, holomorphic Morse inequalities, Tian's convergence theorem and equidistribution of zeros of holomorphic sections.

\section{Preliminaries}\label{s:prelim}
\subsection{Some standard notations}
We denote by $\mathbb N=\set{1,2,\ldots}$ the set of natural numbers and by $\Real$ the set of real numbers. We set $\mathbb N_0=\mathbb N\bigcup\set{0}$.
We use the standard notations $w^\alpha$, $\partial_x^\alpha$ for multi-indices $\alpha=(\alpha_1,\ldots,\alpha_m)\in\mathbb N_0^m$, $w\in\C^m$, $\partial_x=(\partial_{x_1},\ldots,\partial_{x_m})$.
%

Let $M$ be a complex manifold of dimension $n$. We always assume that $M$ is paracompact. 
We denote holomorphic charts on $M$ by $(D,z)$, where $z=(z_1,\ldots,z_n):D\to\C^n$ are local coordinates. The associated real chart is denoted by $(D,x)\cong (D,z)$, where $x=(x_1,\ldots,x_{2n})$ are real coordinates on $M$ given by $z_j=x_{2j-1}+ix_{2j}$, $j=1,\ldots,n$. 
For a multi-index
$J=(j_1,\ldots,j_q)\in\{1,\ldots,n\}^q$ we set $\abs{J}=q$. We say
that $J$ is strictly increasing if $1\leqslant
j_1<j_2<\cdots<j_q\leqslant  n$. We put $d\ol z^J=d\ol z_{j_1}\wedge\cdots\wedge d\ol z_{j_q}$. A $(0,q)$-form $f$
on $M$ has the local representation
\[f|_D=\sideset{}{'}\sum_{\abs{J}=q}f_J(z)d\ol z^J\,,\]
where $\sum^{'}$ means that the summation is performed only
over strictly increasing multi-indices. In this paper all multi-indices will be supposed to be strictly increasing.

Let $\Omega$ be a $\cC^\infty$ paracompact manifold equipped with a smooth density of integration.
We let $T\Omega$ and $T^*\Omega$ denote the tangent bundle of $\Omega$ and the cotangent bundle of $\Omega$ respectively.
The complexified tangent bundle of $\Omega$ and the complexified cotangent bundle of $\Omega$ will be denoted by $\Complex T\Omega$
and $\Complex T^*\Omega$ respectively. We write $\langle\,\cdot\,,\cdot\,\rangle$ to denote the pointwise duality between $T\Omega$ and $T^*\Omega$.
We extend $\langle\,\cdot\,,\cdot\,\rangle$ bilinearly to $\Complex T\Omega\times\Complex T^*\Omega$.
Let $E$ be a $\cC^\infty$ vector bundle over $\Omega$. We write $E^*$ to denote the dual bundle of $E$. 
The fiber of $E$ at $x\in\Omega$ will be denoted by $E_x$. We write ${\rm End\,}(E)$ to denote the vector bundle over $\Omega$ with fiber over $x\in\Omega$ consisting of the linear maps from $E_x$ to $E_x$.
Let $F$ be another vector bundle over $\Omega$. We write
$E\boxtimes F$ to denote the vector bundle over $\Omega\times\Omega$ with fiber over $(x, y)\in \Omega\times\Omega$
consisting of the linear maps from $E_x$ to $F_y$. 
Let $Y\subset\Omega$ be an open set. From now on, the spaces of
smooth sections of $E$ over $Y$ and distribution sections of $E$ over $Y$ will be denoted by $\cC^\infty(Y, E)$ and $\mathscr D'(Y, E)$ respectively.
Let $\mathscr E'(Y, E)$ be the subspace of $\mathscr D'(Y, E)$ whose elements have compact support in $Y$.
For $m\in\Real$, we let $H^m(Y, E)$ denote the Sobolev space
of order $m$ of sections of $E$ over $Y$. Put
\begin{gather*}
H^m_{\rm loc\,}(Y, E)=\big\{u\in\mathscr D'(Y, E);\, \varphi u\in H^m(Y, E),
      \,\varphi\in\cC^\infty_0(Y)\big\}\,,\\
      H^m_{\rm comp\,}(Y, E)=H^m_{\rm loc}(Y, E)\cap\mathscr E'(Y, E)\,.
\end{gather*}

\subsection{Metric data}
Let $(M,\Theta)$ be a complex manifold of dimension $n$, where $\Theta$ is a smooth positive $(1,1)$ form, which
induces a Hermitian metric $\langle\,\cdot\,,\cdot\,\rangle$ on the holomorphic tangent bundle $ T^{(1,0)}M$. 
In local holomorphic coordinates $z=(z_1,\ldots,z_n)$, if
$\Theta=\sqrt{-1}\sum^n_{j,k=1}\Theta_{j,k}dz_j\wedge d\ol z_k$, then $\langle\,\frac{\pr}{\pr z_j}\,,\frac{\pr}{\pr z_k}\,\rangle=\Theta_{j,k}, j, k=1,\ldots,n$.  Let $T^{(0,1)}M$ be the
anti-holomorphic tangent bundle of $M$. We extend the Hermitian metric $\langle\,\cdot\,,\cdot\,\rangle$ to $\Complex TM$ in a natural way by requiring $T^{(1,0)}M$
to be orthogonal to $T^{(0,1)}M$ and $\ol{\langle\,u\,,v\,\rangle}=\langle\,\ol u\,,\ol v\,\rangle$, $u, v\in T^{(0,1)}M$. Let $T^{*(1,0)}M$ be the holomorphic cotangent bundle of $M$ and let $T^{*(0,1)}M$ be the anti-holomorphic cotangent bundle of $M$.  
For $p, q\in\mathbb N_0$, let
$\Lambda^{p,q}T^*M=\Lambda^pT^{*(1,0)}M\otimes\Lambda^qT^{*(0,1)}M$ be the bundle of $(p,q)$ forms of $M$. 
We write $\Lambda^{0,q}T^*M=\Lambda^qT^{*(0,1)}M$. The Hermitian metric $\langle\,\cdot\,,\cdot\,\rangle$ on $\Complex TM$
induces a Hermitian metric on
$\Lambda^{p,q}T^*M$ also denoted by $\langle\,\cdot\,,\cdot\,\rangle$. Let
$D\subset M$ be an open set. Let $\Omega^{p,q}(D)$ denote the space
of smooth sections of $\Lambda^{p,q}T^*M$ over $D$. Similarly, if
$E$ is a vector bundle over $D$, then we let $\Omega^{p,q}(D, E)$
denote the space of smooth sections of $(\Lambda^{p,q}T^*M)\otimes
E$ over $D$. Let $\Omega^{p,q}_0(D, E)$ be the subspace of
$\Omega^{p,q}(D, E)$ whose elements have compact support in $D$.

If $w\in\Lambda^rT^{*(0,1)}_zM$, $r\in\mathbb N$, let
$(w\wedge)^*: \Lambda^{q+r}T^{*(0,1)}_zM\To \Lambda^qT^{*(0,1)}_zM,\ q\geq0$,
be the adjoint of left exterior multiplication
$w\wedge: \Lambda^qT^{*(0,1)}_zM\To \Lambda^{q+r}T^{*(0,1)}_zM$.
That is,
\begin{equation} \label{s1-e0}
\langle\,w\wedge u\,,v\,\rangle=\langle\,u\,,(w\wedge)^*v\,\rangle,
\end{equation}
for all $u\in\Lambda^qT^{*(0,1)}_zM$, $v\in\Lambda^{q+r}T^{*(0,1)}_zM$.
Notice that $(w\wedge)^*$ depends anti-linearly on $w$.

Let $(L,h^L)$ be a Hermitian holomorphic line bundle over $M$, where
the Hermitian metric on $L$ is denoted by $h^L$. Until further notice, we assume that $h^L$ is smooth. Given a local holomorphic frame $s$ of $L$ on an open subset $D\subset M$ we define the associated local weight of $h^L$ by
\begin{equation} \label{s1-e1}
\abs{s(x)}^2=\abs{s(x)}^2_{h^L}=e^{-2\phi(x)},\quad\phi\in\cC^\infty(D, \Real).
\end{equation}
Let $R^L=(\nabla^L)^2$ be the Chern curvature of $L$, where $\nabla^L$ is the Hermitian holomorphic connection. Then $R^L|_D=2\pr\ddbar\phi$.
Let $L^k$, $k>0$, be the $k$-th tensor power of the line bundle $L$.
The Hermitian fiber metric on $L$ induces a Hermitian fiber metric
on $L^k$ that we shall denote by $h^k$. If $s$ is a local
trivializing section of $L$ then $s^k$ is a local trivializing
section of $L^k$. For $p,q\in\mathbb N_0$, the Hermitian metric $\langle\,\cdot\,,\cdot\,\rangle$ on $\Lambda^{p,q}T^*M$ and $h^k$ induce a Hermitian metric on $\Lambda^{p,q}T^*M\otimes L^k$, also denoted by $\langle\,\cdot\,,\cdot\,\rangle$. For $f\in\Omega^{p,q}(M, L^k)$, we denote the
pointwise norm $\abs{f(x)}^2:=\abs{f(x)}^2_{h^k}=\langle f(x),f(x)\rangle$. 
We take $dv_M=dv_M(x)$
as the induced volume form on $M$. The $L^2$--Hermitian inner products on the spaces
$\Omega^{p,q}_0(M,L^k)$ and $\Omega^{p,q}_0(M)$ are given by
\begin{equation}\label{toe2.2}
\begin{split}
&(s_1,s_2)_k=\int_M\langle s_1(x),s_2(x)\rangle\,dv_M(x)\,,\ \ s_1,s_2\in\Omega^{p,q}_0(M,L^k),\\
&(f_1,f_2)=\int_M\langle f_1(x),f_2(x)\rangle\,dv_M(x)\,,\ \ f_1,f_2\in\Omega^{p,q}_0(M).
\end{split}
\end{equation}
We write $\norm{f}^2:=\norm{f}^2_{h^k}=(f,f)_k$, $f\in\Omega^{p,q}_0(M,L^k)$. For $g\in\Omega^{p,q}_0(M)$, we also write $\norm{g}^2:=(g,g)$. Let
$L^2_{(p,q)}(M,L^k)$ be the completion of
$\Omega^{p,q}_0(M,L^k)$ with respect to $\norm{\cdot}$.

\subsection{A self-adjoint extension of the Kodaira Laplacian}\label{gaffney}
We denote by
\begin{equation}\label{e:ddbar}
\ddbar_k:\Omega^{0,r}(M,L^k)\To\Omega^{0,r+1}(M,L^k)\,,\:\:\ol{\pr}^{\,*}_k:\Omega^{0,r+1}(M,L^k)\To\Omega^{0,r}(M,L^k)
\end{equation} 
the Cauchy-Riemann operator acting on sections of $L^k$ and its formal adjoint with respect to $(\cdot\,,\cdot)_k$ respectively.
Let 
\begin{equation}\label{e:Kod_Lap}
\Box_{k}^{(q)}:=\ddbar_k\ol{\pr}^{\,*}_k+\ol{\pr}^{\,*}_k\ddbar_k:
\Omega^{0,q}(M, L^k)\To\Omega^{0,q}(M, L^k)
\end{equation}   
be the Kodaira Laplacian acting on $(0,q)$--forms with values in $L^k$. We extend 
$\ddbar_k$ to $L^2_{(0,r)}(M,L^k)$ by 
\begin{equation}\label{e:ddbar1}
\ddbar_k:{\rm Dom\,}\ddbar_k\subset L^2_{(0,r)}(M, L^k)\To L^2_{(0,r+1)}(M, L^k)\,,
\end{equation}
where ${\rm Dom\,}\ddbar_k:=\{u\in L^2_{(0,r)}(M, L^k);\, \ddbar_ku\in L^2_{(0,r+1)}(M, L^k)\}$, where 
$\ddbar_k u$ is defined in the sense of distributions. 
We also write 
\begin{equation}\label{e:ddbar_star1}
\ol{\pr}^{\,*}_k:{\rm Dom\,}\ol{\pr}^{\,*}_k\subset L^2_{(0,r+1)}(M, L^k)\To L^2_{(0,r)}(M, L^k)
\end{equation}
to denote the Hilbert space adjoint of $\ddbar_k$ in the $L^2$ space with respect to $(\,\cdot\,,\cdot\,)_k$.
Let $\Box^{(q)}_k$ denote the Gaffney extension of the Kodaira Laplacian given by 
\begin{equation}\label{Gaf1}
{\rm Dom\,}\Box^{(q)}_k=\set{s\in L^2_{(0,q)}(M,L^k);\, s\in{\rm Dom\,}\ddbar_k\cap{\rm Dom\,}\ol{\pr}^{\,*}_k,\ \ddbar_ku\in{\rm Dom\,}\ol{\pr}^{\,*}_k,\ \ol{\pr}^{\,*}_ku\in{\rm Dom\,}\ddbar_k}\,,
\end{equation}
and $\Box^{(q)}_ks=\ddbar_k\ol{\pr}^{\,*}_ks+\ol{\pr}^{\,*}_k\ddbar_ks$ for $s\in {\rm Dom\,}\Box^{(q)}_k$. By a result of Gaffney \cite[Prop.\,3.1.2]{MM07}, $\Box^{(q)}_k$ is a positive self-adjoint operator. Note that if $M$ is complete, the Kodaira Laplacian $\Box^{(q)}_k$ is essentially self-adjoint \cite[Cor.\,3.3.4]{MM07} and the Gaffney extension coincides with the Friedrichs extension of $\Box^{(q)}_k$.

\subsection{Schwartz kernel theorem}

We recall the Schwartz kernel theorem \cite[Th.\,5.2.1, 5.2.6]{Hor03}, \cite[p.\,296]{Tay1:96}. 
Let $\Omega$ be a $\cC^\infty$ paracompact manifold equipped with a smooth density of integration.
Let $E$ and $F$ be smooth vector bundles over $\Omega$. Then any continuous linear operator
$A:\cC^\infty_0(\Omega, E)\To \mathscr D'(\Omega, F)$ has a Schwartz distribution kernel, denoted
$K_A(x, y)$ or $A(x, y)$.
Moreover, the following two statements are equivalent
\begin{enumerate}
\item $A$ is continuous: $\mathscr E'(\Omega, E)\To\cC^\infty(\Omega, F)$,
\item $K_A(x,y)\in\cC^\infty(\Omega\times\Omega, E_y\boxtimes F_x)$.
\end{enumerate}
If $A$ satisfies (I) or (II), we say that $A$ is a \emph{smoothing operator}. Furthermore, $A$ is smoothing if and only if
$A: H^s_{\rm comp\,}(\Omega, E)\To H^{s+N}_{\rm loc\,}(\Omega, F)$
is continuous, for all $N\geq0$, $s\in\Real$.
We say that $A$ is properly supported if ${\rm Supp\,}K_A\subset\Omega\times\Omega$ is proper. That is, the two projections: $t_x:(x,y)\in{\rm Supp\,}K_A\To x\in\Omega$, $t_y:(x,y)\in{\rm Supp\,}K_A\To y\in\Omega$ are proper (i.e. the inverse images of $t_x$ and $t_y$ of all compact subsets of $\Omega$ are compact).
We say that $A$ is smoothing away the diagonal
if $\chi_1A\chi_2$ is smoothing, for all $\chi_1, \chi_2\in\cC^\infty_0(\Omega)$ with ${\rm Supp\,}\chi_1\bigcap{\rm Supp\,}\chi_2=\emptyset$.

Let $H(x,y)\in\mathscr D'(\Omega\times\Omega,E_y\boxtimes F_x)$. We write $H$ to denote the unique continuous operator $H:\cC^\infty_0(\Omega,E)\To\mathscr D'(\Omega,F)$ with distribution kernel $H(x,y)$. In this work, we identify $H$ with $H(x,y)$.  Let
$A, B: \cC^\infty_0(\Omega, E)\to \mathscr D'(\Omega, F)$ be continuous operators.
We write $A\equiv B$ or $A(x,y)\equiv B(x,y)$ if $A-B$ is a smoothing operator.

\section{Szeg\"{o} kernels and semi-classical Hodge decomposition}\label{s:sem_hodge}

The goal of this Section is to prove the semiclassical Hodge decomposition for the Kodaira Laplacian, i.e.\ to find a semi-classical partial inverse and an approximate kernel for $\Box^{(q)}_k$, cf.\ Theorem \ref{s3-t4}. For this purpose we reduce the analysis of the Kodaira Laplacian to the analysis of the Kohn Laplacian on the Grauert tube of the line bundle $L$. In Section \ref{s:CR} we recall the construction of these two objects. Section \ref{s:apszk} contains a detailed study of the microlocal Hodge decomposition of the Kohn Laplacian on a non-degenerate CR manifold and especially on the Grauert tube, by following \cite{Hsiao08}. Finally, in Section \ref{s:smcl-hodge} we apply this results in order to obtain the semi-classical Hodge decomposition for the Kodaira Laplacian.

\subsection{The Grauert tube}\label{s:CR}
Let $(M,\Theta)$ be a Hermitian manifold and $(L,h^L)$ be a holomorphic Hermitian line bundle on $M$. 
Let $(L^*,h^{L^*})$ be the dual bundle of $L$. We denote
\begin{equation}\label{grauert_tube}
G:=\set{v\in L^*;\, \abs{v}_{h^{L^*}}<1}\,,\quad X:=\partial G=\set{v\in L^*;\, \abs{v}_{h^{L^*}}=1}\,.
\end{equation}
The domain $G$ is called \emph{Grauert tube} associated to $L$. 
We denote 
\[T^{(1,0)}X:=T^{(1,0)}L^*\cap\Complex TX\,,\quad T^{(0,1)}X:=T^{(0,1)}L^*\cap\Complex TX\,.\]
Then $(X, T^{(1,0)}X)$ is a CR manifold of dimension $2n+1$ and the bundle
$T^{(1,0)}X$ is called the holomorphic tangent bundle of $X$. The manifold $X$ is equipped with a natural  $S^1$ action.
Locally $X$ can be represented in local holomorphic coordinates $(z,\lambda)$, where $\lambda$ is the fiber coordinate, as the set of all $(z,\lambda)$ such that $\abs{\lambda}^2e^{2\phi(z)}=1$,
where $\phi$ is a local weight of $h^L$. The $S^1$ action on $X$ is given by $e^{i\theta}\circ (z,\lambda)=(z,e^{i\theta}\lambda)$, $e^{i\theta}\in S^1$, $(z,\lambda)\in X$. Let $Y$ be the global real vector field on $X$ determined by \[Yu(x)=\frac{\pr}{\pr\theta}u(e^{i\theta}\circ x)\big|_{\theta=0}\,
\quad\text{for all $u\in\cC^\infty(X)$}\,.
\]
Let $\pi:X\To M$ be the natural projection. We have the bijective map:
\begin{align*}
\pi^*:T^{(1,0)}X\oplus T^{(0,1)}X\To T^{(1,0)}M\oplus T^{(0,1)}M\,,\quad
W\To\pi^*W,
\end{align*}
where $(\pi^*W)f=W(f\circ\pi)$, for all $f\in\cC^\infty(M)$. We take the Hermitian metric $\langle\,\cdot\,,\cdot\,\rangle$ on $\Complex TX$ so that
$Y\bot\bigr(T^{(1,0)}X\oplus T^{(0,1)}X\bigr)$, $\langle\,Y\,,Y\,\rangle=1$ and $\langle\,Z\,,W\,\rangle=\langle\,\pi^*Z\,,\pi^*W\,\rangle$, $Z, W\in T^{(1,0)}X\oplus T^{(0,1)}X$. The Hermitian metric $\langle\,\cdot\,,\cdot\,\rangle$ on $\Complex TX$ induces, by duality, a Hermitian
metric on the complexified cotangent bundle $\Complex T^*X$ that we shall also denote by $\langle\,\cdot\,,\cdot\,\rangle$. Define $T^{*(1,0)}X:=\bigr(T^{(0,1)}X\oplus\Complex Y\bigr)^\bot\subset\Complex T^*X$,
$ T^{*(0,1)}X:=\bigr(T^{(1,0)}X\oplus\Complex Y\bigr)^\bot\subset\Complex T^*X$. For $q\in\mathbb N$, the bundle
of $(0,q)$ forms of $X$ is given by $\Lambda^{q}T^{*(0,1)}X:=\Lambda^q\bigr( T^{*(0,1)}X\bigr)$. The Hermitian metric $\langle\,\cdot\,,\cdot\,\rangle$
on $\Complex T^*X$ induces a Hermitian metric on $\Lambda^qT^{*(0,1)}X$ also denoted by $\langle\,\cdot\,,\cdot\,\rangle$.

Locally there is a real one form $\omega_0$ of length one which is pointwise orthogonal to $ T^{*(0,1)}X\oplus T^{*(1,0)}X$.
$\omega_0$ is unique up to the choice of sign. We take $\omega_0$ so that $\langle\,\omega_0\,,Y\,\rangle=-1$. Therefore $\omega_0$, so chosen,
is globally defined. 

The \emph{Levi form} $\mathcal{L}_p$ of $X$ at $p\in X$ is the Hermitian quadratic form on $T^{(1,0)}_pX$ defined as follows: 
\begin{equation}\label{levi_def}
\mathcal{L}_p(U,\ol V)=\frac{1}{2i}\big\langle\,[\mathcal{U}\ ,\ol{\mathcal{V}}\,](p)\,,\omega_0(p)\,\big\rangle\,,\quad U, V\in T^{(1,0)}_pX
\end{equation}
where $\mathcal{U}, \mathcal{V}\in
\cC^\infty(X, T^{(1,0)}X)$ that satisfy
$\mathcal{U}(p)=U$, $\mathcal{V}(p)=V$ and 
$[\mathcal{U}\ ,\ol{\mathcal{V}}\,]=\mathcal{U}\ol{\mathcal{V}}-\ol{\mathcal{V}}\mathcal{U}$ denotes the commutator of $\mathcal{U}$ and $\ol{\mathcal{V}}$.

Let $B\subset X$ be an open set. Let $\Omega^{0,q}(B)$ denote the space of smooth sections of $\Lambda^qT^{*(0,1)}X$ over $B$. Let $\Omega^{0,q}_0(B)$ be the subspace of
$\Omega^{0,q}(B)$ whose elements have compact support in $B$. Let $\ddbar_b:\Omega^{0,q}(X)\To\Omega^{0,q+1}(X)$ be the tangential
Cauchy-Riemann operator. We take $dv_X=dv_X(x)$ as the induced volume form on $X$. Then, we get 
natural inner product $(\,\cdot\,,\cdot)$
on $\Omega^{0,q}(X)$. Let $\ol{\pr}^{\,*}_b:\Omega^{0,q+1}(X)\To\Omega^{0,q}(X)$ be the formal adjoint of $\ddbar_b$ with respect to
$(\,\cdot\,,\cdot)$. The \emph{Kohn Laplacian} on $(0,q)$ forms is given by
\[
\Box^{(q)}_b:=\ddbar_b\ol{\pr}^{\,*}_b+\ol{\pr}^{\,*}_b\ddbar_b:\Omega^{0,q}(X)\To\Omega^{0,q}(X)\,.
\]

We introduce now a local holomorphic frame and local coordinates in terms of which we shall write down the operators explicitly.
Let 

(i) $s$ be a local trivializing section of $L$ on an open set $D\Subset M$, 

(ii) $\phi\in\cC^\infty(D)$ be the local weight of the metric $h^L$ defined by $\abs{s}_{h^L}^2=e^{-2\phi}$. 

\noindent
Then $s^*:=s^{-1}$ is a local trivializing
section of $L^*$ on $D$. We have $\abs{s^*}_{h^{L^*}}^2=e^{2\phi}$. 

We introduce holomorphic and real coordinates on $D$ by
\begin{equation} \label{s3-e0}
z=(z_1,\ldots,z_n)\,,\: x'=(x_1,\ldots,x_{2n})\,,\quad z_j=x_{2j-1}+ix_{2j}\,,\: j=1,\ldots,n\,.
\end{equation}
 We identify $D$ with an open set of $\Complex^{n}$. We have the local diffeomorphism:
\begin{equation} \label{s3-e1}
\tau:D\times]-\varepsilon_0, \varepsilon_0[\To X\,,\:
(z,\theta)\mapsto e^{-\phi(z)}s^*(z)e^{-i\theta}\,, \ 0<\varepsilon_0\leq\pi.
\end{equation}
It is convenient to work with the local coordinates $(z,\theta)$. In terms of these coordinates, it is straightforward to see that $Y=-\frac{\pr}{\pr\theta}$. Moreover,
\[T^{(1,0)}_vX=\C\Big\{\frac{\pr}{\pr z_j}-i\frac{\pr\phi}{\pr z_j}(z)\frac{\pr}{\pr\theta};\, j=1,\ldots,n\Big\}\,,\:v=e^{-\phi(z)}s^*(z)e^{-i\theta}\in X\,.\]
Further, let $\{\ol Z_j\}_{j=1}^n$ be an orthonormal basis for the holomorphic tangent bundle $T^{(1,0)}M$ and let $\{\ol e_j\}_{j=1}^n$ be the dual basis of $T^{*(1,0)}M$.
Then, $\{\ol Z_j-i\ol Z_j(\phi)\frac{\pr}{\pr\theta}\}_{j=1}^n$ is an orthonormal basis for
$T^{(1,0)}X$ and $\{\ol e_j\}_{j=1}^n$ is the dual orthonormal basis for $T^{*(1,0)}X$. Furthermore, we can check that
\begin{equation} \label{s3-e1-1}
\omega_0=d\theta+\sum^n_{j=1}(-iZ_j(\phi)e_j+i\ol Z_j(\phi)\ol e_j).
\end{equation}
From this, we can compute for $j,k=1,\ldots,n$, $p\in X$:
\[
\mathcal{L}_p\Big(\frac{\pr}{\pr z_j}-i\frac{\pr\phi}{\pr z_j}(\pi(p))\frac{\pr}{\pr\theta}, \frac{\pr}{\pr \ol z_k}+i\frac{\pr\phi}{\pr\ol z_k}(\pi(p))\frac{\pr}{\pr\theta}\Big)=\frac{\pr^2\phi}{\pr z_j\pr\ol z_k}(\pi(p))\,.
\]  
Thus, for a given point $p\in X$, we have
\begin{equation} \label{s3-e2}
\begin{split}
\mathcal{L}_p(U, \ol V)&=\langle\,\pr\ddbar\phi(\pi(p))\,, \pi^*U\wedge\pi^*\ol
V\,\rangle
=\langle\,\tfrac{1}{2}R^L(\pi(p))\,, \pi^*U\wedge\pi^*\ol V\,\rangle\\&=\langle\,\tfrac{1}{2}\dot{R}^L(\pi(p))\pi^*U\,,\pi^*V\,\rangle,
\ \ \forall\ U, V\in T^{(1,0)}_pX.
\end{split}
\end{equation}
We deduce the following: 

\begin{prop}
Let $(L,h^L)$ be a Hermitian holomorphic line bundle over a complex manifold $M$ and let $X\subset L^*$
be the boundary of the Grauert tube associated to $L$. Let $p\in X$. If the curvature $R^L$ has signature $(n_-,n_+)$ at $\pi(p)$, then $\mathcal{L}_p$ has signature $(n_-,n_+)$.
\end{prop}
We define the operators $\ddbar_s$, $\ddbar^*_s$, $\Box^{(q)}_s$, which are the local versions of the operators $\ddbar_k$, $\ddbar^*_k$, $\Box^{(q)}_k$ (see \eqref{e:ddbar}-\eqref{Gaf1}), by the following equations:
\begin{equation} \label{s1-e4}
\begin{split}
&\ddbar_s=\ddbar+k(\ddbar\phi)\wedge:\Omega^{0,q}(D)\To\Omega^{0,q+1}(D),\\
&\ol{\pr}^{\,*}_s=\ol{\pr}^{\,*}+k\bigr((\ddbar\phi)\wedge\bigr)^*:\Omega^{0,q+1}(D)\To\Omega^{0,q}(D), \\
&\Box^{(q)}_s=\ddbar_s\ol\pr^*_s+\ol\pr^*_s\ddbar_s:\Omega^{0,q}(D)\To\Omega^{0,q}(D).
\end{split}
\end{equation}
Here $\ol{\pr}^{\,*}:\Omega^{0,q+1}(D)\To\Omega^{0,q}(D)$ is the formal adjoint of $\ddbar$ with respect to $(\,\cdot\,,\cdot)$.
We have the unitary identifications:
\begin{equation} \label{s1-e3}
\begin{split}
\Omega^{0,q}(D, L^k)&\longleftrightarrow\Omega^{0,q}(D) \\
f=s^kg&\longleftrightarrow\widehat f(z)=e^{-k\phi}s^{-k}f=g(z)e^{-k\phi(z)},\ \ g\in\Omega^{0,q}(D),\\
\ddbar_k&\longleftrightarrow\ddbar_s,\ \ \ddbar_kf=s^ke^{k\phi}\ddbar_s\widehat f,\\
\ol{\pr}^{\,*}_k&\longleftrightarrow\ddbar^*_s,\ \ \ol{\pr}^{\,*}_kf=s^ke^{k\phi}\ol{\pr}^{\,*}_s\widehat f,\\
\Box^{(q)}_k&\longleftrightarrow\Box^{(q)}_s,\ \
\Box^{(q)}_kf=s^ke^{k\phi}\Box^{(q)}_s\widehat f\,.
\end{split}
\end{equation}

We continue to work with the local coordinates $(z, \theta)$. As above, let $(Z_j)_{j=1}^{n}$ be an orthonormal basis
for $T^{(0,1)}M$ and let $(e_j)_{j=1}^{n}$ be an orthonormal basis for $ T^{*(0,1)}M$ which is dual to
$(Z_j)_{j=1}^{n}$. We can check that
\begin{equation} \label{s3-e3}
\ddbar_b=\sum^n_{j=1}(e_j\wedge)\circ\bigl(Z_j+iZ_j(\phi)\frac{\pr}{\pr\theta}\bigr)+\sum^n_{j=1}\bigr((\ddbar e_j)\wedge\bigr)\circ(e_j\wedge)^*
\end{equation}
and correspondingly
\begin{equation} \label{s3-e4}
\ol{\pr}^{\,*}_b=\sum^n_{j=1}\bigl((e_j\wedge)^*\bigr)\circ\bigl(Z^*_j+i\ol Z_j(\phi)\frac{\pr}{\pr\theta}\bigr)+\sum^n_{j=1}(e_j\wedge)\circ\bigr((\ddbar e_j)\wedge\bigr)^*,
\end{equation}
where $Z^*_j$ is the formal adjoint of $Z_j$ with respect to $(\,\cdot\,,\cdot)$, $j=1,\ldots,n$.

Let $\ddbar_s$ and $\ol{\pr}^{\,*}_s$ be as in \eqref{s1-e3} and \eqref{s1-e4}.
We can check that
\begin{equation} \label{s3-e6}
\begin{split}
&\ddbar_s=\sum^n_{j=1}(e_j\wedge)\circ(Z_j+kZ_j(\phi))+\sum^n_{j=1}\bigr((\ddbar e_j)\wedge\bigr)\circ(e_j\wedge)^*,\\
&\ol{\pr}^{\,*}_s=\sum^n_{j=1}\bigr((e_j\wedge)^*\bigr)\circ(Z^*_j+k\ol Z_j(\phi))+\sum^n_{j=1}(e_j\wedge)\circ\bigr((\ddbar e_j)\wedge\bigr)^*.
\end{split}
\end{equation}

From now on, we identify $\Lambda^qT^{*(0,1)}M$ with $\Lambda^qT^{*(0,1)}X$.
From \eqref{s1-e3}, \eqref{s1-e4}, explicit formulas of $\ddbar_s$, $\ol\pr^*_s$ and \eqref{s3-e3},
\eqref{s3-e4}, we get
\begin{equation} \label{s3-e8}
\Box^{(q)}_kf=s^ke^{k\phi}e^{ik\theta}\Box^{(q)}_b(\widehat fe^{-ik\theta}),
\end{equation}
for all $f\in\Omega^{0,q}(D, L^k)$, where $\widehat f$ is given by \eqref{s1-e3}.

Let $u(z, \theta)\in\Omega^{0,q}_0(D\times(-\varepsilon_0, \varepsilon_0))$.
Note that
\[k\int e^{i\theta k}u(z,\theta)d\theta=\int(-i)\frac{\pr}{\pr\theta}(e^{i\theta k})u(z, \theta)d\theta=\int e^{i\theta k}i\frac{\pr u}{\pr\theta}(z, \theta)d\theta.\]
From this observation and explicit formulas of $\ddbar_b$,
$\ol\pr^*_b$, $\ddbar_s$ and $\ol\pr^*_s$ (see \eqref{s3-e3},
\eqref{s3-e4} and \eqref{s3-e6}), we conclude that
\begin{equation} \label{s3-e9}
\Box^{(q)}_s\bigr(\int e^{i\theta k}u(z,\theta)d\theta\bigr)=\int e^{i\theta k}(\Box^{(q)}_bu)(z, \theta)d\theta,
\end{equation}
for all $u(z, \theta)\in\Omega^{0,q}_0(D\times(-\varepsilon_0, \varepsilon_0))$.

\subsection{Approximate Szeg\"{o} kernels}\label{s:apszk}

In this Section we review the results in~\cite{Hsiao08} about the existence of a microlocal Hodge decomposition of the Kohn Laplacian on an open set of a CR manifold where the Levi form is non-degenerate.
The approximate Szeg\"{o} kernel is a Fourier integral operator with complex phase in the sense of Melin-Sj\"ostrand \cite{MS74}.
We then specialize to the case of the Grauert tube of a line bundle and give a useful formula for the phase function of the approximate Szeg\"{o} kernel in Theorem \ref{s3-t1-bis}.

Theorems \ref{s3-t0}-\ref{s3-t1-1} are proved in chapter 6, chapter 7 and chapter 8 of part I in \cite{Hsiao08}. In
\cite{Hsiao08} the existence of the microlocal Hodge decomposition is stated for compact CR manifolds, but the construction and arguments used are essentially local.

\begin{thm} \label{s3-t0}
Let $X$ be an orientable CR manifold whose Levi form $\mathcal{L}$ is non-degenerate of constant signature $(n_-,n_+)$ at each point of an open set $B\Subset X$.
Let $q\neq n_-, n_+$. There exists a properly supported continuous operator
\begin{equation} \label{s3-e6-bis00}
A:
\begin{cases}
&H^s_{\rm loc\,}(B,\Lambda^qT^{*(0,1)}X)\To H^{s+1}_{\rm loc\,}(B,\Lambda^qT^{*(0,1)}X),\\
&H^s_{\rm comp\,}(B,\Lambda^qT^{*(0,1)}X)\To H^{s+1}_{\rm comp\,}(B,\Lambda^qT^{*(0,1)}X)\,
\end{cases}
\end{equation}
for all $s\geq0$,
such that $A$ is smoothing away the diagonal and $\Box^{(q)}_bA\equiv I$.
\end{thm}

For $m\in\Real$ let $S^m_{1, 0}$ be the H\"{o}rmander symbol space (see Grigis-Sj\"{o}strand~\cite[Def.\,1.1]{GS94}).
Let $p_0(x, \xi)\in\cC^\infty(T^*X)$ be the principal symbol of $\Box^{(q)}_b$. Note that $p_0(x,\xi)$ is a polynomial of degree $2$ in $\xi$. The characteristic manifold of $\Box^{(q)}_b$ is given by $\Sigma=\Sigma^+\bigcup\Sigma^-$, where
\begin{equation*}
\begin{split}
&\Sigma^+=\set{(x, \lambda\omega_0(x))\in T^*X;\, \lambda>0},\\
&\Sigma^-=\set{(x, \lambda\omega_0(x))\in T^*X;\, \lambda<0}.
\end{split}
\end{equation*}

\begin{thm} \label{s3-t1}
Let $X$, $B$ and $(n_-,n_+)$ be as in Theorem \ref{s3-t0}.
Let $q=n_-$ or $n_+$. Then there exist properly supported continuous operators
\begin{equation} \label{s3-e6-bis0}
\begin{split}
A:
\begin{cases}H^s_{\rm loc\,}(B,\Lambda^qT^{*(0,1)}X)\To H^{s+1}_{\rm loc\,}(B,\Lambda^qT^{*(0,1)}X),\\
H^s_{\rm comp\,}(B,\Lambda^qT^{*(0,1)}X)\To H^{s+1}_{\rm comp\,}(B,\Lambda^qT^{*(0,1)}X)\,,
\end{cases}\\
S_-,S_+:
\begin{cases}
H^s_{\rm loc\,}(B,\Lambda^qT^{*(0,1)}X)\To H^{s}_{\rm loc\,}(B,\Lambda^qT^{*(0,1)}X),\\
H^s_{\rm comp\,}(B,\Lambda^qT^{*(0,1)}X)\To H^{s}_{\rm comp\,}(B,\Lambda^qT^{*(0,1)}X)\,,
\end{cases}
\end{split}
\end{equation}
for all $s\geq0$,
such that $A, S_-, S_+$ are smoothing away the diagonal and
\begin{gather}
\Box^{(q)}_bA+S_-+S_+= I,  \,\quad
\Box^{(q)}_bS_-\equiv0,\ \Box^{(q)}_bS_+\equiv0,\nonumber \\
A\equiv A^*\,,\quad S_-\equiv S_-^*\equiv S_-^2 \,,\quad
S_+\equiv S_+^*\equiv S_+^2\,,\quad 
S_-S_+\equiv S_+S_-\equiv0,\nonumber
\end{gather}
where $A^*$, $S_-^*$ and $S_+^*$ are the formal adjoints of $A$, $S_-$ and $S_+$ with respect to $(\,\cdot\,,\cdot)$ respectively and $K_{S_-}(x,y)$ satisfies
\[K_{S_-}(x, y)\equiv\int^{\infty}_{0}\!\! e^{i\varphi_-(x, y)t}s_-(x, y, t)dt\]
with a symbol
\begin{equation}  \label{s3-e11-bis}\begin{split}
&s_-(x, y, t)\in S^{n}_{1, 0}\big(B\times B\times]0, \infty[, \Lambda^qT^{*(0,1)}_yX\boxtimes \Lambda^qT^{*(0,1)}_xX\big), \\
&s_-(x, y, t)\sim\sum^\infty_{j=0}s^j_-(x, y)t^{n-j}\text{ in }S^{n}_{1, 0}
\big(B\times B\times]0, \infty[, \Lambda^qT^{*(0,1)}_yX\boxtimes\Lambda^qT^{*(0,1)}_xX\big)\,,\\
&s^j_-(x, y)\in\cC^\infty\big(B\times B, \Lambda^qT^{*(0,1)}_yX\boxtimes\Lambda^qT^{*(0,1)}_xX\big),\ \ j\in\N_0,
\end{split}\end{equation}
and phase function
\begin{gather}
\varphi_-\in\cC^\infty(B\times B),\ \ {\rm Im\,}\varphi_-(x, y)\geq0 \label{s3-e12-bis} \,,\
\varphi_-(x, x)=0,\ \ \varphi_-(x, y)\neq0\ \ \mbox{if}\ \ x\neq y,\\
d_x\varphi_-\neq0,\ \ d_y\varphi_-\neq0\ \ \mbox{where}\ \ {\rm Im\,}\varphi_-=0,\\ 
d_x\varphi_-(x, y)|_{x=y}=-\omega_0(x), \ \ d_y\varphi_-(x, y)|_{x=y}=\omega_0(x),\label{s3-e15-bis} \\
\varphi_-(x, y)=-\ol\varphi_-(y, x). \label{s3-e16-bis}
\end{gather}
Moreover, there is a function $f\in\cC^\infty(B\times B)$ such that
\begin{equation} \label{s3-e16-bisbis}
p_0(x, (\varphi_-)'_x(x,y))-f(x,y)\varphi_-(x,y)
\end{equation}
vanishes to infinite order at $x=y$.

Similarly,
\[K_{S_+}(x, y)\equiv\int^{\infty}_{0}\!\! e^{i\varphi_+(x, y)t}s_+(x, y, t)dt\]
with $s_+(x, y, t)\in S^{n}_{1, 0}\big(B\times B\times]0, \infty[, \Lambda^qT^{*(0,1)}_yX\boxtimes\Lambda^qT^{*(0,1)}_xX\big)$, 
\[
s_+(x, y, t)\sim\sum^\infty_{j=0}s^j_+(x, y)t^{n-j}
\]
in $S^{n}_{1, 0}\big(B\times B\times]0, \infty[,
\Lambda^qT^{*(0,1)}_yX\boxtimes\Lambda^qT^{*(0,1)}_xX\big)$,
where
\[s^j_+(x, y)\in\cC^\infty\big(B\times B, \Lambda^qT^{*(0,1)}_yX\boxtimes\Lambda^qT^{*(0,1)}_xX\big),\ \ j\in\N_0,\]
and $-\ol\varphi_+(x, y)$ satisfies \eqref{s3-e12-bis}--\eqref{s3-e16-bisbis}. Moreover, if $q\neq n_+$, then $s_+(x,y,t)$ vanishes to infinite order at $x=y$. If $q\neq n_-$, then $s_-(x,y,t)$ vanishes to infinite order at $x=y$.
\end{thm}
The operators $S_{+}$, $S_{-}$ are called \emph{approximate Szeg\"o kernels}.
\begin{proof}
We only sketch the proof. For all the details, we refer the reader to Part I in \cite{Hsiao08}.
We will use the heat equation method. We work with some
real local coordinates $x=(x_1,\ldots,x_{2n+1})$
defined on $B$. 
We will say that
$a\in\cC^\infty(\ol\Real_+\times B\times\Real^{2n+1})$ is quasi-homogeneous of
degree $j$ if $a(t,x,\lambda\eta)=\lambda^ja(\lambda t,x,\eta)$
for all $\lambda>0$. We consider the problem
\begin{equation} \label{e:i-heat}
\left\{ \begin{array}{ll}
(\pr_t+\Box^{(q)}_b)u(t,x)=0  & \text{ in }\Real_+\times B,  \\
u(0,x)=v(x). \end{array}\right.
\end{equation}
We start by a formal construction. We look for an approximate solution of
\eqref{e:i-heat} of the form $u(t,x)=A(t)v(x)$,
\begin{equation} \label{e:i-fourierheat}
A(t)v(x)=\frac{1}{(2\pi)^{2n+1}}\int\!e^{i(\psi(t,x,\eta)-\langle y,\eta\rangle)}a(t,x,\eta)v(y)dyd\eta
\end{equation}
where formally
\[a(t,x,\eta)\sim\sum^\infty_{j=0}a_j(t,x,\eta),\]
with $a_j(t,x,\eta)$ matrix-valued quasi-homogeneous functions of degree $-j$.

The full symbol of $\Box^{(q)}_b$ equals $\sum^2_{j=0}p_j(x,\xi)$,
where $p_j(x,\xi)$ is positively homogeneous of order $2-j$ in the sense that
\[p_j(x, \lambda\eta)=\lambda^{2-j}p_j(x, \eta),\ \abs{\eta}\geq1,\ \lambda\geq1.\]
We apply $\pr_t+\Box^{(q)}_b$ formally
inside the integral in \eqref{e:i-fourierheat} and then introduce the asymptotic expansion of
$\Box^{(q)}_b(ae^{i\psi})$. Set $(\pr_t+\Box^{(q)}_b)(ae^{i\psi})\sim 0$ and regroup
the terms according to the degree of quasi-homogeneity. The phase $\psi(t, x, \eta)$ should solve
\begin{equation} \label{e:i-chara}
\left\{ \begin{array}{ll}
\frac{\displaystyle\pr\psi}{\displaystyle\pr t}-ip_0(x,\psi'_x)=O(\abs{{\rm Im\,}\psi}^N), & \forall N\geq 0,   \\
 \psi|_{t=0}=\langle x, \eta\rangle. \end{array}\right.
\end{equation}
This equation can be solved with ${\rm Im\,}\psi(t, x,\eta)\geq0$ and the phase $\psi(t, x, \eta)$ is quasi-homogeneous of
degree $1$. Moreover,
\begin{gather*}
\psi(t,x,\eta)=\langle x,\eta\rangle{\;\rm on\;}\Sigma,\; d_{x,\eta}(\psi-\langle x,\eta\rangle)=0 \text{ on }\Sigma,\\
{\rm Im\,}\psi(t, x,\eta)\asymp\Big(\abs{\eta}\frac{t\abs{\eta}}{1+t\abs{\eta}}\Big)
\Big({\rm dist\,}\big((x,\frac{\eta}{\abs{\eta}}),\Sigma\big)\Big)^2,\ \ \abs{\eta}\geq 1.
\end{gather*}
Furthermore, there exists $\psi(\infty,x,\eta)\in\cC^\infty(B\times\dot\Real^{2n+1})$ with a
uniquely determined Taylor expansion at each point of $\Sigma$ such that for every compact set
$K\subset B\times\dot\Real^{2n+1}$ there is a constant $c_K>0$ such that
\[{\rm Im\,}\psi(\infty,x,\eta)\geq c_K\abs{\eta}\Big({\rm dist\,}\big((x,\frac{\eta}{\abs{\eta}}),\Sigma\big)\Big)^2,
\ \ \abs{\eta}\geq 1.\]
If $\lambda\in \cC(T^*B\smallsetminus 0)$, $\lambda>0$ is positively homogeneous of degree $1$ and 
$\lambda|_\Sigma<\min\lambda_j$, $\lambda_j>0$, where
$\pm i\lambda_j$ are the non-vanishing eigenvalues of the fundamental matrix of $\Box^{(q)}_b$,
then the solution $\psi(t,x,\eta)$ of \eqref{e:i-chara} can be chosen so that for every
compact set $K\subset B\times\dot\Real^{2n+1}$ and all indices $\alpha$, $\beta$, $\gamma$,
there is a constant $c_{\alpha,\beta,\gamma,K}$ such that
\[ \abs{\pr^\alpha_x\pr^\beta_\eta\pr^\gamma_t(\psi(t,x,\eta)-\psi(\infty,x,\eta))}
\leq c_{\alpha,\beta,\gamma,K}e^{-\lambda(x,\eta)t}\text{ on }\ol\Real_+\times K.\]
We obtain the transport equations
\begin{equation} \label{e:i-heattransport}
\left\{ \begin{array}{l}
 T(t,x,\eta,\pr_t,\pr_x)a_0=O(\abs{{\rm Im\,}\psi}^N), \ \forall N,   \\
 T(t,x,\eta,\pr_t,\pr_x)a_j+l_j(t,x,\eta,a_0,\ldots,a_{j-1})= O(\abs{{\rm Im\,}\psi}^N), \ \forall N,\ \ j=1,2,\ldots.
 \end{array}\right.
\end{equation}

Following the method of Menikoff-Sj\"{o}strand~\cite{MS78}, we see that we can solve \eqref{e:i-heattransport}. Moreover,
$a_j$ decay exponentially fast in $t$ when $q\neq n_-$, $n_+$, and has subexponential growth in 
general. We assume that $q=n_-$ or $n_+$.
We use $\ddbar_b\Box^{(q)}_b=\Box^{(q+1)}_b\ddbar_b$, $\ol{\pr}_b^{\,*}\Box^{(q)}_b=\Box^{(q-1)}_b\ol{\pr}_b^{\,*}$
and get 
\begin{gather*}
\pr_t(\ddbar_b(e^{i\psi}a))+\Box^{(q+1)}_b(\ddbar_b(e^{i\psi}a))\sim 0, \\
\pr_t(\ol{\pr}_b^{\,*}(e^{i\psi}a))+\Box^{(q-1)}_b(\ol{\pr}_b^{\,*}(e^{i\psi}a))\sim 0.
\end{gather*}
Put
\[\ddbar_b(e^{i\psi}a)=e^{i\psi}\widehat a,\ \ol{\pr}_b^{\,*}(e^{i\psi}a)=e^{i\psi}\Td a.\]
We have
\begin{gather*}
(\pr_t+\Box^{(q+1)}_b)(e^{i\psi}\widehat a)\sim0, \\
(\pr_t+\Box^{(q-1)}_b)(e^{i\psi}\Td a)\sim0.
\end{gather*}
The corresponding degrees of $\widehat a$ and $\Td a$ are $q+1$ and $q-1$. We deduce as above that $\widehat a$ and $\Td a$ decay
exponentially fast in $t$. This also applies to
\[\Box^{(q)}_b(ae^{i\psi})=\ddbar_b(\ol{\pr}_b^{\,*}ae^{i\psi})+\ol{\pr}_b^{\,*}(\ddbar_bae^{i\psi}) 
    =\ddbar_b(e^{i\psi}\Td a)+\ol{\pr}_b^{\,*}(e^{i\psi}\widehat a).\]
Thus, $\pr_t(ae^{i\psi})$ decay exponentially fast in $t$. Since $\pr_t\psi$ decay exponentially fast in $t$ so does $\pr_ta$.
Hence, there exist positively homogeneous functions of degree $-j$
\[a_j(\infty, x, \eta)\in
C^{\infty}\big(T^*B,\Lambda^qT^{*(0,1)}X\boxtimes\Lambda^qT^{*(0,1)}X\big),\ \ j=0,1,2,\ldots,\]
such that $a_j(t, x, \eta)$ converges exponentially
fast to $a_j(\infty, x, \eta)$, $t\to\infty$, for all $j\in\N_0$.

Choose $\chi\in\cC^\infty_0(\Real^{2n+1})$ so that $\chi(\eta)=1$ when $\abs{\eta}<1$ and $\chi(\eta)=0$ when $\abs{\eta}>2$. We formally set
\[\begin{split}
A = \frac{1}{(2\pi)^{2n+1}}\int\int^{\infty}_0\!\!\Bigl(e^{i(\psi(t,x,\eta)-\langle y,\eta\rangle)}a(t,x,\eta) 
  -e^{i(\psi(\infty,x,\eta)-\langle y,\eta\rangle)}a(\infty,x,\eta)\Bigr)(1-\chi(\eta))\,dt\,d\eta
\end{split}\]
and
\[S=\frac{1}{(2\pi)^{2n+1}}\int\! \big(e^{i(\psi(\infty,x,\eta)-\langle y,\eta\rangle)}a(\infty,x,\eta)\big)d\eta.\]
We can show that $A$ is a pseudodifferential operator of order $-1$ and type $(\frac{1}{2},\frac{1}{2})$ satisfying
\[S+\Box^{(q)}_b\circ A\equiv I,\ \ \Box^{(q)}_b\circ S\equiv0.\]
Moreover, the stationary phase formula of Melin-Sj\"{o}strand~\cite{MS74} shows that $S\equiv S_-+S_+$, where $S_-$, $S_+$ are as in Theorem~\ref{s3-t1}.
\end{proof}
The following result describes the phase function in local coordinates.
\begin{thm} \label{s3-t1-1}
Let $X$, $B$ and $(n_-,n_+)$ be as in Theorem \ref{s3-t0}.
For a given point $x_0\in B$, let $\{W_j\}_{j=1}^n$
be an orthonormal frame of $T^{(1, 0)}X$ in a neighborhood of $x_0$, such that
the Levi form is diagonal at $x_0$, i.e.\ $\mathcal{L}_{x_{0}}(W_{j},\overline{W}_{j})=\mu_{j}$, $j=1,\ldots,n$.
We take local coordinates
$x=(x_1,\ldots,x_{2n+1})$, $z_j=x_{2j-1}+ix_{2j}$, $j=1,\ldots,n$,
defined on some neighborhood of $x_0$ such that $\omega_0(x_0)=dx_{2n+1}$, $x(x_0)=0$, and for some $c_j\in\Complex$, $j=1,\ldots,n$\,,
\[W_j=\frac{\pr}{\pr z_j}-i\mu_j\ol z_j\frac{\pr}{\pr x_{2n+1}}-
c_jx_{2n+1}\frac{\pr}{\pr x_{2n+1}}+O(\abs{x}^2),\ j=1,\ldots,n\,.\]
Set
$y=(y_1,\ldots,y_{2n+1})$, $w_j=y_{2j-1}+iy_{2j}$, $j=1,\ldots,n$.
Then, for $\varphi_-$ in Theorem~\ref{s3-t1}, we have
\begin{equation} \label{s3-ephase1}
{\rm Im\,}\varphi_-(x,y)\geq c\sum^{2n}_{j=1}\abs{x_j-y_j}^2,\ \ c>0,
\end{equation}
in some neighborhood of $(0,0)$ and
\begin{equation} \label{s3-ephase2}
\begin{split}
&\varphi_-(x, y)=-x_{2n+1}+y_{2n+1}+i\sum^{n-1}_{j=1}\abs{\mu_j}\abs{z_j-w_j}^2 \\
&\quad+\sum^{n-1}_{j=1}\Bigr(i\mu_j(\ol z_jw_j-z_j\ol w_j)+c_j(-z_jx_{2n+1}+w_jy_{2n+1})\\
&\quad+\ol c_j(-\ol z_jx_{2n+1}+\ol w_jy_{2n+1})\Bigr)+(x_{2n+1}-y_{2n+1})f(x, y) +O(\abs{(x, y)}^3),
\end{split}
\end{equation}
where $f$ is smooth and satisfies $f(0,0)=0$, $f(x, y)=\ol f(y, x)$.
\end{thm}

\begin{rem} \label{s3-rabis}
If we go through the proofs of Theorem~\ref{s3-t0} and Theorem~\ref{s3-t1} (see~\cite{Hsiao08}), it is not difficult to see that Theorem~\ref{s3-t0} and Theorem~\ref{s3-t1} have straightforward generalizations to the case when the functions take values in $\Lambda^qT^{*(0,1)}X\otimes F$, for a given smooth  CR vector bundle $F$ over $X$. We recall that $F$ is a CR vector
bundle if its transition functions are CR.
\end{rem}

\begin{rem} \label{s3-r1}
Let $\widehat\varphi\in\cC^\infty(B\times B)$. We assume that $\widehat\varphi$ satisfies \eqref{s3-e12-bis}--\eqref{s3-e15-bis}, \eqref{s3-e16-bisbis} and
\eqref{s3-ephase1}, \eqref{s3-ephase2}. Then it is well-known
(see~\cite[\S3,7]{Hsiao08} and Menikoff-Sj\"{o}strand~\cite{MS78}) that $\widehat\varphi(x,y)t$, $t>0$, and $\varphi_-(x,y)t$, $t>0$, are equivalent at each point of ${\rm diag\,}\bigr((\Sigma^-\bigcap T^*B)\times(\Sigma^-\bigcap T^*B)\bigr)$ in the sense of Melin-Sj\"{o}strand (see Melin-Sj\"{o}strand~\cite[p.\,172]{MS74}). We recall briefly that $\widehat\varphi(x,y)t$, $t>0$, and $\varphi_-(x,y)t$, $t>0$, are equivalent at each point of 
\[{\rm diag\,}\bigr((\Sigma^-\cap T^*B)\times(\Sigma^-\cap T^*B)\bigr)\] 
if for every 
\[(x_0,-\lambda_0\omega_0)=(x_0,\lambda_0d_x\varphi_-(x_0,x_0))=(x_0,\lambda_0d_x\widehat\varphi(x_0,x_0))\in\Sigma^-\cap T^*B,\]
there is a conic neighborhood $\Gamma$ of $(x_0,x_0,\lambda_0)$, such that for every $a(x,y,t)\in S^m_{{\rm cl\,}}(B\times B\times\Real_+)$, $m\in\mathbb Z$, with support in $\Gamma$, we can find $\widehat a(x,y,t)\in S^m_{{\rm cl\,}}(B\times B\times\Real_+)$ with support in $\Gamma$, such that 
\[\int^\infty_0e^{i\varphi_-(x,y)t}a(x,y,t)dt\equiv\int^\infty_0e^{i\widehat\varphi(x,y)t}\widehat a(x,y,t)dt\] 
and vise versa, where $S^m_{{\rm cl\,}}$ denotes the classical symbol of order $m$ (see~\cite[p.\,38]{GS94} for the definition of $S^m_{{\rm cl\,}}$).
\end{rem}

If $\omega\in T^{*(0,1)}_xX$, as \eqref{s1-e0}, we let $(\omega\wedge)^*:\Lambda^{q+1}T^{*(0,1)}_xX\To\Lambda^qT^{*(0,1)}_xX$, $q\geq0$,
denote the adjoint of left exterior multiplication $\omega\wedge:\Lambda^qT^{*(0,1)}_xX\To\Lambda^{q+1}T^{*(0,1)}_xX$.

The following formula for the principal symbol $s^0_-$ on the diagonal follows from \cite[\S 8]{Hsiao08}, its calculation being local in nature.
\begin{thm} \label{s3-t2}
Let $q=n_-$. For a given point $x_0\in X$, let $\{W_j\}_{j=1}^n$ be an
orthonormal frame of $T^{(1,0)}X$ near $x_0$, for which the Levi form
is diagonal at $x_0$. Put $\mathcal{L}_{x_0}(W_j,\ol W_j)=\mu_j(x_0)$, $j=1,\ldots,n$\,. Let $\{T_j\}_{j=1}^n$ denote the
dual basis of $T^{*(0,1)}X$, dual to $\{\ol W_j\}^n_{j=1}$. We assume that
$\mu_j(x_0)<0$ if\, $1\leq j\leq n_-$. Then, for $s^0_-(x,y)$ in
\eqref{s3-e11-bis}, we have
\[s^0_-(x_0, x_0)=\frac{1}{2}\abs{\mu_1(x_0)}\cdots
\abs{\mu_{n}(x_0)}\pi^{-n-1}\prod_{j=1}^{n_-}(T_j(x_0)\wedge)\circ
(T_j(x_0)\wedge)^*.\]
\end{thm}
We return now to the situation where $X$ is the Grauert tube of a line bundle $L$ as in Section \ref{s:CR} and use the notations introduced there.
Let $(z,\theta)$ be the coordinates as in \eqref{s3-e0}, \eqref{s3-e1} on $B=D\times]-\varepsilon_0,\varepsilon_0[$, $\varepsilon_0>0$,
$D\Subset M$. 
Until further notice, we work with the local coordinates
$(z,\theta)=(x',x_{2n+1})=x$. If we denote the holomorphic coordinates of $D$ by $w_j=y_{2j-1}+iy_{2j}$, $j=1,\ldots,n$, and by $y_{2n+1}$ the coordinate of $]-\varepsilon_0,\varepsilon_0[$, we also write $(w,y_{2n+1})=(y',y_{2n+1})=y$, $y'=(y_1,\ldots,y_{2n})$. Let $\xi$ be the dual variables of $x$. From \eqref{s3-e3} and \eqref{s3-e4}, we can check that the principal symbol of $\Box^{(q)}_b$ satisfies
\begin{equation} \label{s3-e16-bisbisbis}
p_0(x, \xi)=p_0(x', \xi).
\end{equation}
That is, the principal symbol of $\Box^{(q)}_b$ is independent of $x_{2n+1}$.

Using \eqref{s3-e15-bis} and recalling \eqref{s3-e1-1}, we have
\[d_x\varphi_-(x,x)=-dx_{2n+1}+a(x')dx',\ \ a\in\cC^\infty.\]
Thus, near a given point $(x_0,x_0)\in B\times B$, we have $\frac{\pr\varphi_-}{\pr x_{2n+1}}\neq0$.
Using the Malgrange preparation theorem \cite[Th.\,7.57]{Hor03}, we have
\begin{equation} \label{s3-e16-bisb}
\varphi_-(x,y)=g(x,y)(-x_{2n+1}+h(x',y))
\end{equation}
in some neighborhood of $(x_0,x_0)$, where $g, h\in\cC^\infty$,
$g(x,x)=1$, $h(x',x)=x_{2n+1}$. Since ${\rm
Im\,}\varphi_-\geq0$, it is not difficult to see that
${\rm Im\,}h\geq0$ in some neighborhood of $(x_0,x_0)$. We may take $B$
small enough so that \eqref{s3-e16-bisb} holds and ${\rm
Im\,}h\geq0$ on $B\times B$.  From the global theory of Fourier integral
operators \cite[Th.\,4.2]{MS74}, we see that
$\varphi_-(x,y)t$ and $(-x_{2n+1}+h(x',y))t$ are equivalent in the
sense of Melin-Sj\"{o}strand. We can replace the phase $\varphi_-$
by $-x_{2n+1}+h(x',y)$. Again from \eqref{s3-e15-bis}, we have
\[\frac{\pr h}{\pr x'}(x',x)dx'-dx_{2n+1}=-\omega_0(x)=-dx_{2n+1}+a(x')dx'.\]
Thus, $\frac{\pr h}{\pr x'}(x',x)$ is independent of $x_{2n+1}$. We conclude that
\begin{equation} \label{s3-e16-bisbc}
\frac{\pr h}{\pr x'}(x',x)dx'-dx_{2n+1}=\frac{\pr h}{\pr x'}(x',x')dx'-dx_{2n+1}=-\omega_0(x).
\end{equation}
Similarly, we have
\begin{equation} \label{s3-e16-bisbd}
\frac{\pr h}{\pr y}(y',y)dy=dy_{2n+1}+\frac{\pr h}{\pr y'}(y',y')dy'=\omega_0(y).
\end{equation}
Put
\[\widehat\varphi=-x_{2n+1}+y_{2n+1}+h(x',y').\]
Note that
$-x_{2n+1}+h(x',y)$ satisfies \eqref{s3-e16-bisbis}. From this and
\eqref{s3-e16-bisbisbis}, we have
\[p_0\bigr(x,(\frac{\pr h}{\pr x'}(x',y),-1)\bigr)=p_0\bigr(x',(\frac{\pr h}{\pr x'}(x',y),-1)\bigr)=f(x,y)(-x_{2n+1}+h(x',y))+O(\abs{x-y}^N)\]
for all $N\in\mathbb N$, for some $f\in\cC^\infty$. Hence,
\begin{equation} \label{s3-e16-bisbf}
p_0(x, \widehat{\varphi}'_x)=p_0(x', \widehat{\varphi}'_x)=f(x,y')(-x_{2n+1}+h(x',y'))+O(\abs{x'-y'}^N+\abs{x_{2n+1}}^N),
\end{equation}
for all $N\in\mathbb N$. We replace $x_{2n+1}$ by
$x_{2n+1}-y_{2n+1}$ in \eqref{s3-e16-bisbf} and get
\begin{equation} \label{s3-e16-bisbfg}
p_0(x, \widehat{\varphi}'_x)=p_0(x', \widehat{\varphi}'_x)=\widehat f(x,y)\widehat\varphi+O(\abs{x-y}^N),
\end{equation}
for all $N\in\mathbb N$, for some $\widehat f\in\cC^\infty$. Thus, $\widehat\varphi$ satisfies \eqref{s3-e16-bisbis}.
Note that $p_0(x,\widehat{\varphi}'_x)$ is independent of $x_{2n+1}$. Take $x_{2n+1}=y_{2n+1}+h(x',y')$ in \eqref{s3-e16-bisbfg}
and notice that $h(x',x')=0$, we conclude that
\begin{equation} \label{s3-e16-bisbfgh}
p_0(x,\widehat{\varphi}'_x)=O(\abs{x'-y'}^N),\ \ \forall N\in\mathbb N.
\end{equation}

Furthermore, from \eqref{s3-e16-bisbc} and \eqref{s3-e16-bisbd}, we see that $\widehat\varphi$ satisfies \eqref{s3-e15-bis}. Moreover, for a given
point $p\in D$, we may take local coordinates
$z=(z_1,\ldots,z_n)$ centered at $p$ such that 
\begin{equation} \label{s3-e16-bisbb}
\begin{split}
&\Theta(p)=\sqrt{-1}\sum^n_{j=1}dz_j\wedge d\ol z_j\,,\\
&\phi(z)=\sum^{n}_{j=1}\lambda_j\abs{z_j}^2+O(\abs{z}^3)\,,\:\text{$z$ near $p$}\,,\:\{\lambda_j\}_{j=1}^n\subset\R\setminus\{0\}\,.
\end{split}
\end{equation}
From \eqref{s3-ephase2} and \eqref{s3-e16-bisb}, it is not difficult to see that
\begin{equation} \label{s3-e16-bisbe}
h(x',y')=i\sum^n_{j=1}\abs{\lambda_j}\abs{z_j-w_j}^2+i\sum^n_{j=1}\lambda_j(\ol
z_jw_j-z_j\ol w_j)+O(\abs{(x',y')}^3).
\end{equation}
Thus $\widehat\varphi$ satisfies \eqref{s3-ephase2}. 
Formula \eqref{s3-e16-bisbe} and the Taylor expansion of $h(x',y')$ at $x'=y'$ yield
\[{\rm Im\,}h(x',y')\geq c\abs{x'-y'}^2,\ \ c>0.\]
Thus, $\widehat\varphi=0$ if and only if $x=y$. We conclude that $\widehat\varphi$ satisfies \eqref{s3-e12-bis}--\eqref{s3-e15-bis}, \eqref{s3-e16-bisbis} and \eqref{s3-ephase1}, \eqref{s3-ephase2}. In view of Remark~\ref{s3-r1}, we see that $t\varphi_-$ and $t\widehat\varphi $ are equivalent at each point of ${\rm diag\,}\bigr((\Sigma^-\cap T^*B)\times(\Sigma^-\cap T^*B)\bigr)$ in the sense of Melin-Sj\"{o}strand.
Since $\varphi_-(x,y)=-\ol\varphi_-(y,x)$, we can replace $\varphi_-$ by
\[\frac{1}{2}(\widehat\varphi(x,y)-\ol{\widehat\varphi(y,x)})=(-x_{2n+1}+y_{2n+1})+\frac{1}{2}(h(x',y')-\ol h(y',x')).\]
Summing up, we get the following.
\begin{thm} \label{s3-t1-bis}
Let $(L,h^L)$ be a holomorphic Hermitian line bundle over $M$ whose curvature $R^L$ is non-degenerate of constant signature $(n_-,n_+)$ at each point of an open set $D\Subset M$. We assume that $L$ is trivial on $D$.
Let $\pi:X\to M$ be the Grauert tube of $L$ {\rm(}cf.\ \eqref{grauert_tube}{\rm)} and let $B=\pi^{-1}(D)$.
With the notations used before, we can take the phase $\varphi_-(x, y)$ in Theorem~\ref{s3-t1} so that
\begin{equation} \label{s3-e16-bisbbbbb}
\begin{split}
&\varphi_-(x,y)=-x_{2n+1}+y_{2n+1}+\Psi(z,w),\ \
\Psi(x',y')=\Psi(z,w)\in\cC^\infty,\\
&\mbox{$p_0(x,\varphi'_-(x,y))=O(\abs{x'-y'}^N)$, locally uniformly on $B\times B$, for all $N\in\mathbb N$},
\end{split}
\end{equation}
where $p_0(x,\xi)$ is the principal symbol of $\Box^{(q)}_b$ and $\Psi$ satisfies 
\begin{equation}\label{prop_psi}
\Psi(z,w)=-\ol\Psi(w,z)\,,\ \exists\, c>0:\ {\rm Im\,}\Psi\geq c\abs{z-w}^2\,,\ \Psi(z,w)=0\Leftrightarrow z=w \,.
\end{equation}
For a given
point $p\in D$, let $z=(z_1,\ldots,z_n)$ be local holomorphic coordinates
 centered at $p$ satisfying \eqref{s3-e16-bisbb}. Then, near $(0,0)$, we have
\begin{equation} \label{s3-e16-bisbg}
\Psi(z,w)=i\sum^n_{j=1}\abs{\lambda_j}\abs{z_j-w_j}^2+i\sum^n_{j=1}\lambda_j(\ol
z_jw_j-z_j\ol w_j)+O(\abs{(z,w)}^3).
\end{equation}
\end{thm}
From now on, we assume that $\varphi_-$ has the form \eqref{s3-e16-bisbbbbb}.
\subsection{Semi-classical Hodge decomposition for the Kodaira Laplacian}\label{s:smcl-hodge}
In this Section we apply the results about the Szeg\"o kernel previously deduced in order to describe the semiclassical behaviour of the spectrum of the Kodaira Laplacian $\Box^{(q)}_k$.
We work locally in the following setup.
\begin{setup}\label{local_data}
Let $(M,\Theta)$ be a Hermitian manifold, $(L,h^L)$ be a holomorphic Hermitian line bundle on $M$.
Assume that the curvature $\sqrt{-1}R^L$ is non-degenerate of constant signature $(n_-,n_+)$ on the domain of a chart $(D,z)\cong(D,x)\Subset M$. Assume that $L|_D$ is trivial and let $s$ be a frame of $L|_D$ and set $\abs{s}_{h^L}^2=e^{-2\phi}$.
\end{setup}
We introduce some notations.
For an open set $D\Subset M$ and any $k$-dependent continuous function
\[F_k:H^s_{{\rm comp\,}}(D,\Lambda^qT^{*(0,1)}M)\To H^{s'}_{{\rm loc\,}}(D,\Lambda^qT^{*(0,1)}M),\ \ s, s'\in\Real,\]
we write
\[F_k=O(k^{n_0}):H^s_{{\rm comp\,}}(D,\Lambda^qT^{*(0,1)}M)\To H^{s'}_{{\rm loc\,}}(D,\Lambda^qT^{*(0,1)}M),\ \ n_0\in\mathbb Z,\]
if for any $\chi_0, \chi_1\in\cC^\infty_0(D)$, there is a positive constant $c$, $c$ is independent of $k$, such that
\begin{equation} \label{s1-e1su}
\norm{(\chi_0F_k\chi_1)u}_{s'}\leq ck^{n_0}\norm{u}_{s},\ \ \forall u\in H^s_{{\rm loc\,}}(D,\Lambda^qT^{*(0,1)}M),
\end{equation}
where $\norm{u}_s$ is the usual Sobolev norm of order $s$.

A $k$-dependent smoothing operator
$A_k:\Omega^{0,q}_0(D)\To\Omega^{0,q}(D)$ is called $k$-negligible
if the kernel $A_k(x, y)$ of $A_k$ satisfies
$\abs{\pr^\alpha_x\pr^\beta_yA_k(x, y)}=O(k^{-N})$ locally uniformly
on every compact set in $D\times D$, for all multi-indices $\alpha$,
$\beta$ and all $N\in\mathbb N$. $A_k$ is $k$-negligible if and only if
\[A_k=O(k^{-N'}): H^s_{\rm comp\,}(D, \Lambda^qT^{*(0,1)}M)\To H^{s+N}_{\rm loc\,}(D, \Lambda^qT^{*(0,1)}M)\,,\quad\text{for all $N, N'\geq0$, $s\in\mathbb Z$.}\]
Let
$C_k:\Omega^{0,q}_0(D)\To\Omega^{0,q}(D)$ be another $k$-dependent
smoothing operator. We write $A_k\equiv C_k\mod O(k^{-\infty})$ or $A_k(x,y)\equiv C_k(x,y)\mod O(k^{-\infty})$ if
$A_k-C_k$ is $k$-negligible.

We recall the definition of semi-classical H\"{o}rmander symbol spaces:

\begin{defn} \label{s1-d1}
Let $U$ be an open set in $\Real^N$. Let $S(1;U)=S(1)$ be the set of
$a\in\cC^\infty(U)$ such that for every $\alpha\in\mathbb N^N_0$, there
exists $C_\alpha>0$, such that $\abs{\pr^\alpha_xa(x)}\leq
C_\alpha$ on $U$. If $a=a(x,k)$ depends on $k\in]1,\infty[$, we say that
$a(x,k)\in S_{{\rm loc\,}}(1)$ if $\chi(x)a(x,k)$ uniformly bounded
in $S(1)$ when $k$ varies in $]1,\infty[$, for any $\chi\in
\cC^\infty_0(U)$. For $m\in\Real$, we put $S^m_{{\rm
loc}}(1)=k^mS_{{\rm loc\,}}(1)$. If $a_j\in S^{m_j}_{{\rm
loc\,}}(1)$, $m_j\searrow-\infty$, we say that $a\sim
\sum^\infty_{j=0}a_j$ in $S^{m_0}_{{\rm loc\,}}(1)$ if
$a-\sum^{N_0}_{j=0}a_j\in S^{m_{N_0+1}}_{{\rm loc\,}}(1)$ for every
$N_0$. From this, we form $S^m_{{\rm loc\,}}(1;Y, E)$ in the
natural way, where $Y$ is a smooth paracompact manifold and $E$ is a
vector bundle over $Y$.
\end{defn}
Let $D$, $s$, $\phi$ be as in Setup \ref{local_data}.
Let $(z, \theta)$ be the local
coordinates as in \eqref{s3-e0}, \eqref{s3-e1} defined on $D\times]-\varepsilon_0, \varepsilon_0[$, $\pi\geq\varepsilon_0>0$. 
Let $\Box^{(q)}_s$ be the operator as in \eqref{s1-e3} and \eqref{s1-e4}. Since $\pr\ddbar\phi$ has constant signature $(n_-, n_+)$
at each point of $D$, from \eqref{s3-e2}, we know that the Levi form $\mathcal{L}$ has constant signature $(n_-, n_+)$ at each point
of $D\times]-\varepsilon_0, \varepsilon_0[$.

Let $q=n_-$ or $n_+$ and let $S_-$, $S_+$ be the approximate Szeg\"o kernels defined in
Theorem~\ref{s3-t1}. Define also the approximate Szeg\"o kernel
\begin{equation}\label{s3-e15-1}
S=S_-+S_+\,.
\end{equation}
 Let $\chi(\theta), \chi_1(\theta)\in\cC^\infty_0(]-\varepsilon_0, \varepsilon_0[)$, $\chi, \chi_1\geq 0$. We assume that $\chi_1=1$ on ${\rm Supp\,}\chi$. We take $\chi$ so that  $\int\chi(\theta)d\theta=1$. Put
\begin{equation} \label{s3-e16-1}
\chi_k(\theta)=e^{-ik\theta}\chi(\theta).
\end{equation}
The approximate Szeg\"o kernel was introduced in \eqref{s3-e15-1}. We introduce the \emph{localized approximate Szeg\"o kernel} $\mathcal{S}_{k}$ by
\begin{align} \label{s3-e17-1}
\mathcal{S}_{k}: H^s_{{\rm loc\,}}(D, \Lambda^qT^{*(0,1)}M)&\To H^s_{{\rm loc\,}}(D, \Lambda^qT^{*(0,1)}M),\ \ \forall s\in\mathbb N_0, \nonumber\\
u(z)&\To \int e^{i\theta k}\chi_1(\theta)S(\chi_ku)(z,\theta)d\theta.
\end{align}
Let $u(z)\in H^s_{{\rm loc\,}}(D, \Lambda^qT^{*(0,1)}M)$, $s\in\mathbb N_0$.
We have $\chi_k(\theta)u(z)\in H^s_{{\rm loc\,}}(D\times]-\varepsilon_0,\varepsilon_0[, \Lambda^qT^{*(0,1)}X)$.
From Theorem~\ref{s3-t1}, we know that
\[S(\chi_ku)\in H^{s}_{{\rm loc\,}}(D\times]-\varepsilon_0, \varepsilon_0[, \Lambda^qT^{*(0,1)}X).\]
From this, it is not difficult to see that $\int e^{i\theta
k}\chi_1(\theta)S(\chi_ku)(z, \theta)d\theta\in H^s_{{\rm loc\,}}(D,
\Lambda^qT^{*(0,1)}M)$. Thus, the localization $\mathcal{S}_{k}$ is well-defined.
Since $S$ is properly supported, $\mathcal{S}_{k}$ is properly
supported, too. Moreover, from \eqref{s3-e6-bis0} and \eqref{s3-e17-1},
we can check that
\begin{equation} \label{s3-e17-2bis}
\mathcal{S}_{k}=O(k^s): H^s_{{\rm comp\,}}(D, \Lambda^qT^{*(0,1)}M)\To H^{s}_{{\rm comp\,}}(D, \Lambda^qT^{*(0,1)}M),
\end{equation}
for all $s\in\mathbb N_0$.

Let $\mathcal{S}_{k}^*:\mathscr D'(D, \Lambda^qT^{*(0,1)}M)\To\mathscr D'(D, \Lambda^qT^{*(0,1)}M)$ be the formal adjoint of $\mathcal{S}_{k}$ with respect to $(\,\cdot\,,\cdot)$. Then $\mathcal{S}_{k}^*$ is also properly supported and we have
\begin{equation} \label{s3-e17-3}
\mathcal{S}_{k}^*:\mathscr E'(D, \Lambda^qT^{*(0,1)}M)\To\mathscr E'(D, \Lambda^qT^{*(0,1)}M).
\end{equation}
From \eqref{s3-e9}, we have
\begin{equation} \label{s3-e18-1}
\begin{split}
\Box^{(q)}_s\circ\bigr(\int e^{i\theta k}\chi_1(\theta)S(\chi_ku)d\theta\bigr)&=\int e^{i\theta k}\bigr(\Box^{(q)}_b(\chi_1S)\bigr)(\chi_ku)(z, \theta)d\theta\\
&=\int e^{i\theta k}\bigr(\Box^{(q)}_b(\chi_1S\Td\chi)\bigr)(\chi_ku)(z, \theta)d\theta,
\end{split}
\end{equation}
where $\Td\chi\in\cC^\infty_0(]-\varepsilon_0, \varepsilon_0[)$, $\Td\chi=1$ on ${\rm Supp\,}\chi$ and $\chi_1=1$ on ${\rm Supp\,}\Td\chi$ and  $u\in\Omega^{0,q}_0(D)$. Note that $\Box^{(q)}_b(\chi_1S\Td\chi)=\Box^{(q)}_b(S\Td\chi)-\Box^{(q)}_b((1-\chi_1)S\Td\chi)$. From Theorem~\ref{s3-t1}, we know that $\Box^{(q)}_bS$ is smoothing and the kernel of $S$ is smoothing away the diagonal. Thus,
$(1-\chi_1)S\Td\chi$ is smoothing. It follows that $\Box^{(q)}_b((1-\chi_1)S\Td\chi)$ is smoothing. We conclude that
$\Box^{(q)}_b(\chi_1S\Td\chi)$ is smoothing.
Let $K((z, \theta), (w, \eta))\in\cC^\infty$ be the distribution kernel of $\Box^{(q)}_b(\chi_1S\Td\chi)$, where $w=(w_1,\ldots,w_n)$ are the local holomorphic
coordinates of $D$ and $\eta$ is the coordinate of $]-\varepsilon_0,\varepsilon_0[$. From \eqref{s3-e18-1} and recall the form $\chi_k$ (see \eqref{s3-e16-1}), we see that the distribution kernel of $\Box^{(q)}_s\mathcal{S}_{k}$ is given by
\begin{equation} \label{s3-e19-1}
(\Box^{(q)}_s\mathcal{S}_{k})(z, w)=\int e^{i(\theta-\eta)k}K((z,\theta), (w,\eta))\chi(\eta)d\eta d\theta.
\end{equation}
For $N\in\mathbb N$, we have
\begin{equation} \label{s3-e19-1bis}
\begin{split}
\abs{k^N(\Box^{(q)}_s\mathcal{S}_{k})(z,w)}&=\abs{\int \bigr((i\frac{\pr}{\pr\eta})^Ne^{i(\theta-\eta)k})\bigr)K((z,\theta),(w,\eta))\chi(\eta)d\eta d\theta}\\
&=\abs{\int e^{i(\theta-\eta)k}(-i\frac{\pr}{\pr\eta})^N\bigr(K((z,\theta),(w,\eta))\chi(\eta)\bigr)d\eta d\theta}.
\end{split}
\end{equation}
Thus, $(\Box^{(q)}_s\mathcal{S}_{k})(z, w)=O(k^{-N})$, locally uniformly for all $N\in\mathbb N$, and similarly for the derivatives. We deduce that
\begin{equation} \label{s3-e20-1}
\Box^{(q)}_s\mathcal{S}_{k}\equiv 0\mod O(k^{-\infty}).
\end{equation}
Thus,
\begin{equation} \label{s3-e21-1}
\mathcal{S}_{k}^*\Box^{(q)}_s\equiv 0\mod O(k^{-\infty}).
\end{equation}

Let $A$ be the partial parametrix of $\Box^{(q)}_b$ described in Theorem~\ref{s3-t1}. Define the \emph{localized partial parametrix} $\mathcal{A}_{k}$ by
\begin{equation} \label{s3-e17}
\begin{split}
\mathcal{A}_{k}: H^s_{{\rm loc\,}}(D, \Lambda^qT^{*(0,1)}M)&\To H^{s+1}_{{\rm loc\,}}(D, \Lambda^qT^{*(0,1)}M),\ \ \forall s\in\mathbb N_0, \\
u(z)&\mapsto \int e^{i\theta k}\chi_1A(\chi_ku)(z,\theta)d\theta.
\end{split}
\end{equation}
As above, we can show that $\mathcal{A}_{k}$ is well-defined.
Since $A$ is properly supported, $\mathcal{A}_{k}$ is properly
supported, too. Moreover, from \eqref{s3-e6-bis0} and \eqref{s3-e17}, we
can check that
\begin{equation} \label{s3-e18}
\mathcal{A}_{k}=O(k^s): H^s_{{\rm comp\,}}(D, \Lambda^qT^{*(0,1)}M)\To H^{s+1}_{{\rm comp\,}}(D, \Lambda^qT^{*(0,1)}M),
\end{equation}
for all $s\in\mathbb N_0$.

Let $\mathcal{A}_{k}^*:\mathscr D'(D,\Lambda^qT^{*(0,1)}M)\To\mathscr D'(D,\Lambda^qT^{*(0,1)}M)$ be the formal adjoint of
$\mathcal{A}_{k}$ with respect to $(\,\cdot\,,\cdot)$. We can check that
\[(\mathcal{A}_{k}^*v)(z)=\int\ol{\chi_k(\theta)}A^*(ve^{-i\theta k}\chi_1)(z,\theta)d\theta\in\Omega^{0,q}_0(D),\]
for all $v\in\Omega^{0,q}_0(D)$. Thus, $\mathcal{A}_{k}^*:\Omega^{0,q}_0(D)\To\Omega^{0,q}_0(D)$.
Moreover, as before, we can show that
\begin{equation} \label{s3-e18adjoint}
\mathcal{A}_{k}^*=O(k^s): H^s_{{\rm comp\,}}(D, \Lambda^qT^{*(0,1)}M)\To H^{s+1}_{{\rm comp\,}}(D, \Lambda^qT^{*(0,1)}M)\,,\quad\text{for all $s\in\mathbb N_0$.}
\end{equation}

Let $u\in\Omega^{0,q}_0(D)$. From \eqref{s3-e9}, we have
\begin{equation*}
\begin{split}
\Box^{(q)}_s(\mathcal{A}_{k}u)&=\Box^{(q)}_s\circ\bigr(\int e^{i\theta k}\chi_1A(\chi_ku)d\theta\bigr)=\int e^{i\theta k}\bigr(\Box^{(q)}_b\chi_1A\Td\chi)(\chi_ku)(z, \theta)d\theta,
\end{split}
\end{equation*}
where $\Td\chi$ is as in \eqref{s3-e18-1}. Note that $\Box^{(q)}_b(\chi_1A\Td\chi)=\Box^{(q)}_b(A\Td\chi)-\Box^{(q)}_b((1-\chi_1)A\Td\chi)$. From Theorem~\ref{s3-t1}, we know that $\Box^{(q)}_bA+S=I$ and the kernel of $A$ is smoothing away the diagonal. Thus,
$(1-\chi_1)A\Td\chi$ is smoothing. It follows that $\Box^{(q)}_b((1-\chi_1)A\Td\chi)$ is smoothing. We conclude that
$\Box^{(q)}_b(\chi_1A\Td\chi)\equiv(I-S)\Td\chi$. From this, we get
\begin{equation} \label{s3-e21-3}
\begin{split}
\Box^{(q)}_s(\mathcal{A}_{k}u)&=\int e^{i\theta k}(I-S)(\chi_ku)(z, \theta)d\theta+\int e^{i\theta k}F(\chi_ku)(z, \theta)d\theta\\
&=u-\int e^{i\theta k}S(\chi_ku)(z,\theta)d\theta+\int e^{i\theta k}F(\chi_ku)(z, \theta)d\theta\\\
&=(I-\mathcal{S}_{k})u-\int e^{i\theta k}(1-\chi_1)S(\chi_ku)(z,\theta)d\theta+\int e^{i\theta k}F(\chi_ku)(z, \theta)d\theta,
\end{split}
\end{equation}
where $F$ is a smoothing operator. We can repeat the procedure as in \eqref{s3-e19-1bis} and conclude that the operator
\[u\To \int e^{i\theta k}F(\chi_k u)(z,\theta)d\theta,\ \ u\in\Omega^{0,q}_0(D),\]
is $k$-negligible. Similarly, since $(1-\chi_1)S\chi$ is smoothing, the operator
\[u\To \int e^{i\theta k}(1-\chi_1)S(\chi_k u)(z,\theta)d\theta,\ \ u\in\Omega^{0,q}_0(D),\]
is also $k$-negligible. Summing up, we obtain
\begin{equation} \label{s3-e21-4}
\Box^{(q)}_s\mathcal{A}_{k}+\mathcal{S}_{k}\equiv I\mod O(k^{-\infty}).
\end{equation}

We may replace $\mathcal{S}_{k}$ by $I-\Box^{(q)}_s\mathcal{A}_{k}$ and we have $\Box^{(q)}_s\mathcal{A}_{k}+\mathcal{S}_{k}=I$ and hence $\mathcal{A}_{k}^*\Box^{(q)}_s+\mathcal{S}_{k}^*=I$. Thus,
\begin{equation} \label{s3-e19}
\mathcal{S}_{k}=(\mathcal{A}_{k}^*\Box^{(q)}_s+\mathcal{S}_{k}^*)\mathcal{S}_{k}=\mathcal{A}_{k}^*\Box^{(q)}_s\mathcal{S}_{k}+\mathcal{S}_{k}^*\mathcal{S}_{k}.
\end{equation}
From \eqref{s3-e20-1} and \eqref{s3-e18adjoint}, we see that
\[\mathcal{A}_{k}^*\Box^{(q)}_s\mathcal{S}_{k}=O(k^{-N'}):H^s_{{\rm comp\,}}(D,\Lambda^qT^{*(0,1)}M)\To H^{s+N}_{{\rm comp\,}}(D, \Lambda^qT^{*(0,1)}M'),\] for all $s\in\mathbb Z$ and $N', N\in\mathbb N$. Thus, $\mathcal{A}_{k}^*\Box^{(q)}_s\mathcal{S}_{k}\equiv0\mod O(k^{-\infty})$. From this and \eqref{s3-e19}, we get
\begin{equation} \label{s3-e20}
\mathcal{S}_{k}^*\mathcal{S}_{k}\equiv \mathcal{S}_{k}\mod O(k^{-\infty}).
\end{equation}
It follows that
\begin{equation} \label{s3-e20bis}
\mathcal{S}_{k}\equiv\mathcal{S}_{k}^*\mod O(k^{-\infty}),\ \mathcal{S}_{k}^2\equiv\mathcal{S}_{k}\mod O(k^{-\infty}).
\end{equation}

From \eqref{s3-e17-2bis}, \eqref{s3-e20-1}, \eqref{s3-e21-1}, \eqref{s3-e18}, \eqref{s3-e18adjoint}, \eqref{s3-e21-4}, \eqref{s3-e20} and \eqref{s3-e20bis}, we get our main technical result:
\begin{thm} \label{s3-t4}
In the situation of Setup \ref{local_data} 
let $q=n_-$ or $n_+$ and let
$\mathcal{S}_{k}$ be the localized approximate Szeg\"o kernel \eqref{s3-e17-1} and $\mathcal{A}_{k}$ the localized partial parametrix \eqref{s3-e17}. Then,
\begin{equation} \label{s3-e22-00bis}
\begin{split}
\mathcal{S}_{k}^*, \mathcal{S}_{k}=O(k^s): H^s_{{\rm comp\,}}(D, \Lambda^qT^{*(0,1)}M)\To H^{s}_{{\rm comp\,}}(D, \Lambda^qT^{*(0,1)}M),\ \ \forall s\in\mathbb N_0,\\
\mathcal{A}_{k}^*,\mathcal{A}_{k}=O(k^s): H^s_{{\rm comp\,}}(D, \Lambda^qT^{*(0,1)}M)\To H^{s+1}_{{\rm comp\,}}(D, \Lambda^qT^{*(0,1)}M),\ \ \forall s\in\mathbb N_0,
\end{split}
\end{equation}
and we have
\begin{gather}
\Box^{(q)}_s\mathcal{S}_{k}\equiv 0\mod
O(k^{-\infty}),\ \ \mathcal{S}_{k}^*\Box^{(q)}_s\equiv0\mod O(k^{-\infty}), \label{s3-e22-0a1} \\
\mathcal{S}_{k}\equiv\mathcal{S}_{k}^*\mod O(k^{-\infty}),\ \ \mathcal{S}_{k}\equiv\mathcal{S}_{k}^2\mod O(k^{-\infty}),\ \ \mathcal{S}_{k}\equiv\mathcal{S}_{k}^*\mathcal{S}_{k}\mod O(k^{-\infty}), \label{s3-e22-0a}\\
\mathcal{S}_{k}^*+\mathcal{A}_{k}^*\Box^{(q)}_s\equiv I\mod O(k^{-\infty}),\ \ \mathcal{S}_{k}+\Box^{(q)}_s\mathcal{A}_{k}\equiv I\mod O(k^{-\infty}), \label{s3-e23}
\end{gather}
where $\mathcal{S}_{k}^*$ and $\mathcal{A}_{k}^*$ are the formal adjoints of
$\mathcal{S}_{k}$ and $\mathcal{A}_{k}$ with respect to $(\,\cdot\,,\cdot)$
respectively and $\Box^{(q)}_s$ is given by \eqref{s1-e3} and
\eqref{s1-e4}.
\end{thm}
We notice that $\mathcal{S}_{k}$, $\mathcal{S}_{k}^*$, $\mathcal{A}_{k}$, $\mathcal{A}_{k}^*$, are all properly supported.
We need

\begin{thm} \label{s3-t5}
The localized approximate Szeg\"o kernel $\mathcal{S}_{k}$ given by \eqref{s3-e17-1} is a smoothing operator. Moreover, if
$q=n_-$, then the kernel of $\mathcal{S}_{k}$ satisfies
\begin{equation} \label{s3-es-1}
\mathcal{S}_{k}(z,w)\equiv e^{ik\Psi(z,w)}b(z,w,k)\mod
O(k^{-\infty}),
\end{equation}
with
\begin{equation}  \label{s3-es-2}
\begin{split}
&b(z,w,k)\in S^{n}_{{\rm loc\,}}\big(1;D\times D, \Lambda^qT^{*(0,1)}_wM\boxtimes\Lambda^qT^{*(0,1)}_zM\big), \\
&b(z,w,k)\sim\sum^\infty_{j=0}b_j(z, w)k^{n-j}\text{ in }S^{n}_{{\rm loc\,}}
\big(1;D\times D, \Lambda^qT^{*(0,1)}_wM\boxtimes\Lambda^qT^{*(0,1)}_zM\big), \\
&b_j(z, w)\in\cC^\infty\big(D\times D, \Lambda^qT^{*(0,1)}_wM\boxtimes\Lambda^qT^{*(0,1)}_zM\big),\ \ j=0,1,2,\ldots,
\end{split}
\end{equation}
and $\Psi(z,w)$ is as in Theorem~\ref{s3-t1-bis}.

If $q=n_+$, $n_-\neq n_+$, then
\begin{equation} \label{s3-es-2-0bis}
\mathcal{S}_{k}(z,w)\equiv0\mod O(k^{-\infty}).
\end{equation}
\end{thm}
\begin{proof}
Theorem~\ref{s3-t5} essentially follows from the stationary phase
formula of Melin-Sj\"{o}strand~\cite{MS74}.
Let $D$, $s$, $\phi$ be as in Setup \ref{local_data}.
Let $q=n_-$ or $n_+$. Let $(z,\theta)=x=(x',x_{2n+1})$ be the local coordinates as
in \eqref{s3-e0}, \eqref{s3-e1} defined on $D\times]-\varepsilon_0,
\varepsilon_0[$.
We identify $x'$ with
$(x',0)$. If we denote the holomorphic coordinates of $D$ by $w_j=y_{2j-1}+iy_{2j}$, $j=1,\ldots,n$, and by $\eta$ the coordinate of
$]-\varepsilon_0, \varepsilon_0[$, we also write 
\[
(w,\eta)=(y',\eta)=y\,,\:y'=(y_1,\ldots,y_{2n})\,.
\]
From the definition \eqref{s3-e17-1} of
$\mathcal{S}_{k}$ and Theorem~\ref{s3-t1}, we see
that the distribution kernel of $\mathcal{S}_{k}$ is given by
\begin{equation} \label{sp-e1}
\begin{split}
&\mathcal{S}_{k}(x',y')\equiv\int_{t\geq0}e^{i\varphi_-(x,y)t+i\theta k-i\eta   k}s_-(x,y,t)\chi_1(\theta)\chi(\eta   )d\theta dtd\eta   \\
&+\int_{t\geq0}e^{i\varphi_+(x,y)t+i\theta k-i\eta   k}s_+(x,y,t)\chi_1(\theta)\chi(\eta   )d\theta\,dt\,d\eta   \mod O(k^{-\infty})\\
&\equiv I_0(x',y')+I_1(x',y')\mod O(k^{-\infty}),
\end{split}
\end{equation}
where the integrals above are defined as oscillatory integrals.
First, we study the kernel
\[I_1(x',y')=\int_{t\geq0}e^{i\varphi_+(x,y)t+i\theta k-i\eta   k}s_+(x,y,t)\chi_1(\theta )\chi(\eta   )d\theta dtd\eta   .\]
By the change of variables $t=k\sigma$ we get
\[I_1(x',y')=
\int_{\sigma\geq0}e^{ik\bigr(\varphi_+(x,y)\sigma+\theta -\eta   \bigr)}ks_+(x,y,k\sigma)\chi_1(\theta  )\chi(\eta   )d\theta  d\sigma
d\eta   .\] Note that $d_x\varphi_+(x,x)=\omega_0(x)$. Taking into account  the
form $\omega_0(x)$ (cf.\ \eqref{s3-e1-1}), we see that
$\frac{\pr\varphi_+}{\pr \theta  }(x,x)=1$. In view of
Theorem~\ref{s3-t1}, we see that $\varphi_+(x,y)=0$ if and only if
$x=y$. We conclude that
\[\bigr(d_\sigma(\varphi_+(x,y)\sigma+\theta  -\eta   ),d_{\theta  }(\varphi_+(x,y)\sigma+\theta  -\eta   )\bigr)\neq\bigr(0,0\bigr),\ \ \sigma\geq0.\]
Thus, we can integrate by parts in $\sigma$ and $\theta  $ and
conclude that $I_1$ is smoothing and
\begin{equation} \label{sp-e2}
I_1\equiv0\mod O(k^{-\infty}).
\end{equation}
Now, we study the kernel
\[I_0(x',y')=\int_{t\geq0}e^{i\varphi_-(x,y)t+i\theta  k-i\eta   k}s_-(x,y,t)\chi_1(\theta  )\chi(\eta   )d\theta  dtd\eta   .\]
As before, by letting $t=k\sigma$, we get
\begin{equation} \label{sp-e3}
I_0(x',y')=\int_{\sigma\geq0}e^{ik\bigr(\varphi_-(x,y)\sigma+\theta  -\eta   \bigr)}ks_-(x,y,k\sigma)\chi_1(\theta  )\chi(\eta   )d\theta  d\sigma
d\eta   .
\end{equation}
In view of \eqref{s3-e16-bisbbbbb}, we see that
\begin{equation} \label{sp-e3-1}
\varphi_-(x,y)=\Psi(x',y')+\eta   -\theta  ,\ \ {\rm
Im\,}\Psi(x',y')\geq0.
\end{equation}
Put
\begin{equation} \label{sp-e3-2}
\Psi(x,y,\sigma)=(\Psi(x',y')+\eta   -\theta  )\sigma+\theta  -\eta   .
\end{equation}
Let $\varphi(\sigma)\in\cC^\infty_0(\Real_+)$ with
$\varphi(\sigma)=1$ in some small neighborhood of $1$. We introduce the cut-off functions $\varphi(\sigma)$ and $1-\varphi(\sigma)$ in the integral \eqref{sp-e3}:
\begin{align}
I^0_0(x',y'):=\int_{\sigma\geq0}e^{ik\Psi(x,y,\sigma)}\varphi(\sigma)ks_-(x,y,k\sigma)\chi_1(\theta  )\chi(\eta   )d\theta  d\sigma d\eta   , \label{s4-e4}\\
I^1_0(x',y'):=\int_{\sigma\geq0}e^{ik\Psi(x,y,\sigma)}(1-\varphi(\sigma))ks_-(x,y,k\sigma)\chi_1(\theta  )\chi(\eta   )d\theta  d\sigma
d\eta   \,, \label{sp-e5}
\end{align}
so that \[I_0(x',y')=I^0_0(x',y')+I^1_0(x',y')\,.\]
First, we study $I^1_0(x',y')$. Note that when $\sigma\neq1$,
$d_{\theta  }\Psi(x,y,\sigma)=1-\sigma\neq0$. Thus, we can integrate by
parts and get that $I^1_0$ is smoothing and $I^1_0(x',y')\equiv0\mod O(k^{-\infty})$.

Next, we
study the kernel $I^0_0(x',y')$. First, we assume that $q=n_+$, $n_+\neq n_-$. In view of Theorem~\ref{s3-t1}, we see that
$s_-(x,y,t)$ vanishes to infinite order at $x=y$. From this observation, it is straightforward to see that 
$I^0_0\equiv0\mod O(k^{-\infty})$. Therefore, we get \eqref{s3-es-2-0bis}.

Now, we assume that $q=n_-$.
Since the integral \eqref{s4-e4}
converges, we have
\begin{equation} \label{sp-e6}
\begin{split}
&I^0_0(x',y')=\int H(x',y)\chi(\eta   )d\eta   ,\\
&H(x',y)=\int_{\sigma\geq0}e^{ik\Psi(x,y,\sigma)}\varphi(\sigma)ks_-(x,y,k\sigma)\chi_1(\theta  )d\theta  d\sigma.
\end{split}
\end{equation}
Recalling the form of $\Psi(x,y,\sigma)$, we have ${\rm
Im\,}\Psi(x,y,\sigma)\geq0$, $d_\sigma\Psi(x,y,\sigma)=0$ if and
only if $x=y$ and $d_{\theta  }\Psi(x,y,\sigma)|_{x=y}=1-\sigma$.
Thus, $x=y$ and $\sigma=1$ are real critical points. Moreover, we can
check that the Hessian of $\Psi(x,y,\sigma)$ at $x=y$, $\sigma=1$,
is given by
\[\left(
\begin{array}[c]{cc}
  \Psi''_{\sigma\sigma}(x,x,1)& \Psi''_{\theta  \sigma}(x,x,1) \\
  \Psi''_{\sigma \theta  }(x,x,1) & \Psi''_{\theta  \theta  }(x,x,1)
\end{array}\right)=\left(
\begin{array}[c]{cc}
 0 & -1 \\
 -1 &0
\end{array}\right).\]
Thus, $\Psi(x,y,\sigma)$ is a non-degenerate complex valued phase function
in the sense of Melin-Sj\"{o}strand~\cite{MS74}. Let
\[\Td\Psi(\Td x,\Td y,\Td\sigma):=\bigr(\Td\Psi(\Td x', \Td y')+(\Td \eta   -\Td \theta  )\bigr)\Td\sigma+\Td \theta  -\Td  \eta   \]
be an almost analytic extension of $\Psi(x,y,\sigma)$, where
$\Td\Psi(\Td x',\Td y')$ is an almost analytic extension of $\Psi(x',y')$
(with $\Psi(x',y')$ as in \eqref{sp-e3-1}) and
similarly for $\Td \eta   $, $\Td \theta  $ and $\Td\sigma$ (see \cite[\S 2]{MS74} for the precise meaning of the almost analytic extension). We can
check that given $\eta   $ and $(x',y')$, $\Td \theta  =\eta   +\Psi(x', y')$, $\Td\sigma=1$ are the solutions of 
\[
\frac{\pr\Td\Psi}{\pr\Td\sigma}=0\,,\: \frac{\pr\Td\Psi}{\pr\Td \theta  }=0\,.
\]
From this and by the stationary phase formula of Melin-Sj\"{o}strand~\cite{MS74}, we
get
\begin{equation} \label{sp-e7}
H(x',y)\equiv e^{ik\Psi(x',y')}a(x',y,k)\mod O(k^{-\infty}),
\end{equation}
where $a(x', y, k)\in
\cC^\infty\bigr(D\times(D\times]-\varepsilon_0,\varepsilon_0[), \Lambda^qT^{*(0,1)}M\boxtimes\Lambda^qT^{*(0,1)}M\bigr)$,
\[a(x',y,k)\sim\sum^\infty_{j=0}k^{n-j}a_j(x',y) \ \
\mbox{in $S^n_{{\rm loc\,}}\bigr(1;D\times(D\times]-\varepsilon_0,\varepsilon_0[),
\Lambda^qT^{*(0,1)}M\boxtimes\Lambda^qT^{*(0,1)}M\bigr)$},\]
\[a_j(x',y)\in\cC^\infty\bigr(D\times(D\times]-\varepsilon_0,\varepsilon_0[), \Lambda^qT^{*(0,1)}M\boxtimes\Lambda^qT^{*(0,1)}M\bigr),\ \
j\in\N_0\,,\] and
\begin{equation} \label{sp-e7-1}
a_0(x', y)=2\pi\Td
s^0_-\bigr((x',\eta   +\Psi(x',y')),y\bigr),
\end{equation}
where $\Td s^0_-$ is an almost analytic extension of $s^0_-$,
$s^0_-$ is as in \eqref{s3-e11-bis}.
From \eqref{sp-e6} and \eqref{sp-e7} we get
\begin{equation} \label{sp-e8}
I^0_0(x',y')\equiv e^{ik\Psi(x',y')}b(x',y',k)\mod O(k^{-\infty}),
\end{equation}
where
\[b(x',y',k)\sim\sum^\infty_{j=0}k^{n-j}b_j(x',y') \ \
\mbox{in $S^n_{{\rm loc\,}}(1;D\times D, \Lambda^qT^{*(0,1)}M\boxtimes\Lambda^qT^{*(0,1)}M)$},\]
with
\begin{equation} \label{sp-e9}
b_j(x',y')=\int a_j(x',y)\chi(\eta   )d\eta   \in\cC^\infty(D\times
D, \Lambda^qT^{*(0,1)}M\boxtimes\Lambda^qT^{*(0,1)}M)\,,\:j\in\N_0\,.
\end{equation}
Theorem~\ref{s3-t5} follows.
\end{proof}
Let $D$, $s$, $\phi$ be as in Setup \ref{local_data}.
In view of Theorem~\ref{s3-t4} and \eqref{s3-es-2-0bis}, we see that when $q=n_+$, $n_+\neq n_-$, we have
\begin{equation} \label{eXXXXV}
\Box^{(q)}_s\mathcal{A}_{k}\equiv I\mod O(k^{-\infty}),
\end{equation}
where $\mathcal{A}_{k}$ is as in Theorem~\ref{s3-t4}.

Now, we assume that $q\neq n_-, n_+$. Using Theorem~\ref{s3-t0} and repeating the proof of Theorem~\ref{s3-t4} we conclude that there exists a properly supported continuous operator
\[
\mathcal{A}_{k}=O(k^{s}):H^s_{{\rm comp\,}}(D,\Lambda^qT^{*(0,1)}M)\To H^{s+1}_{{\rm comp\,}}(D,\Lambda^qT^{*(0,1)}M),\ \ \forall s\in\mathbb N_0,\]
such that
\begin{equation} \label{eXXXXVI}
\Box^{(q)}_s\mathcal{A}_{k}\equiv I\mod O(k^{-\infty}).
\end{equation}
Summing up, we obtain

\begin{thm} \label{s3-t4-1}
In the situation of Setup \ref{local_data}
let $q\neq n_-$. Then,
there exists a properly supported continuous operator
\[
\mathcal{A}_{k}=O(k^{s}):H^s_{{\rm comp\,}}(D,\Lambda^qT^{*(0,1)}M)\To H^{s+1}_{{\rm comp\,}}(D,\Lambda^qT^{*(0,1)}M),\ \ \forall s\in\mathbb N_0,\]
such that
\[\Box^{(q)}_s\mathcal{A}_{k}\equiv I\mod O(k^{-\infty}).\]
\end{thm}

\begin{rem} \label{rabis1}
From Remark~\ref{s3-rabis}, we can generalize Theorem~\ref{s3-t4} and Theorem~\ref{s3-t4-1} with essentially the same proofs to the case when the forms take values in $L^k\otimes E$, for a given holomorphic vector bundle $E$ over $M$.
\end{rem}

We have the following

\begin{thm} \label{s3-t6}
In the situation of Setup \ref{local_data}
let $q=n_-$. For a given point $p\in D$, let $V_1,\ldots,V_{n}$
be an orthonormal frame of $T^{(1,0)}M$ in a neighborhood of $p$, for which $\dot{R}^L$ is diagonalized at $p$, namely,
\begin{equation*}
\begin{split}
&\dot{R}^L(p)V_j(p)=\lambda_j(p)V_j(p)\,,\quad j=1,\ldots,n\,,\\
&\lambda_j(p)<0\,,\quad j=1,\ldots,q\,, \\
&\lambda_j(p)>0\,,\quad j=q+1,\ldots,n.
\end{split}
\end{equation*}
Let $(T_j)_{j=1}^{n}$
denote the basis of $T^{*(0,1)}M$, which is dual to $(\ol V_j)_{j=1}^{n}$. Then,
\begin{equation} \label{s3-es-6}
\begin{split}
b_0(p,p)&=(2\pi)^{-n}\abs{{\rm det\,}\dot{R}^L(p)}\prod^{q}_{j=1} (T_j(p)\wedge)\circ(T_j(p)\wedge)^*\\
&=(2\pi)^{-n}\big|\det\dot{R}^L(p)\big|I_{\det\ov{W}^{\,*}}\,,
\end{split}
\end{equation}
where $I_{\det\ov{W}^{\,*}}\in\End(\Lambda^qT^{*(0,1)}M)$ is as in the discussion after \eqref{s0-e2}.
\end{thm}
\begin{proof}
We use the same notations as in the proof of Theorem~\ref{s3-t5}.
From \eqref{sp-e7-1} and
\eqref{sp-e9}, we have
\begin{equation} \label{sp-e10}
b_0(x', x')=2\pi\int
s^0_-\bigr((x',\eta   ),(x',\eta   )\bigr)\chi(\eta   )d\eta   .
\end{equation}
In view of Theorem~\ref{s3-t2}, we know that
\begin{equation} \label{sp-e101}
s^0_-\bigr((x',\eta   ),(x',\eta)\bigr)=\frac{1}{2}\abs{\mu_1(x')}\cdots
\abs{\mu_{n}(x')}\pi^{-n-1}\prod_{j=1}^{n_-}(T_j(x')\wedge)\circ
(T_j(x')\wedge)^*,
\end{equation} 
where $\{\mu_j(x')\}_{j=1}^n$ are the
eigenvalues of $\mathcal{L}_{x'}$ 
and $\{T_j(x')\}_{j=1}^n$ are as in Theorem~\ref{s3-t2}. Here we identify $x'\in D$ with $(x',0)\in X$.
Notice that 
\begin{equation} \label{sp-e102}
\abs{\mu_1(p)}\cdots\abs{\mu_n(p)}=2^{-n}\abs{\lambda_1(p)}\cdots\abs{\lambda_n(p)}=2^{-n}\abs{\det\dot R^L(p)}.
\end{equation}
Now, \eqref{s3-e2}, \eqref{sp-e10}, \eqref{sp-e101} and \eqref{sp-e102} 
yield \eqref{s3-es-6}.
\end{proof} 
\section{Asymptotic expansion of the spectral function for lower energy forms}\label{s:exp_spec_fct}
Let $(M,\Theta)$ be a Hermitian manifold and let $(L,h^L)$ be a Hermitian holomorphic line bundle on $M$.
We recall that (cf.\ \eqref{s1-specsp}) $\cE^q_{ k^{-\!N_0}}(M, L^k)$ denote the spectral space of $\Box^{(q)}_k$ corresponding to energy less than $k^{-N_0}$.
In the present Section we study the asymptotic expansion of the spectral function associated to $\cE^q_{ k^{-\!N_0}}(M, L^k)$.
In Section \ref{ss:upper_bound} we prove pointwise upper bounds for the eigenforms of the spectral spaces $\cE^q_{ k^{-\!N_0}}(M, L^k)$ in terms of their $L^2$-norm (Theorem \ref{tI}).
In Section \ref{s:kern_spec} we compare the localized spectral projection with the localized approximate Szeg\"o projection  $\mathcal{S}_{k}$.  In  Section\ref{s:kern_spec_asy} we apply this results to prove the asymptotic expansion of the spectral function and thus give the proof of Theorem \ref{s1-main1}.
In Section \ref{s:asymp_BK} we exhibit the asymptotic expansion of the Bergman kernel and prove Theorem \ref{s1-main2}. Finally, in Section \ref{s:coeff}  we calculate the coefficients $b^{0}_1$ and $b^{0}_2$ and thus prove Theorem \ref{s1-main12}.

\subsection{Asymptotic upper bounds}\label{ss:upper_bound}

Fix $N_0\geq 1$. In this Section we will give pointwise upper bounds for $u$ and $\pr^\alpha u$, where $u\in\cE^q_{ k^{-\!N_0}}(M, L^k)$.

Let $D\Subset M$ be a chart domain such that $L|_D$ is trivial. Let $s$ be a local frame of $L$ on $D$ and set $\abs{s}_{h^L}^2=e^{-2\phi}$.
Let $(\,,)_{k\phi}$ be the inner product on the space
$\Omega^{0,q}_0(D)$ defined as follows:
\[(f,g)_{k\phi}=\int_D\!\langle\,f\,,g\,\rangle e^{-2k\phi}dv_M(x)\,,\quad f, g\in\Omega^{0,q}_0(D)\,.\]
 Let
$\ddbar^{*,k\phi}:\Omega^{0,q+1}(D)\To\Omega^{0,q}(D)$
be the formal adjoint of $\ddbar$ with respect to $(\,,)_{k\phi}$. Put
$\Box^{(q)}_{k\phi}=\ddbar\,\ddbar^{*,k\phi}+\ddbar^{*,k\phi}\ddbar:\Omega^{0,q}(D)\To\Omega^{0,q}(D)$.
Let $u\in\Omega^{0,q}(D, L^k)$.
On $D$, we write $u=s^k\Td u$, $\Td u\in\Omega^{0,q}(D)$. We have
\begin{equation} \label{eI}
\Box^{(q)}_ku=s^k\Box^{(q)}_{k\phi}\Td u.
\end{equation} 
Fix $p\in D$ and consider local coordinates $(D,z)\cong (D,x)$, 
such that $x(p)=z(p)=0$ and $\phi(z)=O(\abs{z}^2)$ near $p$. Let
$F_k(z):=\frac{z}{\sqrt{k}}$ be the scaling map. For $r>0$, let $D_r=\set{x;\, \abs{x_j}<r,j=1,\ldots,2n}$. Let $f\in\Omega^{0,q}(D_{\frac{\log k}{\sqrt{k}}})$, $f=\sum'_{\abs{J}=q}f_Jd\ol z^J$.
We define the scaled form $F^*_kf\in\Omega^{0,q}(D_{\log k})$ by
\[F^*_kf=\sideset{}{'}\sum_{\abs{J}=q}f_J\big(k^{-1/2}z\big)d\ol z^J\in\Omega^{0,q}(D_{\log k}).\]
Let $\Box^{(q)}_{k\phi,(k)}:\Omega^{0,q}(D_{\log k})\To\Omega^{0,q}(D_{\log k})$ be the scaled Laplacian defined by
\begin{equation} \label{eI.I}
\Box^{(q)}_{k\phi,(k)}(F^*_ku)=\frac{1}{k}F^*_k(\Box^{(q)}_{k\phi}u),\ \ u\in\Omega^{0,q}(D_{\frac{\log k}{\sqrt{k}}}).
\end{equation}
By Berman~\cite[\S\,2]{Be04} and Hsiao-Marinescu~\cite[\S\,2]{HM09} it is known all the derivatives of the coefficients of the operator
$\Box^{(q)}_{k\phi,(k)}$ are uniformly bounded in $k$ on $D_{\log k}$. Let $D_r\subset D_{\log k}$ and let $W^s_{kF^*_k\phi}(D_r,\Lambda^{q}T^{*(0,1)}M)$, $s\in\N_0$, denote the Sobolev space of order $s$ of sections of $\Lambda^{q}T^{*(0,1)}M$
over $D_r$ with respect to the weight $e^{-2kF^*_k\phi}$. The Sobolev norm on this space is given by
\[\norm{u}^2_{kF^*_k\phi,s,D_r}=\sideset{}{'}\sum_{\alpha\in\mathbb N^{2n}_0,\abs{\alpha}\leq s, \abs{J}=q}\int_{D_r}\abs{\pr^\alpha_xu_J}^2e^{-2kF^*_k\phi}(F^*m)(x)dx,\]
where $u=\sum'_{\abs{J}=q}u_Jd\ol z^J\in W^s_{kF^*_k\phi}(D_r,\Lambda^qT^{*(0,1)}M)$ and $m(x)dx$ is the volume form.
If $s=0$, we write $\norm{\cdot}_{kF^*_k\phi,D_r}$ to denote $\norm{\cdot}_{kF^*_k\phi,0,D_r}$.

\begin{lem} \label{lI}
For every $r>0$ with $D_{2r}\subset D_{\log k}$ and $s\in\N_0$, there is a constant $C_{r,s}>0$
independent of $k$, such that
\begin{equation} \label{eII}
\norm{u}^2_{kF^*_k\phi,2s,D_{r}}\leqslant  C_{r,s}\Bigr(\norm{u}^2_{kF^*_k\phi,D_{2r}}+\sum^s_{m=1}\big\|(\Box^{(q)}_{k\phi,(k)})^mu\big\|^2_{kF^*_k\phi,D_{2r}}\Bigl)\,,\;
u\in\Omega^{0,q}(D_{\log k})\,.
\end{equation}
\end{lem}

\begin{proof}
Since $\Box^{(q)}_{k\phi,(k)}$ is elliptic, we conclude from G{\aa}rding's inequality that for every $r>0$,
$D_{2r}\subset D_{\log k}$ and $s\in\N_0$, we have
\begin{equation} \label{eIII}
\norm{u}^2_{kF^*_k\phi,s+2,D_{r}}\leqslant\Td C_{r',s}\Bigr(\norm{u}^2_{kF^*_k\phi,D_{r'}}+\big\|\Box^{(q)}_{k\phi,(k)}u\big\|^2_{kF^*_k\phi,s,D_{r'}}\Bigl)\,,\;
u\in\Omega^{0,q}(D_{\log k})\,,
\end{equation}
for some $r'>r$. Since all the derivatives of the coefficients of the operator $\Box^{(q)}_{k\phi,(k)}$ are uniformly bounded in $k$,
it is straightforward to see that $\Td C_{r',s}$ can be taken to be independent of $k$. (See Proposition 2.4 and Remark 2.5 in Hsiao-Marinescu~\cite{HM09}.) From \eqref{eIII} and using induction, we get \eqref{eII}.
\end{proof} 

\begin{lem} \label{lII}
For $k$ large and for every $\alpha\in\mathbb N^{2n}_0$, there is a constant $C_{\alpha}>0$ independent of $k$, such that
\begin{equation} \label{eIV}
\abs{(\pr^\alpha_xu)(0)}\leq C_\alpha,
\end{equation}
where $u\in\Omega^{0,q}(D_{\log k})$, $\norm{u}_{kF^*_k\phi,D_{\log k}}\leq 1$, $\norm{(\Box^{(q)}_{k\phi,(k)})^mu}_{kF^*_k\phi,D_{\log k}}\leq k^{-m}$, $\forall m\in\N_0$\,.
\end{lem}

\begin{proof}
Let $u\in\Omega^{0,q}(D_{\log k})$, $\norm{u}_{kF^*_k\phi,D_{\log k}}\leq 1$, $\norm{(\Box^{(q)}_{k\phi,(k)})^mu}_{kF^*_k\phi,D_{\log k}}\leq k^{-m}$, $\forall m\in\N_0$.
By using Fourier transform, it is easy to see that (cf.\ Lemma 2.6 in~\cite{HM09})
\begin{equation} \label{eV}
\abs{(\pr^\alpha_xu)(0)}\leq C\norm{u}_{kF^*_k\phi,n+1+\abs{\alpha},D_r},
\end{equation}
for some $r>0$, where $C>0$ only depends on the dimension and the length of $\alpha$. From \eqref{eII}, we see that
\begin{equation} \label{eVI}
\begin{split}
\norm{u}^2_{kF^*_k\phi,n+1+\abs{\alpha},D_r}&\leq C_{r,\alpha}\Bigr(\norm{u}^2_{kF^*_k\phi,D_{2r}}+\sum^N_{m=1}\norm{(\Box^{(q)}_{k\phi,(k)})^mu}_{kF^*_k\phi,D_{2r}}\Bigr),\ \ 2N\geq n+1+\abs{\alpha},\\
&\leq C_{r,\alpha}\Bigr(1+\sum^\infty_{m=1}k^{-m}\Bigr)\leq\Td C_\alpha
\end{split}
\end{equation}
if $k$ large, where $\Td C_{\alpha}>0$ is independent of $k$. Combining \eqref{eV} with \eqref{eVI}, \eqref{eIV} follows.
\end{proof}

Now, we can prove

\begin{thm} \label{tI}
For $k$ large and for every $\alpha\in\mathbb N^{2n}_0$, $D'\Subset D$, there is a constant $C_{\alpha,D'}>0$ independent of $k$, such that
\begin{equation} \label{eVI-I}
\abs{(\pr^\alpha_x(\Td ue^{-k\phi}))(x)}\leq C_{\alpha,D'}k^{\frac{n}{2}+\abs{\alpha}}\norm{u},\ \ \forall x\in D',
\end{equation}
where $u\in\cE^q_{k^{-\!N_0}}(M,L^k)$, $N_0\geq1$, $u|_D=s^k\Td u$, $\Td u\in\Omega^{0,q}(D)$.
\end{thm}

\begin{rem} \label{rI}
Let $s_1$ be another local frame of $L$ on $D$, $\abs{s_1}^2=e^{-2\phi_1}$. We have $s_1=gs$ for some holomorphic function $g\in\cC^\infty(D)$, $g\neq0$ on $D$. Let $u\in\Omega^{0,q}(D,L^k)$.
On $D$, we write $u=s^k\Td u=s^k_1\Td v$. Then, we can check that
\begin{equation} \label{e*}
\Td ve^{-k\phi_1}=\Td u({\ol g}^{\,1/2}g^{-1/2})^ke^{-k\phi}.
\end{equation}
From \eqref{e*}, it is easy to see that if $\Td u$ satisfies \eqref{eVI-I}, then $\Td v$ also satisfies \eqref{eVI-I}.
Thus, the conclusion of Theorem~\ref{tI} makes sense.
\end{rem}

\begin{proof} [Proof of Theorem~\ref{tI}]
We may assume that $0\in D'$.
Let $u\in\cE^q_{k^{-\!N_0}}(M,L^k)$, $N_0\geq1$, $u|_D=s^k\Td u$, $\Td u\in\Omega^{0,q}(D)$. We may assume that
$D_{\frac{\log k}{\sqrt{k}}}\subset D$ and consider $\Td u|_{D_{\frac{\log k}{\sqrt{k}}}}$.
Set $\beta_k:=k^{-\frac{n}{2}}F^*_k\Td u=k^{-\frac{n}{2}}\Td u(\frac{x}{\sqrt{k}})\in\Omega^{0,q}(D_{\log k})$. We can check that
\begin{equation} \label{eVII}
\norm{\beta_k}_{kF^*_k\phi,D_{\log k}}\leq\norm{u}.
\end{equation}
Since $u\in\cE^q_{k^{-\!N_0}}(M,L^k)$, we have $\norm{(\Box^{(q)}_k)^mu}\leq k^{-mN_0}\norm{u}$ for all $m\in\N$. From this observation
and \eqref{eI.I}, we have
\begin{equation} \label{eVIII}
\begin{split}
\norm{(\Box^{(q)}_{k\phi,(k)})^m\beta_k}_{kF^*_k\phi,D_{\log k}}&=\frac{1}{k^{m+\frac{n}{2}}}\norm{F^*_k\bigr((\Box^{(q)}_{k\phi})^m\Td u\bigr)}_{kF^*_k\phi,D_{\log k}}\\
&\leq\frac{1}{k^m}\norm{(\Box^{(q)}_k)^mu}\leq k^{-mN_0-m}\norm{u}.
\end{split}
\end{equation}
From \eqref{eVII}, \eqref{eVIII} and Lemma~\ref{lII}, we conclude that for every $\alpha\in\mathbb N^{2n}_0$, there is a constant $\Td C_{\alpha}>0$ independent of $k$, such that
\[\abs{k^{-\frac{n}{2}-\frac{\abs{\alpha}}{2}}(\pr^\alpha_x\Td u)(0)}=\abs{(\pr^\alpha_x\beta_k)(0)}\leq\Td C_{\alpha}\norm{u}.\]
Thus, for every $\alpha\in\mathbb N^{2n}_0$, there is a constant $C_{\alpha}>0$ independent of $k$, such that
\[\abs{(\pr^\alpha_x(\Td ue^{-k\phi}))(0)}\leq C_\alpha k^{\frac{n}{2}+\abs{\alpha}}\norm{u}.\]

Let $x_0$ be another point of $D'$. We can repeat the procedure above and conclude that for every $\alpha\in\mathbb N^{2n}_0$, there is a $C_\alpha(x_0)>0$ independent of $k$, such that
\[\abs{(\pr^\alpha_x(\Td ue^{-k\phi}))(x_0)}\leq C_\alpha(x_0) k^{\frac{n}{2}+\abs{\alpha}}\norm{u}.\]
It is straightforward to see that the constant $C_\alpha(x_0)$ depends continuously on $\phi$ and the coefficients of $\Box^{(q)}_{k\phi,(k)}$ in $\cC^m(D)$ topology, for some $m\in\N_0$. (See Remark 2.5 and Theorem 2.7 in~\cite{HM09}, for the details.) Since $\ol D'\subset D$ is compact, $C_\alpha(x_0)$ can be taken to be independent of the point $x_0$. The theorem follows.
\end{proof} 

\subsection{Kernel of the spectral function}\label{s:kern_spec}
As in \eqref{e:orto_proj}, let 
\[P^{(q)}_{k,k^{-N_0}}:L^2_{(0,q)}(M,L^k)\to\cE^q_{k^{-N_0}}(M, L^k)\]
be the spectral projection on the spectral space of $\Box^{(q)}_k$ corresponding to energy less than $k^{-N_0}$.
The goal of this Section is to compare the localized spectral projection $\widehat P^{(q)}_{k,k^{-N_0},s}$ (see \eqref{eXI}) to the localized approximate Szeg\"o projection  $\mathcal{S}_{k}$ defined in \eqref{s3-e17-1}. This will be achieved in Proposition \ref{pIII}.

We introduce some notations. Let $(e_1,\ldots,e_n)$ be a smooth local orthonormal
frame of $T^{*(0,1)}_xM$ over an open set $D\Subset M$. Then
$(e^J:=e_{j_1}\wedge\cdots\wedge e_{j_q})_{1\leqslant
j_1<j_2<\cdots<j_q\leqslant  n}$ is an orthonormal frame of
$\Lambda^qT^{*(0,1)}_xM$ over $D$. For $f\in\Omega^{0,q}(D)$, we may write
$f=\sum'_{\abs{J}=q} f_Je^J$, with $f_J=\langle\,f\,,e^J\,\rangle\in
\cC^\infty(D)$. We call $f_J$ the
component of $f$ along $e^J$. Let
$A:\Omega^{0,q}_0(D)\To\Omega^{0,q}(D)$ be a continuous operator with smooth kernel. We write
\begin{equation} \label{s1-e13}
A(x,y)=\sideset{}{'}\sum_{\abs{I}=q,\abs{J}=q}e^I(x)A_{I,J}(x,y)e^J(y)\,,\:\:A_{I,J}\in\cC^\infty(D\times D)\,.
\end{equation}
We have
\begin{equation} \label{s1-e14}
(Au)(x)=\sideset{}{'}\sum_{\abs{I}=q,\abs{J}=q}e^I(x)\int_DA_{I,J}(x,y)u_J(y)dv_M(y)\,,\:\:u=\sideset{}{'}\sum_{\abs{J}=q}u_Je^J\in\Omega^{0,q}_0(D)\,.
\end{equation}
Let $A^*$ be the formal adjoint of $A$ with respect to $(\,\cdot\,,\cdot)$. 
We can check that
\begin{equation} \label{s5-e2-0}
A^*(x,y)=\sideset{}{'}\sum_{\abs{I}=q,\abs{J}=q}e^I(x)A^*_{I,J}(x,y)e^J(y)\,,\:\:A^*_{I,J}(x,y)=\ol{A_{J,I}(y,x)},
\end{equation}
Let
\begin{equation*}
\begin{split}
&B:\Omega^{0,q}(D)\To\Omega^{0,q}(D),\ \ \Omega^{0,q}_0(D)\To\Omega^{0,q}_0(D),\\
&B(x,y)=\sideset{}{'}\sum_{\abs{I}=q,\abs{J}=q}e^I(x)B_{I,J}(x,y)e^J(y),
\end{split}
\end{equation*}
be a properly supported smoothing operator. We write \[(B\circ
A)(x,y)=\sideset{}{'}\sum_{\abs{I}=q,\abs{J}=q}e^I(x)(B\circ
A)_{I,J}(x,y)e^J(y)\] in the sense of \eqref{s1-e14}. It is not
difficult to see that
\begin{equation} \label{s5-e2-1}
(B\circ
A)_{I,J}(x,y)=\sideset{}{'}\sum_{\abs{K}=q}\int_DB_{I,K}(x,z)A_{K,J}(z,y)dv_M(z)\,.
\end{equation}

Now, we return to our situation. Let \[P^{(q)}_{k,\lambda}(x,y)\in C^\infty\bigr(M\times M,(\Lambda^qT^{*(0,1)}_yM\otimes L^k_y)\boxtimes(\Lambda^qT^{*(0,1)}_xM\otimes L^k_x)\bigr)\]
be the spectral function, i.\,e., the Schwartz kernel of $P^{(q)}_{k,\lambda}$:
\begin{equation}\label{specfdef}
(P^{(q)}_{k,\lambda}u)(x)=\int_MP^{(q)}_{k,\lambda}(x,y)u(y)dv_M(y),\ \ u\in L^2_{(0,q)}(M, L^k).
\end{equation}
Let $s$ be a local frame of $L$ over $D$, where $D\subset M$. Then on $D\times D$ we can write
\[P^{(q)}_{k,\lambda}(x,y)=s(x)^kP^{(q)}_{k,\lambda,s}(x, y)s^*(y)^k,\]
where $P^{(q)}_{k,\lambda,s}(x, y)$
is smooth on $D\times D$,
so that for $x\in D$, $u\in\Omega^{0,q}_0(D,L^k)$,
\begin{equation} \label{s1-e2}
\begin{split}
(P^{(q)}_{k,\lambda}u)(x)&=s(x)^k\int_MP^{(q)}_{k,\lambda,s}(x, y)\langle\,u(y)\,, s^*(y)^k\rangle\,dv_M(y)\\
&=s(x)^k\int_M P^{(q)}_{k,\lambda,s}(x, y)\Td u(y)dv_M(y),\ \ u=s^k\Td u,\ \ \Td u\in\Omega^{0,q}_0(D).
\end{split}
\end{equation}
For $x=y$, we can check that the function
$P^{(q)}_{k,\lambda,s}(x,x)\in\cC^\infty(D, \End(\Lambda^qT^{*(0,1)}M))$
is independent of the choices of local frame $s$. 

Let $D\Subset M$ be an open set. Assume that $L|_D$ is trivial and let $s$ be a local frame of $L$ on $D$ and set $\abs{s}_{h^L}^2=e^{-2\phi}$. 
Let $(D,z)\cong(D,x)$ be local coordinates of $D$. Fix $N_0\geq1$. We define the \emph{localized spectral projection} (with respect to the trivializing section $s$) by
\begin{align} \label{eXI}
\widehat P^{(q)}_{k,k^{-N_0},s}:L^2_{(0,q)}(D)\cap\mathscr E'(D,\Lambda^qT^{*(0,1)}M)&\To\Omega^{0,q}(D),\nonumber \\
u&\To e^{-k\phi}s^{-k}P^{(q)}_{k, k^{-N_0}}(s^ke^{k\phi}u).
\end{align}
That is, if $P^{(q)}_{k, k^{-N_0}}(s^ke^{k\phi}u)=s^kv$ on $D$, then
$\widehat P^{(q)}_{k,k^{-N_0},s}u=e^{-k\phi}v$. We notice that
\begin{equation} \label{s1-e11}
\widehat P^{(q)}_{k,k^{-N_0},s}(x,y)=e^{-k\phi(x)}P^{(q)}_{k,k^{-N_0},s}(x,y)e^{k\phi(y)},
\end{equation}
where $\widehat P^{(q)}_{k,k^{-N_0},s}(x,y)$ is the kernel of
$\widehat P^{(q)}_{k,k^{-N_0},s}$ with respect to $(\,\cdot\,,\cdot)$ and $P^{(q)}_{k, k^{-N_0},s}(x,y)$ is as in
\eqref{s1-e2}. We write
\begin{equation} \label{s5-e3}
\widehat P^{(q)}_{k,k^{-N_0},s}(x,y)=\sideset{}{'}\sum_{\abs{I}=q,\abs{J}=q}e^I(x)\widehat P^{(q)}_{k,k^{-N_0},s,I,J}(x,y)e^J(y)\,,\:\;\widehat P^{(q)}_{k,k^{-N_0},s,I,J}\in\cC^\infty(D\times D),
\end{equation}
in the sense of \eqref{s1-e14}. 
Since $P^{(q)}_{k,k^{-N_0}}$ is self-adjoint, we have
\begin{equation} \label{s5-e4}
\widehat P^{(q)}_{k,k^{-N_0},s,I,J}(x,y)=\ol{\widehat P^{(q)}_{k,k^{-N_0},s,J,I}(y,x)},
\end{equation}
for all strictly increasing $I,J$ with $\abs{I}=\abs{J}=q$. 

\par Let $\{f_j\}_{j=1}^{d_k}\subset\Omega^{0,q}(M,L^k)$ be
an orthonormal frame for $\cE^q_{k^{-\!N_0}}(M, L^k)$, $d_k\in\N_0\bigcup\set{\infty}$. For each $j$,
we write 
\[
f_j|_D=\sideset{}{'}\sum_{\abs{J}=q}f_{j,J}(x)e^J(x)\,,\:\:f_{j,J}\in\cC^\infty(D,L^k)\,.
\]   
For $j=1,\ldots,d_k$ and strictly increasing $J$ with $\abs{J}=q$ we define $\Td f_{j,J}\in\cC^\infty(D)$ and $\Td f_j\in\Omega^{0,q}(D)$ by
\begin{equation*}
\begin{split}
f_{j,J}=s^k\Td f_{j,J},\ \ 
\Td f_j=\sideset{}{'}\sum_{\abs{J}=q}\Td f_{j,J}(x)e^J(x)\ .
\end{split}
\end{equation*}
Then, $f_j|_D=s^k\Td f_j$, $j=1,\ldots,d_k$, and it is not difficult to see that
\begin{equation} \label{s5-e5}
\widehat P^{(q)}_{k,k^{-N_0},s,I,J}(x,y)=\sum^{d_k}_{j=1}\Td f_{j,I}(x)\ol{\Td f_{j,J}(y)}\,e^{-k(\phi(x)+\phi(y))},
\end{equation}
for all strictly increasing $I,J$ with $\abs{I}=\abs{J}=q$. Since $\widehat P^{(q)}_{k,k^{-N_0},s,I,J}$ are smooth for all strictly increasing $I$, $J$, $\abs{I}=\abs{J}=q$, we conclude that for all $\alpha\in\N_0^{2n}$, 
\begin{equation} \label{s0-con1}
\mbox{$\sum^{d_k}_{j=1}\abs{(\pr^\alpha_x(\Td f_je^{-k\phi}))(x)}^2$ converges at each point of $x\in D$}.
\end{equation} 
Similarly, if $F:\mathscr E'(D,\Lambda^qT^{*(0,1)}M)\To \mathscr E'(D,\Lambda^qT^{*(0,1)}M)$ is a properly supported continuous operator such that for all $s\in\mathbb N_0$, $F:H^s_{{\rm comp\,}}(D,\Lambda^qT^{*(0,1)}M)\To H^{s+s_0}_{{\rm comp\,}}(D,\Lambda^qT^{*(0,1)}M)$ is continuous, for some $s_0\in\Real$. Then, we can check that 
\begin{equation} \label{s0-con2}
\mbox{$\sum^{d_k}_{j=1}\abs{(F(\Td f_je^{-k\phi}))(x)}^2$ converges at each point of $x\in D$}.
\end{equation} 
\begin{prop} \label{pI}
With the notations used above, for every $\alpha\in\mathbb N^{2n}_0$, $D'\Subset D$, there is a constant 
$C_{\alpha,D'}>0$ independent of $k$, such that
\begin{equation} \label{eIX}
\sum^{d_k}_{j=1}\abs{(\pr^\alpha_x(\Td f_je^{-k\phi}))(x)}^2\leq C_{\alpha,D'}k^{n+2\abs{\alpha}},\ \ \forall x\in D'.
\end{equation}
\end{prop}

\begin{proof}
Fix $\alpha\in\mathbb N^{2n}_0$ and $p\in D'$. We may assume that $\sum^{d_k}_{j=1}\abs{(\pr^\alpha_x(\Td f_je^{-k\phi}))(p)}^2\neq0$. Set
\[u(x)=\frac{1}{\sqrt{\sum^{d_k}_{j=1}\abs{(\pr^\alpha_x(\Td f_je^{-k\phi}))(p)}^2}}
\sum^{d_k}_{j=1}f_j(x)\ol{(\pr^\alpha_x(\Td f_je^{-k\phi}))(p)}.\] 
Since $\sum^{d_k}_{j=1}\abs{(\pr^\alpha_x(\Td f_je^{-k\phi}))(p)}^2$ converges, we
can check that $u\in\cE^q_{k^{-N_0}}(M,L^k)$, $\norm{u}=1$. On $D$, we write $u=s^k\Td u$, $\Td u\in\Omega^{0,q}(D)$. We can check that
\begin{equation} \label{eX}
\Td u=\frac{1}{\sqrt{\sum^{d_k}_{j=1}\abs{(\pr^\alpha_x(\Td f_je^{-k\phi}))(p)}^2}}
\sum^{d_k}_{j=1}\Td f_j(x)\ol{(\pr^\alpha_x(\Td f_je^{-k\phi}))(p)}.
\end{equation}
In view of Theorem~\ref{tI}, we see that $\abs{(\pr^\alpha_x(\Td ue^{-k\phi}))(p)}\leq C_\alpha k^{\frac{n}{2}+\abs{\alpha}}$, with 
$C_\alpha>0$ independent of $k$ and of the point $p$. From \eqref{eX}, it is straightforward to see that
\[\abs{(\pr^\alpha_x(\Td ue^{-k\phi}))(p)}=\sqrt{\sum^{d_k}_{j=1}\abs{(\pr^\alpha_x(\Td f_je^{-k\phi}))(p)}^2}\leq C_\alpha k^{\frac{n}{2}+\abs{\alpha}}.\]
The proposition follows.
\end{proof} 

Now, we assume that $\pr\ddbar\phi$ is non-degenerate of constant signature $(n_-,n_+)$ at each point of $D$ and let $q=n_-$.
Let $\mathcal{S}_k$, $\mathcal{A}_k$ be as in Theorem~\ref{s3-t4} and let $\Box^{(q)}_s$ be as in \eqref{s1-e3}, \eqref{s1-e4}.
If we replace $\mathcal{S}_{k}$ by $I-\Box^{(q)}_s\mathcal{A}_k$, then $\Box^{(q)}_s\mathcal{A}_k+\mathcal{S}_k=I=\mathcal{A}^*_k\Box^{(q)}_s+\mathcal{S}^*_k$ on
$\mathscr D'(D, \Lambda^{q}T^{*(0,1)}M)$. Now,
\begin{equation} \label{eXII}
\widehat P^{(q)}_{k,k^{-N_0},s}=(\mathcal{A}^*_k\Box^{(q)}_s+\mathcal{S}^*_k)\widehat P^{(q)}_{k,k^{-N_0},s}=R+\mathcal{S}^*_k\widehat P^{(q)}_{k,k^{-N_0},s}\ \ \mbox{on $\mathscr E'(D,\Lambda^qT^{*(0,1)}M)$},
\end{equation}
where we denote
\[R=\mathcal{A}^*_k\Box^{(q)}_s\widehat P^{(q)}_{k,k^{-N_0},s}.\]
We write \[R(x,y)=\sideset{}{'}\sum_{\abs{I}=q,\abs{J}=q}e^I(x)R_{I,J}(x,y)e^J(y)\,,\:\:R_{I,J}\in\cC^\infty(D\times D)\,,\]
in the sense of \eqref{s1-e14}. From \eqref{s5-e5}, it is straightforward to see that
\begin{equation} \label{eXIII}
\begin{split}
&R_{I,J}(x,y)=\sum^{d_k}_{j=1}\Td g_{j,I}(x)\ol{\Td f_{j,J}(y)}e^{-k\phi(y)},\\
&\Td g_j=\mathcal{A}^*_k\Box^{(q)}_s(\Td f_je^{-k\phi})(x),\ \ \Td g_j(x)=\sideset{}{'}\sum_{\abs{I}=q}\Td g_{j,I}(x)e^I(x),\ \ j=1,\ldots,d_k,
\end{split}
\end{equation}
for all strictly increasing $I$, $J$, $\abs{I}=\abs{J}=q$. From \eqref{s0-con2}, we see that for all $\alpha\in\mathbb N^{2n}_0$, 
\begin{equation*}
\mbox{$\sum^{d_k}_{j=1}\abs{(\pr^\alpha_x\Td g_j)(x)}^2$ converges at each point of $x\in D$}.
\end{equation*} 
To estimate $R_{I,J}(x,y)$, we first need

\begin{lem} \label{lIII}
With the notations used above, for every $D'\Subset D$, $\alpha\in\mathbb N^{2n}_0$, there is a constant $C_{\alpha,D'}>0$ independent of $k$, such that for all $u\in\cE^q_{k^{-\!N_0}}(M,L^k)$, $\norm{u}=1$, $u|_D=s^k\Td u$, $\Td u\in\Omega^{0,q}(D)$, if we set $\Td v(x)=\mathcal{A}^*_k\Box^{(q)}_s(\Td ue^{-k\phi})$, then
\[\abs{(\pr^\alpha_x\Td v)(x)}\leq C_{\alpha,D'}k^{\frac{5n}{2}+2\abs{\alpha}-N_0},\ \ \forall x\in D'.\]
\end{lem}

\begin{proof}
Let $u\in\cE^q_{k^{-\!N_0}}(M,L^k)$, $\norm{u}=1$, $u|_D=s^k\Td u$, $\Td u\in\Omega^{0,q}(D)$. Set $\Td v(x)=\mathcal{A}^*_k\Box^{(q)}_s(\Td ue^{-k\phi})$. We recall that
\begin{equation} \label{eXIV}
\mathcal{A}^*_k:O(k^s):H^s_{{\rm comp\,}}(D,\Lambda^qT^{*(0,1)}M)\To H^{s+1}_{{\rm comp\,}}(D,\Lambda^qT^{*(0,1)}M),\ \ \forall s\in\N_0.
\end{equation}
Let $D'\Subset D''\Subset D$. By using Fourier transforms, we see that for all $x\in D'$, we have
\[\abs{(\pr^\alpha_x\Td v)(x)}\leq C_\alpha\norm{\Td v}_{n+1+\abs{\alpha},D''},\]
where $C_\alpha$ only depends on the dimension and the length of $\alpha$. Here $\norm{.}_{s,D''}$ denotes the usual Sobolev norm of order $s$ on $D''$.
From this observation and \eqref{eXIV}, we see that
\begin{equation} \label{eXV}
\abs{(\pr^\alpha_x\Td v)(x)}\leq C_\alpha\norm{\Td v}_{n+1+\abs{\alpha},D''}\leq C'_\alpha k^{n+\abs{\alpha}}\norm{\Box^{(q)}_s(\Td ue^{-k\phi})}_{n+\abs{\alpha},D''},
\end{equation}
where $C'_\alpha>0$ is independent of $k$. Let $\Box^{(q)}_ku=f$, $f|_D=s^k\Td f$, $\Td f\in\Omega^{0,q}(D)$. We can check that $f\in\cE^q_{k^{-\!N_0}}(M,L^k)$ and $\norm{f}\leq k^{-N_0}$. From \eqref{s1-e3}, we see that
\begin{equation} \label{eXVI}
\Box^{(q)}_s(e^{-k\phi}\Td u)=e^{-k\phi}\Td f.
\end{equation}
In view of Theorem~\ref{tI}, we know that for all $\beta\in\mathbb N^{2n}_0$,
\[\abs{\pr^\beta_x(\Box^{(q)}_s(e^{-k\phi}\Td u))}=\abs{\pr^\beta_x(e^{-k\phi}\Td f)}\leq C_\beta k^{\frac{n}{2}+\abs{\beta}}\norm{f}\leq C_\beta k^{\frac{n}{2}+\abs{\beta}-N_0}\ \ \mbox{on $D''$},\]
where $C_\beta>0$ is independent of $k$. Thus,
\begin{equation} \label{eXVII}
\norm{\Box^{(q)}_s(e^{-k\phi}\Td u)}_{n+\abs{\alpha},D''}\leq\Td C_\alpha k^{\frac{3n}{2}+\abs{\alpha}-N_0},
\end{equation}
where $\Td C_\alpha>0$ is independent of $k$. Combining \eqref{eXVII} with \eqref{eXV}, the lemma follows.
\end{proof}
\begin{lem} \label{lIV}
Let $\Td g_j(x)\in\Omega^{0,q}(D)$, $j=1,\ldots,d_k$, be as in \eqref{eXIII}. For every $D'\Subset D$, $\alpha\in\mathbb N^{2n}_0$,
there is a constant $C_\alpha>0$ independent of $k$, such that for all $x\in D'$
\[\sum^{d_k}_{j=1}\abs{(\pr^\alpha_x\Td g_j)(x)}^2\leq C_\alpha k^{5n+4\abs{\alpha}-2N_0}\,.\]
\end{lem}

\begin{proof}
Fix $\alpha\in\mathbb N^{2n}_0$ and $p\in D'$. We may assume that $\sum^{d_k}_{j=1}\abs{(\pr^\alpha_x\Td g_j)(p)}^2\neq0$. Set
\[h(x)=\frac{1}{\sqrt{\sum^{d_k}_{j=1}\abs{(\pr^\alpha_x\Td g_j)(p)}^2}}
\sum^{d_k}_{j=1}f_j(x)\ol{(\pr^\alpha_x\Td g_j)(p)}.\] 
Since $\sum^{d_k}_{j=1}\abs{(\pr^\alpha_x\Td g_j)(p)}^2$ converges, we
can check that $h\in\cE^q_{k^{-\!N_0}}(M,L^k)$, $\norm{h}=1$. On $D$, we write $h=s^k\Td h$. We can check that
\[\mathcal{A}^*_k\Box^{(q)}_s(\Td he^{-k\phi})=\frac{1}{\sqrt{\sum^{d_k}_{j=1}\abs{(\pr^\alpha_x\Td g_j)(p)}^2}}
\sum^{d_k}_{j=1}\Td g_j(x)\ol{(\pr^\alpha_x\Td g_j)(p)}.\]
In view of Lemma~\ref{lIII}, we see that
\[\abs{\pr^\alpha_x(\mathcal{A}^*_k\Box^{(q)}_s(\Td he^{-k\phi}))(p)}=\sqrt{\sum^{d_k}_{j=1}\abs{(\pr^\alpha_x\Td g_j)(p)}^2}\leq C_\alpha k^{\frac{5n}{2}+2\abs{\alpha}-N_0},\]
where $C_\alpha>0$ is independent of $k$ and the point $p$.
The lemma follows.
\end{proof}

Now, we can prove

\begin{prop} \label{pII}
With the notations used above, for every $D'\Subset D$, $\alpha, \beta\in\mathbb N^{2n}_0$, there is a constant
$C_{\alpha,\beta}>0$ independent of $k$, such that
\begin{equation} \label{eXVIII-0}
\abs{(\pr^\alpha_x\pr^\beta_yR_{I,J})(x,y)}\leq C_{\alpha,\beta}k^{3n+2\abs{\alpha}+\abs{\beta}-N_0},\ \ \forall (x,y)\in D'\times D',
\end{equation}
for all strictly increasing $I$, $J$, $\abs{I}=\abs{J}=q$, where $R_{I,J}(x,y)$ is as in \eqref{eXIII}.
\end{prop}

\begin{proof}
Fix $p\in D'$ and $J$ strictly increasing, $\abs{J}=q$. Let $\alpha, \beta\in\mathbb N^{2n}_0$. We may assume that
$\sum^{d_k}_{j=1}\abs{(\pr^\beta_y(\Td f_{j,J}e^{-k\phi}))(p)}^2\neq0$. Put
\begin{equation} \label{eXVIII}
u(x)=\frac{1}{\sqrt{\sum^{d_k}_{j=1}\abs{(\pr^\beta_y(\Td f_{j,J}e^{-k\phi}))(p)}^2}}
\sum^{d_k}_{j=1}f_j(x)\ol{(\pr^\beta_y(\Td f_{j,J}e^{-k\phi}))(p)}.
\end{equation}
Then, $u\in\cE^q_{k^{-\!N_0}}(M,L^k)$, $\norm{u}=1$. On $D$, we write $u=s^k\Td u$, $\Td u=\sum'_{\abs{I}=q}\Td u_Ie^I$.
Put $\Td v=\mathcal{A}^*_k\Box^{(q)}_s(\Td ue^{-k\phi})=\sum'_{\abs{I}=q}\Td v_Ie^I\in\Omega^{0,q}(D)$. It is not difficult to
check that
\[\Td v=\frac{1}{\sqrt{\sum^{d_k}_{j=1}\abs{(\pr^\beta_y(\Td f_{j,J}e^{-k\phi}))(p)}^2}}\sum^{d_k}_{j=1}\Td g_j\ol{(\pr^\beta_y(\Td f_{j,J}e^{-k\phi}))(p)},\]
where $\{\Td g_j\}_{j=1}^{d_k}$ are as in \eqref{eXIII}. In view of Lemma~\ref{lIII}, there exists $C_\alpha>0$ independent of $k$ and the point $p$ such that
$\abs{(\pr^\alpha_x\Td v)(x)}\leq C_\alpha k^{\frac{5n}{2}+2\abs{\alpha}-N_0}$, for all $x\in D'$. In particular,
\begin{equation} \label{eXIX}
\begin{split}
\abs{(\pr^\alpha_x\Td v_I)(x)}&=\frac{1}{\sqrt{\sum^{d_k}_{j=1}\abs{(\pr^\beta_y(\Td f_{j,J}e^{-k\phi}))(p)}^2}}\abs{\sum^{d_k}_{j=1}(\pr^\alpha_x\Td g_{j,I})(x)\ol{(\pr^\beta_y(\Td f_{j,J}e^{-k\phi}))(p)}}\\
&\leq C_\alpha k^{\frac{5n}{2}+2\abs{\alpha}-N_0},\ \ \forall x\in D',
\end{split}
\end{equation}
for all strictly increasing $I$, $\abs{I}=q$. In view of Proposition~\ref{pI}, we see that
\[\sum^{d_k}_{j=1}\abs{(\pr^\beta_y(\Td f_je^{-k\phi}))(p)}^2\leq C_\beta k^{n+2\abs{\beta}},\]
where $C_\beta>0$ is independent of $k$ and the point $p$. From this and \eqref{eXIX}, we conclude the existence of a constant $C_{\alpha,\beta}>0$ independent of $k$ and the point $p$ with
\[\abs{(\pr^\alpha_x\pr^\beta_yR_{I,J})(x,p)}=
\sqrt{\sum^{d_k}_{j=1}\abs{(\pr^\beta_y(\Td f_{j,J}e^{-k\phi}))(p)}^2}\abs{(\pr^\alpha_x\Td v_I)(x)}\leq C_{\alpha,\beta}k^{3n+2\abs{\alpha}+\abs{\beta}-N_0},\]
for all $x\in D'$, all strictly increasing $I,J$ with $|I|=|J|=q$.
The proposition follows.
\end{proof}

From \eqref{eXII} and Proposition~\ref{pII}, we know that
\[\widehat P^{(q)}_{k,k^{-N_0},s}=R+\mathcal{S}^*_k\widehat P^{(q)}_{k,k^{-N_0},s},\]
where $R(x,y)$ satisfies \eqref{eXVIII-0}. We have
\begin{equation} \label{eXIX-0}
\widehat P^{(q)}_{k,k^{-N_0},s}\,\mathcal{S}_k=(R+\mathcal{S}^*_k\widehat P^{(q)}_{k,k^{-N_0},s})\,\mathcal{S}_k=R\mathcal{S}_k+\mathcal{S}^*_k\widehat P^{(q)}_{k,k^{-N_0},s}\,\mathcal{S}_k.
\end{equation}
Let $R^*$ be the formal adjoint $R$ with respect to $(\,\cdot\,,\cdot)$. Then,
\begin{equation} \label{eXX}
\widehat P^{(q)}_{k,k^{-N_0},s}=R^*+\widehat P^{(q)}_{k,k^{-N_0},s}\,\mathcal{S}_k.
\end{equation}
From \eqref{eXX} and \eqref{eXIX-0}, we get
\begin{equation} \label{eXXI}
\widehat P^{(q)}_{k,k^{-N_0},s}=R^*+R\mathcal{S}_k+\mathcal{S}^*_k\widehat P^{(q)}_{k,k^{-N_0},s}\mathcal{S}_k.
\end{equation}
We also write
\[R^*(x,y)=\sideset{}{'}\sum_{\abs{I}=q,\abs{J}=q}e^I(x)R^*_{I,J}(x,y)e^J(y).\]
Since $R^*_{I,J}(x,y)=\ol{R_{J,I}(y,x)}$,
$R^*(x,y)$ also satisfies \eqref{eXVIII-0}.

Now, we study the kernel of $R\mathcal{S}_k$. We write
\[(R\mathcal{S}_k)(x,y)=\sideset{}{'}\sum_{\abs{I}=q,\abs{J}=q}e^I(x)(R\mathcal{S}_k)_{I,J}(x,y)e^J(y).\]
From \eqref{s5-e2-1}, we know that for all strictly increasing $I$, $J$, $\abs{I}=\abs{J}=q$,
\begin{equation} \label{eXXII}
(R\mathcal{S}_k)_{I,J}(x,y)=\sideset{}{'}\sum_{\abs{K}=q}\int_DR_{I,K}(x,z)\mathcal{S}_{kK,J}(z,y)dv_M(z)\,.
\end{equation}

\begin{lem} \label{lV}
For every $D'\Subset D$, $\alpha\in\mathbb N^{2n}_0$, there is a constant $C_\alpha>0$ independent of $k$, such that for all strictly increasing $I$, $\abs{I}=q$, we have
\begin{equation} \label{eXXIII}
\sideset{}{'}\sum_{\abs{K}=q}\int_D\abs{(\pr^\alpha_xR_{I,K})(x,z)}^2dv_M(z)\leq C_\alpha k^{5n+4\abs{\alpha}-2N_0},\ \ x\in D'.
\end{equation}
\end{lem}

\begin{proof}
From \eqref{eXIII}, we see that for $\alpha\in\mathbb N^{2n}_0$ we have
\begin{equation} \label{eXXIV}
(\pr^\alpha_xR_{I,K})(x,y)=\sum^{d_k}_{j=1}(\pr^\alpha_x\Td g_{j,I})(x)\ol{\Td f_{j,K}(y)}e^{-k\phi(y)}\,.
\end{equation}
We claim that
\begin{equation} \label{eXXV}
\sideset{}{'}\sum_{\abs{K}=q}\int_D\abs{(\pr^\alpha_xR_{I,K})(x,y)}^2dv_M(y)\leq\sum^{d_k}_{j=1}\abs{(\pr^\alpha_x\Td g_{j,I})(x)}^2,
\end{equation}
for all $x\in D$, strictly increasing $I$, $\abs{I}=q$. Fix such $I$ and $p\in D$. We may assume that $\sum^{d_k}_{j=1}\abs{(\pr^\alpha_x\Td g_{j,I})(p)}^2\neq0$. Put
\[u(x)=\frac{1}{\sqrt{\sum^{d_k}_{j=1}\abs{(\pr^\alpha_x\Td g_{j,I})(p)}^2}}\sum^{d_k}_{j=1}\ol{(\pr^\alpha_x\Td g_{j,I})(p)}f_j(x)\in\cE^q_{k^{-\!N_0}}(M,L^k).\]
We see that $\norm{u}=1$. Thus, $\int_D\abs{u}^2\leq1$. On $D$, we can check that
\begin{equation} \label{eXXVI}
\int_D\abs{u}^2=\frac{1}{\sum^{d_k}_{j=1}\abs{(\pr^\alpha_x\Td g_{j,I})(p)}^2}\sideset{}{'}\sum_{\abs{K}=q}\int_D\abs{\sum^{d_k}_{j=1}(\pr^\alpha_x\Td g_{j,I})(p)\ol{\Td f_{j,K}(y)}}^2e^{-2k\phi(y)}dv_M(y)\leq1.
\end{equation}
From \eqref{eXXIV} and \eqref{eXXVI}, we see that
\[\sideset{}{'}\sum_{\abs{K}=q}\int_D\abs{(\pr^\alpha_xR_{I,K})(p,y)}^2dv_M(y)
\leq\sum^{d_k}_{j=1}\abs{(\pr^\alpha_x\Td g_{j,I})(p)}^2.\]
\eqref{eXXV} follows. From \eqref{eXXV} and Lemma~\ref{lIV}, the lemma follows.
\end{proof}

\noindent
From \eqref{eXXII}, for all strictly increasing $I$, $J$, $\abs{I}=\abs{J}=q$, we have
\begin{equation} \label{eXXVII}
\begin{split}
&\abs{\pr^\alpha_x\pr^\beta_y((R\mathcal{S}_k)_{I,J})(x,y)}\\
&=\abs{\sideset{}{'}\sum_{\abs{K}=q}\int_D(\pr^\alpha_xR_{I,K})(x,z)(\pr^\beta_y\mathcal{S}_{kK,J})(z,y)dv_M(z)}\\
&\leq\sideset{}{'}\sum_{\abs{K}=q}\Bigr(\int_D\abs{(\pr^\alpha_xR_{I,K})(x,z)}^2dv_M(z)\Bigr)^{\frac{1}{2}}
\Bigr(\int_D\abs{(\pr^\beta_y\mathcal{S}_{kK,J})(z,y)}^2dv_M(z)\Bigr)^{\frac{1}{2}}\\
&\leq\Bigr(\sideset{}{'}\sum_{\abs{K}=q}\int_D\abs{(\pr^\alpha_xR_{I,K})(x,z)}^2dv_M(z)\Bigr)^{\frac{1}{2}}
\Bigr(\sideset{}{'}\sum_{\abs{K}=q}\int_D\abs{(\pr^\beta_y\mathcal{S}_{kK,J})(z,y)}^2dv_M(z)\Bigr)^{\frac{1}{2}}.
\end{split}
\end{equation}
Note that
\begin{equation} \label{eXXVIII}
\begin{split}
&\sideset{}{'}\sum_{\abs{K}=q}\int_D\abs{(\pr^\beta_y\mathcal{S}_{kK,J})(z,y)}^2dv_M(z)\\
&=\sideset{}{'}\sum_{\abs{K}=q}\int_D(\pr^\beta_x\mathcal{S}^*_{kJ,K})(y,z)(\pr^\beta_y\mathcal{S}_{kK,J})(z,y)dv_M(z)\\
&=(\pr^\beta_x\pr^\beta_y(\mathcal{S}_k^*\mathcal{S}_k)_{J,J})(y,y).
\end{split}
\end{equation}
We notice that $\mathcal{S}^*_k\mathcal{S}_k\equiv\mathcal{S}_k\mod O(k^{-\infty})$. From this observation and the explicit formula of the kernel of $\mathcal{S}_k$ (see~\eqref{s3-es-1}), we conclude that
\begin{equation} \label{eXXIX}
\abs{(\pr^\beta_x\pr^\beta_y(\mathcal{S}^*_k\mathcal{S}_k)_{J,J})(y,y)}\leq C_\beta k^{n+2\abs{\beta}},
\end{equation}
locally uniformly on $D$, for all strictly increasing $J$, $\abs{J}=q$, where $C_\beta>0$ is independent of $k$. From \eqref{eXXIX}, \eqref{eXXVIII}, \eqref{eXXVII} and Lemma~\ref{lV}, we conclude that for all strictly increasing $I$, $J$, $\abs{I}=\abs{J}=q$, 
\[\abs{(\pr^\alpha_x\pr^\beta_y(R\mathcal{S}_k)_{I,J})(x,y)}\leq C_{\alpha,\beta}k^{3n+2\abs{\alpha}+\abs{\beta}-N_0},\]
locally uniformly on $D$, where $C_{\alpha,\beta}>0$ is independent of $k$. Put
\[T=R^*+R\mathcal{S}_k.\]
We write
\[T(x,y)=\sideset{}{'}\sum_{\abs{I}=q,\abs{J}=q}e^I(x)T_{I,J}(x,y)e^J(y)\]
in the sense of \eqref{s1-e14}. From \eqref{eXXI},
we know that
\begin{equation} \label{eXXX}
\widehat P^{(q)}_{k,k^{-N_0},s}=T+\mathcal{S}^*_k\widehat P^{(q)}_{k,k^{-N_0},s}\mathcal{S}_k.
\end{equation}
From the discussion above, we know that for every $D'\Subset D$, $\alpha, \beta\in\mathbb N^{2n}_0$, every strictly increasing $I,J$, $\abs{I}=\abs{J}=q$, there is a constant $C_{\alpha,\beta}>0$ independent of $k$ such that
\begin{equation} \label{eXXXI}
\abs{(\pr^\alpha_x\pr^\beta_yT_{I,J})(x,y)}\leq C_{\alpha,\beta}k^{3n+2\abs{\alpha}+\abs{\beta}-N_0},\ \  (x,y)\in D'\times D'\,.
\end{equation}
Let $T^*$ be the formal adjoint of $T$ with respect to $(\,\cdot\,,\cdot\,)$. From \eqref{eXXX}, we see that $T^*=T$. Thus,
\[\abs{(\pr^\alpha_x\pr^\beta_yT_{I,J})(x,y)}=\abs{\ol{(\pr^\alpha_x\pr^\beta_yT_{I,J})(x,y)}}=
\abs{(\pr^\alpha_y\pr^\beta_xT_{J,I})(y,x)}\leq C_{\beta,\alpha}k^{3n+2\abs{\beta}+\abs{\alpha}-N_0}.\]
Combining this with \eqref{eXXXI}, we conclude that for every $D'\Subset D$, $\alpha, \beta\in\mathbb N^{2n}_0$, every strictly increasing $I$, $J$, $\abs{I}=\abs{J}=q$, there is a constant $C_{\alpha,\beta}>0$ independent of $k$ such that
\begin{equation} \label{eXXXII}
\abs{(\pr^\alpha_x\pr^\beta_yT_{I,J})(x,y)}\leq C_{\alpha,\beta}{\rm min\,}\set{k^{3n+2\abs{\alpha}+\abs{\beta}-N_0},k^{3n+\abs{\alpha}+2\abs{\beta}-N_0}},\ \ (x,y)\in D'\times D'\,.
\end{equation}
Summing up, we get the following.

\begin{prop} \label{pIII}
In the situation of Setup \ref{local_data}
let $q=n_-$. Fix $N_0\geq1$.
Let $\mathcal{S}_{k}$ be the localized approximate Szeg\"o kernel \eqref{s3-e17-1} and let $\widehat P^{(q)}_{k,k^{-N_0},s}$ be the localized spectral projection \eqref{eXI}. Then,
\[\widehat P^{(q)}_{k,k^{-N_0},s}=T+\mathcal{S}^*_k\widehat P^{(q)}_{k,k^{-N_0},s}\mathcal{S}_k,\]
where $T$ is smoothing and the distribution kernel of $T$ satisfies \eqref{eXXXII}.
\end{prop} 

\subsection{Asymptotic expansion of the spectral function. Proof of Theorem \ref{s1-main1}}\label{s:kern_spec_asy}

Consider $\lambda\geqslant0$ and denote by $\cE^q_{>\lambda}(M, L^k)\subset L^2_{(0,q)}(M,L^k)$ the spectral space given by the range of $E((\lambda,\infty))$, where $E$ is the spectral measure of $\Box^{(q)}_k$.
Let
\[P^{(q)}_{k,>\lambda}:L^2_{(0,q)}(M,L^k)\To\cE^q_{>\lambda}(M,L^k)\]
be the orthogonal projection. As before, let $s$ be a local frame of $L$ on an open set $D\Subset M$ and $\abs{s}_{h^L}^2=e^{-2\phi}$. Consider the localization
\begin{equation} \label{eXXXIII}
\begin{split}
\widehat P^{(q)}_{k,>\lambda,s}:L^2_{(0,q)}(D)\cap\mathscr E'(D,\Lambda^qT^{*(0,1)}M)&\To L^2_{(0,q)}(D),\\
u&\mapsto  e^{-k\phi}s^{-k}P^{(q)}_{k,>\lambda}(s^ke^{k\phi}u).
\end{split}
\end{equation}
Fix $N_0\geq1$. It is well-known that (see Davies~\cite[Section\,2]{Dav95}) 
\[L^2_{(0,q)}(M)=\cE^q_{k^{-\!N_0}}(M,L^k)\oplus\cE^q_{>k^{-\!N_0}}(M,L^k)\] 
and
\begin{equation} \label{eXXXIV}
\norm{u}\leq k^{N_0}\norm{\Box^{(q)}_ku},\ \  \forall u\in\cE^q_{>k^{-\!N_0}}(M,L^k)\cap{\rm Dom\,}\Box^{(q)}_k.
\end{equation}
We have the decomposition
\begin{equation} \label{eXXXV}
u=\widehat P^{(q)}_{k,k^{-N_0},s}u+\widehat P^{(q)}_{k,>k^{-N_0},s}u,\ \ u\in\Omega^{0,q}_0(D).
\end{equation}

Now, we assume that $\pr\ddbar\phi$ is non-degenerate of constant signature $(n_-,n_+)$ at each point of $D$ and 
let $q=n_-$. Let $\mathcal{S}_{k}$ be the localized approximate Szeg\"o kernel \eqref{s3-e17-1}. 
From the explicit formula of the kernel of $\mathcal{S}_{k}$ (see \eqref{s3-es-1}), we can check that
\begin{equation} \label{eXXXIX}
\mathcal{S}_{k}^*, \mathcal{S}_{k}=O(k^{n+\abs{s_1}+\abs{s}}):H^{s_1}_{{\rm loc\,}}(D,\Lambda^qT^{*(0,1)}M)\To H^{s}_{{\rm loc\,}}(D,\Lambda^qT^{*(0,1)}M),
\end{equation}
locally uniformly on $D$, for all $s, s_1\in\mathbb Z$, $s_1\leq0$, $s\geq0$.

Let $u\in H^{s_1}_{{\rm comp\,}}(D,\Lambda^qT^{*(0,1)}M)$, $s_1\leq0$, $s_1\in\mathbb Z$.
From \eqref{eXXXV}, we have
\begin{equation} \label{eXXXVI}
\mathcal{S}_{k}u=\widehat P^{(q)}_{k,k^{-N_0},s}\,\mathcal{S}_{k}u+\widehat P^{(q)}_{k,>k^{-N_0},s}\,\mathcal{S}_{k}u.
\end{equation}
From \eqref{eXXXIII} and \eqref{eXXXIV}, we can check that
\begin{equation} \label{eXXXVII}
\begin{split}
\norm{\widehat P^{(q)}_{k,>k^{-N_0},s}\,\mathcal{S}_{k}u}_D&\leq\norm{P^{(q)}_{k,>k^{-\!N_0}}(s^ke^{k\phi}(\mathcal{S}_{k}u))}
\leq k^{N_0}\norm{\Box^{(q)}_kP^{(q)}_{k,>k^{-\!N_0}}(s^ke^{k\phi}(\mathcal{S}_{k}u))}\\
&\leq k^{N_0}\norm{\Box^{(q)}_k(s^ke^{k\phi}(\mathcal{S}_{k}u))}=k^{N_0}\norm{\Box^{(q)}_s(\mathcal{S}_{k}u)}.
\end{split}
\end{equation}
Here we have used \eqref{s1-e3}. In view of Theorem~\ref{s3-t4}, we see that $\Box^{(q)}_s\mathcal{S}_{k}\equiv0\mod O(k^{-\infty})$.
From this observation and \eqref{eXXXVII}, we conclude that
\begin{equation} \label{eXXXVIII}
\widehat P^{(q)}_{k,>k^{-N_0},s}\,\mathcal{S}_{k}=O(k^{-N}):H^{s_1}_{{\rm comp\,}}(D,\Lambda^qT^{*(0,1)}M)\To H^0_{{\rm loc\,}}(D,\Lambda^qT^{*(0,1)}M),
\end{equation}
locally uniformly on $D$, for all $N\geq0$, $s_1\in\mathbb Z$, $s_1\leq0$.
From \eqref{eXXXIX} and \eqref{eXXXVIII}, we conclude that
\begin{equation} \label{eXXXX}
\mathcal{S}_{k}^*\widehat P^{(q)}_{k,>k^{-N_0},s}\,\mathcal{S}_{k}\equiv0\mod O(k^{-\infty})\,.
\end{equation}
Combining \eqref{eXXXX} with \eqref{eXXXVI} and using that $\mathcal{S}_{k}^*\mathcal{S}_{k}\equiv\mathcal{S}_{k}\mod O(k^{-\infty})$, we get
\begin{equation} \label{eXXXXI}
\mathcal{S}_{k}\equiv\mathcal{S}_{k}^*\widehat P^{(q)}_{k,k^{-N_0},s}\,\mathcal{S}_{k}\mod O(k^{-\infty})\,.
\end{equation}
From \eqref{eXXXXI} and Proposition~\ref{pIII}, Theorem~\ref{s3-t5} and Theorem~\ref{s3-t6},
we get one of the main results of this work:
\begin{thm} \label{tII}
In the situation of Setup \ref{local_data}
let $q=n_-$, fix $N_0\geq1$ and let
$\widehat P^{(q)}_{k,k^{-N_0},s}$ be the localized spectral projection \eqref{eXI} and let
$\widehat P^{(q)}_{k,k^{-N_0},s}(\cdot,\cdot)$
be its distribution kernel. Then, for every $D'\Subset D$, $\alpha, \beta\in\mathbb N^{2n}_0$, there is a constant $C_{\alpha,\beta}>0$ independent of $k$, such that 
\begin{equation} \label{eXXXXIV}
\abs{\pr^\alpha_x\pr^\beta_y\big(\widehat P^{(q)}_{k,k^{-N_0},s}(x,y)-\mathcal{S}_{k}(x,y)\big)}\leq
C_{\alpha,\beta}\,{\min}\set{k^{3n+2\abs{\alpha}+\abs{\beta}-N_0},k^{3n+\abs{\alpha}+2\abs{\beta}-N_0}}\ \ 
\end{equation}
holds on $D'\times D'$, where
\[\mathcal{S}_{k}(x,y)=\mathcal{S}_{k}(z,w)\equiv e^{ik\Psi(z,w)}b(z,w,k)\mod
O(k^{-\infty}),\]
with 
\[
\begin{split}
&b(z,w,k)\in S^{n}_{{\rm loc\,}}\big(1;D\times D, \Lambda^qT^{*(0,1)}_wM\boxtimes\Lambda^qT^{*(0,1)}_zM\big), \\
&b(z,w,k)\sim\sum^\infty_{j=0}b_j(z, w)k^{n-j}\text{ in }S^{n}_{{\rm loc\,}}
\big(1;D\times D, \Lambda^qT^{*(0,1)}_wM\boxtimes\Lambda^qT^{*(0,1)}_zM\big), \\
&b_j(z, w)\in\cC^\infty\big(D\times D, \Lambda^qT^{*(0,1)}_wM\boxtimes\Lambda^qT^{*(0,1)}_zM\big),\ \ j=0,1,2,\ldots,\\
&\mbox{$b_0(z,z)$ is given by \eqref{s3-es-6}}\,,
\end{split}
\]
and $\Psi\in\cC^\infty(D\times D)$ satisfying \eqref{prop_psi} and for a given
point $p\in D$, consider local holomorphic coordinates
$z=(z_1,\ldots,z_n)$ centered at $p$ as in \eqref{s3-e16-bisbb}.
Then $\Psi$ has the form \eqref{s3-e16-bisbg} near $(0,0)$.
Moreover, let $\{Z_j\}^n_{j=1}$ be a smooth orthonormal frame of $T^{(0,1)}M$ over $D$.
Then,
\begin{equation} \label{eikonal}
\sum^n_{j=1}\Bigr(\bigr(iZ_j\Psi\bigr)(z,w)+\bigr(Z_j\phi\bigr)(z)\Bigr)\Bigr(\bigr(-i\ol Z_j\Psi\bigr)(z,w)+\bigr(\ol Z_j\phi\bigr)(z)\Bigr)=O(\abs{z-w}^N),
\end{equation}
locally uniformly on $D\times D$, for all $N\in\mathbb N$.
\end{thm}
\noindent
When $q\neq n_-$, we use Theorem~\ref{s3-t4-1} and repeat the proof of Theorem~\ref{tII} to conclude that

\begin{thm} \label{tIII}
In the situation of Setup~\ref{local_data} let $q\neq n_-$, fix $N_0\geq1$. With the notations used in Theorem~\ref{tII}.
Then, for every $D'\Subset D$, $\alpha, \beta\in\mathbb N^{2n}_0$, there is a constant $C_{\alpha,\beta}>0$ independent of $k$, such that
\begin{equation} \label{eXXXXIV-I}
\abs{\pr^\alpha_x\pr^\beta_y(\widehat P^{(q)}_{k,k^{-N_0},s}(x,y))}\leq
C_{\alpha,\beta}{\rm min\,}\set{k^{3n+2\abs{\alpha}+\abs{\beta}-N_0},k^{3n+\abs{\alpha}+2\abs{\beta}-N_0}}\ \ \mbox{on $D'\times D'$}.
\end{equation}
\end{thm}

\begin{proof}[Proof of Theorem \ref{s1-main1}]
Combining Theorem~\ref{tII} and Theorem~\ref{tIII}, we get \eqref{s1-e1main}, \eqref{s1-e2main} and \eqref{s1-e21main}.
\end{proof}
\begin{rem} \label{rabis2}
In view of Remark~\ref{rabis1}, we can generalize Theorem~\ref{tII} and Theorem~\ref{tIII} with essentially the same proofs to the case when the forms take values
in $L^k\otimes E$, for a given holomorphic vector bundle $E$ over $M$.
\end{rem}

\subsection{Asymptotic expansion of the Bergman kernel. Proof of Theorem \ref{s1-main2}} \label{s:asymp_BK}

We are now ready to prove Theorem~\ref{s1-main2}. 
In the situation of Setup~\ref{local_data} let $q=n_-$.
Define the \emph{localized Bergman projection} (with respect to $s$) by
\begin{align} \label{e:localspectral1}
\widehat P^{(q)}_{k,s}:L^2_{(0,q)}(D)\cap\mathscr E'(D,\Lambda^qT^{*(0,1)}M)&\To\Omega^{0,q}(D),\nonumber \\
u&\mapsto e^{-k\phi}s^{-k}P^{(q)}_{k}(s^ke^{k\phi}u).
\end{align} 
Let $\widehat P^{(q)}_{k,s}(x,y)$ be the distribution kernel of $\widehat P^{(q)}_{k,s}$. We have the following 

\begin{thm}\label{t:localspectral}  
With the assumptions and notations above, fix $N_0\geq1$ and assume that $\Box^{(q)}_k$ has $O(k^{-n_0})$ small spectral gap on $D$. Then for every $D'\Subset D$, $\alpha, \beta\in\mathbb N^{2n}_0$, there is a constant $C_{\alpha,\beta}>0$ independent of $k$, such that
\begin{equation} \label{e:localspectral2}
\begin{split}
&\abs{\pr^\alpha_x\pr^\beta_y(\widehat P^{(q)}_{k,k^{-N_0},s}(x,y)-\widehat P^{(q)}_{k,s}(x,y))}\\
&\quad\leq
C_{\alpha,\beta}{\rm min\,}\set{k^{3n+2\abs{\alpha}+\abs{\beta}-N_0},k^{3n+\abs{\alpha}+2\abs{\beta}-N_0}}\ \ \mbox{on $D'\times D'$},
\end{split}
\end{equation}
where $\widehat P^{(q)}_{k,k^{-N_0},s}$ is as in Theorem~\ref{tII}. 
In particular, 
\[\widehat P^{(q)}_{k,s}\equiv\mathcal{S}_{k}\mod O(k^{-\infty})\]
locally uniformly on $D$, where $\mathcal{S}_k$ is as in Theorem~\ref{tII}.
\end{thm} 

\begin{proof} 
Let $\mathcal{S}_k$ be as in Theorem~\ref{tII}. 
We can repeat the proof of Proposition~\ref{pIII} and conclude that 
\begin{equation}\label{e:localspectral3}
\widehat P^{(q)}_{k,k^{-N_0},s}-\widehat P^{(q)}_{k,s}=T+\mathcal{S}^*_k\Bigr(\widehat P^{(q)}_{k,k^{-N_0},s}-\widehat P^{(q)}_{k,s}\Bigr)\mathcal{S}_k,
\end{equation} 
where $T$ is smoothing and the distribution kernel $T(x,y)$ of $T$ satisfies \eqref{eXXXII}. 
Let 
\[u\in H^m_{{\rm comp\,}}(D,\Lambda^qT^{*(0,1)}M),\ \ m\leq0,\ \ m\in\mathbb Z.\]
We consider
\[v=s^ke^{k\phi}\mathcal{S}_{k}u- P^{(q)}_{k}(s^ke^{k\phi}\mathcal{S}_{k}u).\]
Since $\mathcal{S}_{k}$ is a smoothing operator, $v\in\cC^\infty(M,L^k)$. Moreover, it is easy to see that $v\bot \cH^0(M,L^k)$. We have
\begin{equation} \label{e:localspectral4}
\Box^{(q)}_kv=s^ke^{k\phi}\Box^{(q)}_s\mathcal{S}_{k}u.
\end{equation}
From Theorem~\ref{s3-t4}, we see that $\Box^{(q)}_s\mathcal{S}_{k}\equiv0\mod O(k^{-\infty})$. Combining this with 
\eqref{e:localspectral4}, we obtain
\[\norm{\Box^{(q)}_kv}\leq C_Nk^{-N}\norm{u}_m,\]
for every $N>0$, where $C_N>0$ is independent of $k$. Since $v\bot\cH^0(M,L^k)$, from
Definition~\ref{s1-d2bis} we conclude that
\[\norm{v}\leq\Td C_Nk^{-N}\norm{u}_m,\]
for every $N>0$, where $\Td C_N>0$ is independent of $k$. Thus,
\[\mathcal{S}_{k}-\widehat P^{(q)}_{k,s}\mathcal{S}_{k}=O(k^{-N}):H^m_{{\rm comp\,}}(D,\Lambda^qT^{*(0,1)}M)\To L^2(D,\Lambda^qT^{*(0,1)}M),\]
for all $N>0$, $m\in\mathbb Z$, $m\leq0$, and hence
\[\mathcal{S}_{k}^*\mathcal{S}_{k}-\mathcal{S}_{k}^*\widehat P^{(q)}_{k,s}\mathcal{S}_{k}=O(k^{-N}):H^m_{{\rm comp\,}}(D,\Lambda^qT^{*(0,1)}M)\To H^{m+N_1}_{{\rm loc\,}}(D,\Lambda^qT^{*(0,1)}M),\]
for all $N, N_1>0$, $m\in\mathbb Z$. We conclude that
\[\mathcal{S}_{k}^*\mathcal{S}_{k}\equiv\mathcal{S}_{k}^*\widehat P^{(q)}_{k,s}\mathcal{S}_{k}\mod O(k^{-\infty}).\]
From this, \eqref{s3-e22-0a} and \eqref{eXXXXIV}, we obtain 
\[\widehat P^{(q)}_{k,k^{-N_0},s}=\Td T+\mathcal{S}_{k}^*\widehat P^{(q)}_{k,s}\mathcal{S}_{k},\]
where $\Td T$ is smoothing and the distribution kernel $\Td T(x,y)$ of $\Td T$ satisfies \eqref{eXXXII}. 
From this and Proposition~\ref{pIII}, we conclude that the distribution kernel 
of $\mathcal{S}_{k}^*\bigl(\widehat P^{(q)}_{k,k^{-N_0},s}-\widehat P^{(q)}_{k,s}\bigr)\mathcal{S}_{k}$ satisfies \eqref{eXXXII}. 
Combining this with \eqref{e:localspectral3}, \eqref{e:localspectral2} follows.
\end{proof} 

Since Theorem~\ref{s3-t4} and Theorem~\ref{tII} hold in the case when the forms take values in $L^k\otimes E$, for a given holomorphic vector bundle $E$ over $M$, we can generalize Theorem~\ref{t:localspectral} with the same proof to the case when the forms take values in
$L^k\otimes E$. 

\subsection{Calculation of the leading coefficients. Proof of Theorem \ref{s1-main12}}\label{s:coeff} 

Now, we prove \eqref{coeII} and \eqref{coeIII}. In this Section we assume that $q=0$.
First let us review the necessary definitions from Riemannian geometry. We will use the same notations as in the discussion 
after \eqref{s1-e21main}.

Consider the K\"ahler metric $\omega=\frac{\imat}{2\pi}R^L$ introduced in \eqref{coeI}. 
Let $\langle\,\cdot\,,\cdot\,\rangle_\omega$ be the Hermitian metric on $\Complex TM$ induced by $\omega$.
In local holomorphic coordinates $z=(z_1,\ldots,z_n)$, put
\begin{equation} \label{sa1-e7} 
\begin{split}
\omega=\sqrt{-1}\sum^n_{j,k=1}\omega_{j,k}dz_j\wedge d\ol z_k,\:
\Theta=\sqrt{-1}\sum^n_{j,k=1}\Theta_{j,k}dz_j\wedge d\ol z_k, 
\end{split}
\end{equation}
where $\Theta_{j,k}=\langle\,\frac{\pr}{\pr z_j}\,,\frac{\pr}{\pr z_k}\,\rangle$, $\omega_{j,k}=\langle\,\frac{\pr}{\pr z_j}\,,\frac{\pr}{\pr z_k}\,\rangle_{\omega}$, $j,k=1,\ldots,n$.
Put 
\begin{equation} \label{sa1-e8}
h=\left(h_{j,k}\right)^n_{j,k=1},\ \ h_{j,k}=\omega_{k,j},\ \ j, k=1,\ldots,n,
\end{equation}
and $h^{-1}=\left(h^{j,k}\right)^n_{j,k=1}$, $h^{-1}$ is the inverse matrix of $h$. The complex Laplacian with respect to $\omega$ is given by 
\begin{equation} \label{sa1-e9} 
\triangle_{\omega}=(-2)\sum^n_{j,k=1}h^{j,k}\frac{\pr^2}{\pr z_j\pr\ol z_k}.
\end{equation} 
We notice that $h^{j,k}=\langle\,dz_j\,,dz_k\,\rangle_{\omega}$, $j, k=1,\ldots,n$. Put 
\begin{equation} \label{sa1-e10} 
\begin{split}
V_\omega:=\det\left(\omega_{j,k}\right)^n_{j,k=1},\:
V_\Theta:=\det\left(\Theta_{j,k}\right)^n_{j,k=1}
\end{split}
\end{equation}
and set  
\begin{equation} \label{sa1-e11} 
\begin{split}
r=\triangle_{\omega}\log V_\omega,\:
\widehat r=\triangle_{\omega}\log V_\Theta.
\end{split}
\end{equation} 
Then $r$ is the scalar curvature of $g^{TX}_\omega$. Let $R^{\det}_\Theta$ be the curvature of the canonical line bundle $K_M=\det T^{*(1,0)}M$ with respect to the real two form $\Theta$. We recall that 
\begin{equation} \label{sa1-e12}
R^{\rm det\,}_\Theta=-\ddbar\pr\log V_\Theta.
\end{equation} 

Let $h$ be as in \eqref{sa1-e8}. The connection 
matrix of the Chern connection on $T^{(1,0)}M$ is given by  
$\theta=h^{-1}\pr h=\left(\theta_{j,k}\right)^n_{j,k=1}$, $\theta_{j,k}\in T^{*(1,0)}M$, $j,k=1,\ldots,n$.
$\theta$ is the Chern connection 
matrix with respect to $\omega$. The Chern curvature with respect to $\omega$ is given by
\begin{equation} \label{sa1-e13}
\begin{split}
&R^{TM}_{\omega}=\ddbar\theta=\left(\ddbar\theta_{j,k}\right)^n_{j,k=1}=\left(\mathcal{R}_{j,k}\right)^n_{j,k=1}
\in\cC^\infty\bigr(M,\Lambda^{1,1}T^*M\otimes{\rm End\,}(T^{(1,0)}M)\bigr),\\
&R^{TM}_{\omega}(\ol U, V)\in {\rm End\,}(T^{(1,0}M),\ \ \forall U, V\in T^{(1,0)}M,\\
&R^{TM}_\omega(\ol U,V)\xi=\sum^n_{j,k=1}\langle\,\mathcal{R}_{j,k}\,,\ol U\wedge V\,\rangle\xi_k\frac{\pr}{\pr z_j},\ \ \xi=\sum^n_{j=1}\xi_j\frac{\pr}{\pr z_j},\ \ U, V\in T^{(1,0)}M.
\end{split}
\end{equation} 
Set 
\begin{equation} \label{sa1-e14} 
\abs{R^{TM}_{\omega}}^2_{\omega}:=\sum^n_{j,k,s,t=1}\abs{\langle\,R^{TM}_{\omega}(\ol e_j,e_k)e_s\,,e_t\,\rangle_{\omega}}^2,
\end{equation} 
where $e_1,\ldots,e_n$ is an orthonormal frame for $T^{(1,0)}M$ with respect to $\langle\,\cdot\,,\cdot\,\rangle_{\omega}$. 
It is straightforward to see that the definition of $\abs{R^{TM}_{\omega}}^2_{\omega}$ is independent of the choices of orthonormal frames. Thus, $\abs{R^{TM}_{\omega}}^2_{\omega}$ is globally defined. The Ricci curvature with respect to $\omega$ is given by 
\begin{equation} \label{sa1-e15}
{\rm Ric\,}_{\omega}:=-\sum^n_{j=1}\langle\,R^{TM}_\omega(\cdot,e_j)\,\cdot\,,e_j\,\rangle_\omega,
\end{equation}
where $e_1,\ldots,e_n$ is an orthonormal frame for $T^{(1,0)}M$ with respect to $\langle\,\cdot\,,\cdot\,\rangle_{\omega}$. That is, 
\[\langle\,{\rm Ric\,}_{\omega}\,, U\wedge V\,\rangle=-\sum^n_{j=1}\langle\,R^{TX}_\omega(U,e_j)V\,,e_j\,\rangle_\omega,\ \ 
U, V\in\Complex TM.\]
${\rm Ric\,}_{\omega}$ is a global $(1,1)$ form. We can check that 
\[{\rm Ric\,}_\omega=-\ddbar\pr\log V_\omega,\]
where $V_\omega$ is as in \eqref{sa1-e10}.

Let $\mathcal{S}_k$, $b(z,w,k)$, $b_j(z,w)$, $j=0,1,2,\ldots$, be as in Theorem~\ref{tII}.
We will calculate $b_1(p,p)$ and $b_2(p,p)$ at a fixed $p\in D$. 
In a small neighbourhood $D\Subset M(0)$ of the point $p$ there exist local coordinates $(D,z)\cong(D,x)$ centered at $p$ and 
a local frame $s$ of $L$, $\abs{s}_{h^L}^2=e^{-2\phi}$ so that $\phi$ is a K\"ahler potential of $\omega$ satisfying
\begin{equation} \label{coecaluI}
\begin{split}
&\phi(z)=\sum^n_{j=1}\lambda_j\abs{z_j}^2+\phi_1(z),\\
&\phi_1(z)=O(\abs{z})^4),\ \ \frac{\pr^{\abs{\alpha}+\abs{\beta}}\phi_1}{\pr z^\alpha\pr\ol z^\beta}(0)=0\ \ \text{for $\alpha, \beta\in\N_0^n$, $\abs{\alpha}\leq1$ or $\abs{\beta}\leq1$}\,,
\end{split}
\end{equation} 
and moreover 
\begin{equation} \label{coecaluI1}
\Theta(z)=\sqrt{-1}\sum^n_{j=1}dz_j\wedge d\ol z_j+O(\abs{z})\,,\:\:z\to0\,,
\end{equation} 
(this is possible by Ruan~\cite{Ruan98}). 
First, we claim that
\begin{equation} \label{coecaluII}
\ddbar_s\mathcal{S}_{k}\equiv0\mod O(k^{-\infty}),
\end{equation} 
where $\ddbar_s$ is as in \eqref{s1-e4}.
We notice that $\Box^{(0)}_s\mathcal{S}_{k}\equiv0\mod O(k^{-\infty})$. Thus, $\Box^{(1)}_s\ddbar_s\mathcal{S}_{k}\equiv0\mod O(k^{-\infty})$. From Theorem~\ref{s3-t4-1}, we know that $\Box^{(1)}_s$ has semi-classical parametrices. Thus, $\ddbar_s\mathcal{S}_{k}\equiv0\mod O(k^{-\infty})$ so \eqref{coecaluII} follows. Now, we claim that 
\begin{equation} \label{coecaluIII}
\mbox{$\ddbar_z\bigr(i\Psi(z,w)+\phi(z)\bigr)$ vanishes to infinite order at $z=w$}.
\end{equation} 
We write $w=(w_1,\ldots,w_n)=(y_1,\ldots,y_{2n})=y$, $w_j=y_{2j-1}+iy_{2j}$, $j=1,\ldots,n$.
We assume that there exist $\alpha_0, \beta_0\in N^{2n}_0$, $\abs{\alpha_0}+\abs{\beta_0}\geq1$ and $(z_0,z_0)\in D\times D$, such that 
\begin{equation} \label{coecaluIV}
\rabs{\pr^{\alpha_0}_x\pr^{\beta_0}_{y}\Bigr(\ddbar_z\bigr(i\Psi(z,w)+\phi(z)\bigr)\Bigr)}_{(z_0,z_0)}=C_{\alpha_0,\beta_0}\neq0,
\end{equation}
and
\begin{equation} \label{coecaluV}
\rabs{\pr^{\alpha}_x\pr^\beta_y\Bigr(\ddbar_z\bigr(i\Psi(z,w)+\phi(z)\bigr)\Bigr)}_{(z_0,z_0)}=0,\ \ \mbox{if $\abs{\alpha}+\abs{\beta}<\abs{\alpha_0}+\abs{\beta_0}$},\ \ \alpha, \beta\in\mathbb N^{2n}_0.
\end{equation} 
From \eqref{coecaluIV}, \eqref{coecaluV} and since $b_0(z_0,z_0)\neq0$, $\Psi(z_0,z_0)=0$, we can check that 
\begin{equation} \label{coecaluVI}
\lim_{k\To\infty}k^{-n-1}\rabs{\pr^{\alpha_0}_x\pr^{\beta_0}_y\Bigr(\ddbar_s\bigr(e^{ik\Psi(z,w)}b(z,w,k)\bigr)\Bigr)}_{(z_0,z_0)}
=C_{\alpha_0,\beta_0}b_0(z_0,z_0)\neq0.
\end{equation}
On the other hand, since $\ddbar_s(e^{ik\Psi(z,w)}b(z,w,k))\equiv0\mod O(k^{-\infty})$, we can check that 
\begin{equation} \label{coecaluVII}
\lim_{k\To\infty}k^{-n-1}\rabs{\pr^{\alpha_0}_x\pr^{\beta_0}_y\Bigr(\ddbar_s\bigr(e^{ik\Psi(z,w)}b(z,w,k)\bigr)\Bigr)}_{(z_0,z_0)}=0.
\end{equation} 
We get thus a contradiction, hence the claim \eqref{coecaluIII} follows. Similarly, we have 
\begin{equation} \label{coecaluVIII}
\mbox{$\pr_w\bigr(i\Psi(z,w)+\phi(w)\bigr)$ vanishes to infinite order at $z=w$}.
\end{equation}
In particular, we have 
\begin{equation} \label{coecaluIX}
\mbox{$\ddbar_z\bigr(i\Psi(z,0)+\phi(z)\bigr)$  and $\pr_z\bigr(i\Psi(0,z)+\phi(z)\bigr)$ vanish to infinite order at $z=0$}.
\end{equation} 
Combining \eqref{coecaluIII}, \eqref{coecaluVIII}, \eqref{coecaluIX} with $\Psi(z,z)=0$, it is easy to check that for all $\alpha\in\mathbb N^n_0$, 
\begin{equation}\label{coecaluX}
\begin{split}
&\rabs{i\frac{\pr^{\abs{\alpha}}\Psi(z,0)}{\pr z^\alpha}}_{z=0}=-\rabs{i\frac{\pr^{\abs{\alpha}}\Psi(0,z)}{\pr z^\alpha}}_{z=0}=\frac{\pr^{\abs{\alpha}}\phi}{\pr z^\alpha}(0)=0\ \ \mbox{here we used}~\eqref{coecaluI},\\ 
&\rabs{i\frac{\pr^{\abs{\alpha}}\Psi(0,z)}{\pr\ol z^\alpha}}_{z=0}=-\rabs{i\frac{\pr^{\abs{\alpha}}\Psi(z,0)}{\pr\ol z^\alpha}}_{z=0}=\frac{\pr^{\abs{\alpha}}\phi}{\pr\ol z^\alpha}(0)=0\ \ \mbox{here we used}~\eqref{coecaluI}.
\end{split}
\end{equation} 
From \eqref{coecaluIX} and \eqref{coecaluX}, we deduce that for every $N\in\N_0$
\begin{equation} \label{coecaluXI} 
\begin{split}
\Psi(z,0)=i\phi(z)+O(\abs{z}^{N})\,,\:\:\Psi(0,z)=i\phi(z)+O(\abs{z}^{N}).
\end{split}
\end{equation} 
We claim that 
\begin{equation} \label{coecaluXI-I}
\mbox{$\ddbar_zb_j(z,w)$ and $\pr_wb_j(z,w)$ vanish to infinite order at $z=w$, for all $j\in\N_0 $.}
\end{equation}
In view of \eqref{coecaluIII}, we see that $\ddbar_z(i\Psi(z,w)+\phi(z))$ vanishes to infinite order at $z=w$. From this observation and \eqref{coecaluII}, we conclude that 
\begin{equation}\label{coecaluXII}
e^{ik\Psi(z,w)}\ddbar_zb(z,w,k)=H_k(z,w),
\end{equation}
where $H_k(z,w)\equiv0\mod O(k^{-\infty})$. We assume that there exist $\gamma_0, \delta_0\in\mathbb N^{2n}_0$, 
$\abs{\gamma_0}+\abs{\delta_0}\geq1$ and $(z_1,z_1)\in D\times D$, such that 
\[\rabs{\pr^{\gamma_0}_x\pr^{\delta_0}_y(\ddbar_zb_0(z,w))}_{(z_1,z_1)}=D_{\gamma_0,\delta_0}\neq0,\]
and
\[\rabs{\pr^{\gamma}_x\pr^{\delta}_y(\ddbar_zb_0(z,w))}_{(z_1,z_1)}=0\ \ \mbox{if $\abs{\gamma}+\abs{\delta}<\abs{\gamma_0}+\abs{\delta_0}$},\ \ \gamma, \delta\in\mathbb N^{2n}_0.\] 
From \eqref{coecaluXII}, we have 
\begin{equation}\label{coecaluXIII}
\rabs{\pr^{\gamma_0}_x\pr^{\delta_0}_y\Bigr(\ddbar_zb(z,w,k)\Bigr)}_{(z_1,z_1)}
=\rabs{\pr^{\gamma_0}_x\pr^{\delta_0}_y\Bigr(e^{-ik\Psi(z,w)}H_k(z,w)\Bigr)}_{(z_1,z_1)}.
\end{equation}
Since $\Psi(z_1,z_1)=0$, we have 
\begin{equation} \label{coecaluXIV}
\lim_{k\To\infty}k^{-n}\rabs{\pr^{\gamma_0}_x\pr^{\delta_0}_y\Bigr(e^{-ik\Psi(z,w)}H_k(z,w)\Bigr)}_{(z_1,z_1)}=0.
\end{equation} 
On the other hand, we can check that 
\begin{equation} \label{coecaluXV}
\lim_{k\To\infty}k^{-n}\rabs{\pr^{\gamma_0}_x\pr^{\delta_0}_y\Bigr(\ddbar_zb(z,w,k)\Bigr)}_{(z_1,z_1)}=D_{\gamma_0,\delta_0}\neq0.
\end{equation} 
From \eqref{coecaluXV}, \eqref{coecaluXIV} and \eqref{coecaluXIII}, we get a contradiction. 
Thus, $\ddbar_zb_0(z,w)$ vanishes to infinite order at $z=w$. Similarly, we can repeat the procedure above and conclude 
that $\ddbar_zb_j(z,w)$ and $\pr_wb_j(z,w)$ vanish to infinite order at $z=w$, $\forall j\in\N_0 $. The claim \eqref{coecaluXI-I} follows.  

Now, we are ready to calculate $b_1(0,0)$ and $b_2(0,0)$. We notice that 
\[b_0(z,z)=(2\pi)^{-n}\det\dot R^L(z).\]
From this and \eqref{coecaluXI-I}, it is easy to see that for all $\alpha\in\mathbb N^n_0$, 
\begin{equation}\label{coecaluXVI}
\begin{split}
&\rabs{\frac{\pr^{\abs{\alpha}}b_0(z,0)}{\pr z^\alpha}}_{z=0}=(2\pi)^{-n}\rabs{\frac{\pr^{\abs{\alpha}}(\det\dot R^L(z))}{\pr z^\alpha}}_{z=0},\: \rabs{\frac{\pr^{\abs{\alpha}}b_0(z,0)}{\pr\ol z^\alpha}}_{z=0}=0.
\end{split}
\end{equation} 
Since $\mathcal{S}_k\circ\mathcal{S}_k\equiv\mathcal{S}_k\mod O(k^{-\infty})$, we have 
\begin{equation} \label{coecaluXVII}
b(0,0,k)=\int_De^{ik(\Psi(0,z)+\Psi(z,0))}b(0,z,k)b(z,0,k)V_\Theta(z)d\lambda(z)+r_k,
\end{equation}
where $d\lambda(z)=2^ndx_1dx_2\cdots dx_{2n}$, $V_\Theta$ is given by \eqref{sa1-e10} and 
\[\lim_{k\to\infty}\frac{r_k}{k^N}=0,\ \ \forall N\geq0.\] 
We notice that since $b(z,w,k)$ is properly supported, we have 
\[b(0,z,k)\in C^\infty_0(D),\ \ b(z,0,k)\in C^\infty_0(D).\]
We apply the stationary phase formula of H\"{o}rmander \cite[Theorem~7.7.5]{Hor03}) to the integral in \eqref{coecaluXVII} and obtain 
( for the details see Hsiao~\cite[Section\,4]{Hsiao09}): 

\begin{thm} \label{t-coecalu} 
We have 
\begin{equation} \label{coecaluXVIII}
\begin{split}
b_1(0,0)&=(2\pi)^n(\det\dot{R}^L(0))^{-1}\Bigr(2b_0(0,0)b_1(0,0)\\
&\quad+\frac{1}{2}\triangle_0\bigr(V_\Theta b_0(0,z)b_0(z,0)\bigr)(0)
-\frac{1}{4}\triangle^2_0\bigr(\phi_1V_\Theta b_0(0,z)b_0(z,0)\bigr)(0)\Bigr)
\end{split}
\end{equation}
and 
\begin{equation} \label{coecaluXIX}
\begin{split}
b_2(0,0)&=(2\pi)^n(\det\dot{R}^L(0))^{-1}\Bigr(2b_0(0,0)b_2(0,0)+b_1(0,0)^2\\
&\quad+\frac{1}{2}\triangle_0\bigr(V_\Theta (b_0(0,z)b_1(z,0)+b_1(0,z)b_0(z,0))\bigr)(0)\\
&\quad-\frac{1}{4}\triangle^2_0\bigr(\phi_1V_\Theta (b_0(0,z)b_1(z,0)+b_1(0,z)b_0(z,0))\bigr)(0)\\
&\quad+\frac{1}{8}\triangle^2_0\bigr(V_\Theta b_0(0,z)b_0(z,0)\bigr)(0)
-\frac{1}{24}\triangle^3_0\bigr(\phi_1V_\Theta b_0(0,z)b_0(z,0)\bigr)(0)\\
&\quad+\frac{1}{192}\triangle^4_0\bigr(\phi_1^2V_\Theta b_0(0,z)b_0(z,0)\bigr)(0)\Bigr),
\end{split}
\end{equation}
where $\triangle_0=\sum^n_{j=1}\frac{1}{\lambda_j}\frac{\pr^2}{\pr\ol z_j\pr z_j}$, $\phi_1$ is as in \eqref{coecaluI} 
and $V_\Theta$ is as in \eqref{sa1-e10}.
\end{thm} 

From \eqref{coecaluXI-I}, \eqref{coecaluXVI} and \eqref{coecaluXVIII}, it is straightforward to see that (see Section 4.2 in \cite{Hsiao09}, for the details) 
\begin{equation}\label{coecaluXX}
\begin{split}
b_1(0,0)&=(2\pi)^{-n}\det\dot R^L(0)\Bigr(\frac{1}{4\pi}\widehat r(0)-\frac{1}{8\pi}r(0)\Bigr)\\
&=\frac{V_\omega(0)}{V_\Theta(0)}\Bigr(\frac{1}{4\pi}\bigr(\triangle_\omega\log V_\Theta\bigr)(0)-\frac{1}{8\pi}\bigr(\triangle_\omega\log V_\omega\bigr)(0)\Bigr),
\end{split}
\end{equation} 
where $\widehat r$ and $r$ are as in \eqref{sa1-e11} and $V_\omega$ is as in \eqref{sa1-e10}.
From this, \eqref{coeII} follows. 

Similarly, from \eqref{coeII} and \eqref{coecaluXI-I}, it is easy to see that for all $\alpha\in\mathbb N^n_0$, 
\begin{equation}\label{coecaluXXI}
\begin{split}
&\rabs{\frac{\pr^{\abs{\alpha}}b_1(z,0)}{\pr z^\alpha}}_{z=0}=
(2\pi)^{-n}\rabs{\frac{\pr^{\abs{\alpha}}\Bigr(\det\dot R^L(z)\bigr(\frac{1}{4\pi}\widehat r(z)-\frac{1}{8\pi}r(z)\bigr)\Bigr)}{\pr z^\alpha}}_{z=0},\\
&\rabs{\frac{\pr^{\abs{\alpha}}b_1(z,0)}{\pr\ol z^\alpha}}_{z=0}=0.
\end{split}
\end{equation}
From \eqref{coecaluXI-I}, \eqref{coecaluXXI} and \eqref{coecaluXIX}, it is straightforward to see that (see Section 4.3 in \cite{Hsiao09}, for the details) 
\begin{equation}\label{coecaluXXII}
\begin{split}
b_2(0,0)&=(2\pi)^{-n}\det\dot{R}^L(0)\Bigr(\frac{1}{128\pi^2}r^2-\frac{1}{32\pi^2}r\widehat r+\frac{1}{32\pi^2}(\widehat r)^2
-\frac{1}{32\pi^2}\triangle_{\omega}\widehat r
-\frac{1}{8\pi^2}\abs{R^{\det}_\Theta}^2_{\omega}\\
&\quad+\frac{1}{8\pi^2}\langle\,{\rm Ric\,}_{\omega}\,,R^{\det}_\Theta\,\rangle_{\omega}+\frac{1}{96\pi^2}\triangle_{\omega}r-
\frac{1}{24\pi^2}\abs{{\rm Ric\,}_{\omega}}^2_{\omega}+\frac{1}{96\pi^2}\abs{R^{TM}_{\omega}}^2_{\omega}\Bigr)(0),
\end{split}
\end{equation} 
where $\triangle_\omega$, $R^{\det}_\Theta$, ${\rm Ric\,}_{\omega}$ and $R^{TM}_{\omega}$ are as in \eqref{sa1-e9}, 
\eqref{sa1-e12}, \eqref{sa1-e15} and \eqref{sa1-e13} respectively, and $\langle\,\cdot,\cdot\,\rangle_\omega$, $\abs{\cdot}_\omega$ are as 
as in the discussion after \eqref{coeI} and $\abs{R^{TM}_{\omega}}^2_{\omega}$ is given by \eqref{sa1-e14}. From \eqref{coecaluXXII}, \eqref{coeIII} follows.

\section{Asymptotic upper bounds near the degeneracy set}\label{s:deg_set}

In this Section, we will use the heat equation expansion for $\Box^{(q)}_k$ of Ma-Marinescu \cite[\S\,1.6]{MM07} to get an asymptotic upper bound near the degenerate part of $L$. The goal of this Section is to prove
\eqref{s1-e3main}.

By the spectral theorem (see Davies~\cite[Th.\,2.5.1]{Dav95}), there exists a finite measure $\mu$ on $\mathbb S\times\mathbb N$, where $\mathbb S$ denotes the spectrum of $\Box^{(q)}_k$, and a unitary operator  
\begin{equation} \label{spectralmeasure}
U:L^2_{(0,q)}(M,L^k)\To L^2(\mathbb S\times\mathbb N,d\mu)
\end{equation}
with the following properties. If $h:\mathbb S\times \mathbb N\To\Real$ is the function $h(s,n)=s$, then the element $\xi$ of $L^2_{(0,q)}(M,L^k)$ lies in ${\rm Dom\,}\Box^{(q)}_k$ if and only if $hU(\xi)\in L^2$. We have 
$U\Box^{(q)}_kU^{-1}\varphi=h\varphi$ for all $\varphi\in U({\rm Dom\,}\Box^{(q)}_k)$.

We identify $L^2_{(0,q)}(M,L^k)$ with $L^2(\mathbb S\times\mathbb N,d\mu)$. Then the heat operator $\exp(-t\,\Box^{(q)}_k)$, $t>0$, is the operator on $L^2(\mathbb S\times\mathbb N,d\mu)$ given by
\[\begin{split}
\exp(-t\,\Box^{(q)}_k):L^2(\mathbb S\times\mathbb N,d\mu)&\To L^2(\mathbb S\times\mathbb N,d\mu)\,,\quad u(s,n)\mapsto e^{-st}u(s,n).
\end{split}\]

Since $\Box^{(q)}_k$ is elliptic, the distribution kernel of $\exp(-t\,\Box^{(q)}_k)$ is smooth (see \cite[Th.\,D.1.2]{MM07}).
Let
\[\exp(-t\,\Box^{(q)}_k)(x,y)\in\cC^\infty(M\times M, L^k_y\otimes\Lambda^qT^{*(0,1)}_yM\boxtimes L^k_x\otimes\Lambda^qT^{*(0,1)}_xM)\]
be the distribution kernel of $\exp(-t\,\Box^{(q)}_k)$ with respect to $(\cdot\,,\cdot)_k$.
That is,
\[(\exp(-t\,\Box^{(q)}_k)u)(x)=\int_M\exp(-t\,\Box^{(q)}_k)(x,y)u(y)dv_M(y),\ \ u\in L^2_{(0,q)}(M, L^k).\]
Let $s$ be a local section of $L$ over $\Td X$, where $\Td X\subset M$. Then on $\Td X\times\Td X$ we can write
\[\exp(-t\,\Box^{(q)}_k)(x,y)=\exp(-t\,\Box^{(q)}_k)_s(x,y)s(x)^ks^*(y)^k,\]
where $\exp(-t\,\Box^{(q)}_k)_s(x,y)\in\cC^\infty(\Td X\times\Td X, \Lambda^qT^*_yM\boxtimes \Lambda^qT^*_xM)$. For $u\in\Omega^{0,q}_0(\Td X,L^k)$ consider $\Td u\in\Omega^{0,q}_0(\Td X)$ with $u=s^k\Td u$. Then for $x\in\Td X$, 
\begin{equation} \label{e-s4-I}
\begin{split}
(\exp(-t\,\Box^{(q)}_k)u)(x)&=s(x)^k\int_M\exp(-t\,\Box^{(q)}_k)_s(x,y)\langle\,u(y)\,, s^*(y)^k\,\rangle dv_M(y)\\
&=s(x)^k\int_M\exp(-t\,\Box^{(q)}_k)_s(x,y)\Td u(y)dv_M(y)\ .
\end{split}
\end{equation}
For $x=y$, we can check that the function
\[\exp(-t\,\Box^{(q)}_k)_s(x,x)\in\cC^\infty(\Td X, \Lambda^qT^*_xM\boxtimes \Lambda^qT^*_xM)\]
is independent of the choices of local section $s$. We identify $\exp(-t\,\Box^{(q)}_k)_s(x,x)$ with $\exp(-t\,\Box^{(q)}_k)(x,x)$.
The trace of $\exp(-t\,\Box^{(q)}_k)(x,x)$ is given by
\[{\rm Tr\,}\exp(-t\,\Box^{(q)}_k)(x,x):=\sum^d_{j=1}\big\langle\exp(-t\,\Box^{(q)}_k)(x,x)e_{J_j}(x)\,,e_{J_j}(x)\big\rangle,\]
where $\{e_{J_j}(x)\}_{j=1}^d$ is an orthonormal basis of the space $\Lambda^qT^{*(0,1)}_xM$ with respect to $\langle\,\cdot\,,\cdot\,\rangle$, ${\rm dim\,}\Lambda^qT^{*(0,1)}_xM=d$.  
\begin{prop}\label{p:degenerate}
Fix $t>0$ and $N_0\geq1$. We have for $k$ large, 
\begin{equation} \label{e-s4-III-1}
{\rm Tr\,}\exp\big(-\frac{t}{k}\Box^{(q)}_k\big)(x,x)\geq(1-k^{-N_0}){\rm Tr\,}P^{(q)}_{k,k^{-N_0}}(x,x),\ \ \forall x\in M,
\end{equation}
where $P^{(q)}_{k,k^{-N_0}}(x,x)$ is as in \eqref{e:state}.
\end{prop} 

\begin{proof} 
First, we claim that for all $u\in\Omega^{0,q}_0(M,L^k)$,
\begin{equation}\label{e:degenerate}
\big(\exp\big(-\frac{t}{k}\Box^{(q)}_k\big)u,u\big)_k\geq(1-k^{-N_0})(P^{(q)}_{k,k^{-N_0}}u,u)_k\ .
\end{equation} 
We identify $L^2_{(0,q)}(M,L^k)$ with $L^2(\mathbb S\times\mathbb N,d\mu)$. Then 
\[\exp\big(-\frac{t}{k}\Box^{(q)}_k\big):
u(s,n)\in L^2(\mathbb S\times\mathbb N,d\mu)\mapsto e^{-s\frac{t}{k}}u(s,n)\] 
and 
\[P^{(q)}_{k,k^{-N_0}}:
u(s,n)\in L^2(\mathbb S\times\mathbb N,d\mu)\mapsto u(s,n)\mathds{1}_{[0,k^{-N_0}]}(s).\] 
For $u(s,n)\in L^2(\mathbb S\times\mathbb N,d\mu)$, we have 
\begin{equation}\label{e:degenerate1}
\begin{split}
&\big(\exp\big(-\frac{t}{k}\Box^{(q)}_k\big)u,u\big)_k
=\int_{\mathbb S\times\mathbb N}e^{-s\frac{t}{k}}\abs{u(s,n)}^2d\mu\geq\int_{\mathbb S\times\mathbb N}e^{-s\frac{t}{k}}\abs{u(s,n)}^2\mathds{1}_{[0,k^{-N_0}]}(s)d\mu\\
&\geq\int_{\mathbb S\times\mathbb N}\abs{u(s,n)}^2\mathds{1}_{[0,k^{-N_0}]}(s)d\mu\quad-\int_{\mathbb S\times\mathbb N}\abs{e^{-s\frac{t}{k}}-1}\abs{u(s,n)}^2\mathds{1}_{[0,k^{-N_0}]}(s)d\mu\\
&\geq\Bigr(1-\sup_{s\in[0,k^{-N_0}]}(1-e^{-s\frac{t}{k}})\Bigr)(P^{(q)}_{k,k^{-N_0}}u,u)_k.
\end{split}\end{equation} 
It is easy to see that fix $t>0$, we have $\sup_{s\in[0,k^{-N_0}]}(1-e^{-s\frac{t}{k}})\leq k^{-N_0}$ if $k$ large. 
From this observation and \eqref{e:degenerate1}, the claim \eqref{e:degenerate} follows. 

Now, fix $p\in M$ and let $s$ be a local section of $L$ defined in some open neighborhood $D$ of $p$, $\abs{s}_{h^L}^2=e^{-2\phi}$. Let $e_{J_1}(p),\ldots,e_{J_d}(p)$, be an orthonormal basis of $\Lambda^qT^{*(0,1)}_pM$ with respect to $\langle\,\cdot\,,\cdot\,\rangle$. Fix $l\in\set{1,\ldots,d}$. For each $j\in\mathbb N$, take $\chi_j\in\Omega^{0,q}_0(D,L^k)$ so that for every continuous operator $F:\cC^\infty(D,L^k\otimes\Lambda^qT^{*(0,1)}M)\To\cC^\infty(D,L^k\otimes\Lambda^qT^{*(0,1)}M)$ with smooth kernel $F(x,y)\in \cC^\infty(M\times M,L^k_y\otimes\Lambda^qT^{*(0,1)}_yM\boxtimes L^k_x\otimes\Lambda^qT^{*(0,1)}_xM)$, we have 
\[(F\chi_j, \chi_j)_k\To \langle F(p,p)e_{J_l}(p),e_{J_l}(p)\rangle,\ \ j\To\infty.\]
Then, we have
\[\begin{split}
&\big(\exp\big(-\frac{t}{k}\Box^{(q)}_k\big)\chi_j, \chi_j\big)_k\To\big\langle\exp(-\frac{t}{k}\Box^{(q)}_k)(p,p)e_{J_l}(p),e_{J_l}(p)\big\rangle,\ \ j\To\infty,\\
&\big(P^{(q)}_{k,k^{-N_0}}\chi_j, \chi_j\big)_k\To\big\langle P^{(q)}_{k,k^{-N_0}}(p,p)e_{J_l}(p),e_{J_l}(p)\big\rangle,\ \ j\To\infty.
\end{split}\] 
Combining this with \eqref{e:degenerate}, we conclude that 
\[\big\langle\exp\big(-\frac{t}{k}\Box^{(q)}_k\big)(p,p)e_{J_l}(p),e_{J_l}(p)\big\rangle
\geq(1-k^{-N_0})\big\langle P^{(q)}_{k,k^{-N_0}}(p,p)e_{J_l}(p),e_{J_l}(p)\big\rangle.\] 
Thus, 
\[{\rm Tr\,}\exp\big(-\frac{t}{k}\Box^{(q)}_k\big)(p,p)\geq(1-k^{-N_0}){\rm Tr\,}P^{(q)}_{k,k^{-N_0}}(p,p),\]
so \eqref{e-s4-III-1} follows.
\end{proof} 

\begin{thm}[{\cite[Th.\,1.6.1]{MM07}}] \label{tI-s4}
For each $t>0$ fixed and any $D\Subset M$, $m\in\mathbb N$, we have as $k\To\infty$,
\begin{equation} \label{e-s4-IV}
{\rm Tr\,}\exp(-\frac{t}{k}\Box^{(q)}_k)(x,x)=k^n(2\pi)^{-n}\Bigr(\sum_{j_1<j_2<\cdots<j_q}\exp(-t\sum^q_{i=1}a_{j_i}(x))\Bigr)\prod^n_{j=1}\frac{a_j(x)}{1-e^{-ta_j(x)}}+o(k^n),
\end{equation}
in the $\cC^m$ norm on $\cC^\infty(D,\Lambda^qT^{*(0,1)}M\boxtimes\Lambda^qT^{*(0,1)}M)$, where $a_1(x),\ldots,a_n(x)$ are the eigenvalues of $\dot{R}^L(x)$. Here we use the convention that  $\frac{a_j(x)}{1-e^{-ta_j(x)}}:=\frac{1}{t}$, if $a_j(x)=0$.
\end{thm}

From \eqref{e-s4-III-1} and \eqref{e-s4-IV}, we know that
\begin{equation} \label{e-s4-V}
(1-k^{-N_0}){\rm Tr\,}P^{(q)}_{k, k^{-N_0}}(x,x)\leq k^n(2\pi)^{-n}\Bigr(\sum_{j_1<j_2<\cdots<j_q}\exp(-t\sum^q_{i=1}a_{j_i}(x))\Bigr)\prod^n_{j=1}\frac{a_j(x)}{1-e^{-ta_j(x)}}+o(k^n),
\end{equation}
locally uniformly on $M$. 

Now, let $M_{\mathrm{deg}}$ be as in Theorem~\ref{s1-maindege}. Fix $t>1$, $t$ large and $x_0\in M_{\mathrm{deg}}$ and let $U$ be a small neighborhood of $x_0$ such that for every point $x\in U$, there is an eigenvalue $a_0(x)$ of $\dot{R}^L(x)$ such that $\abs{ta_0(x)}<1$.
Fix $p\in U$. Set
\[\iota(p)=\set{j\in\{1,\ldots,n\};\, \abs{a_j(p)t}<1,\ \ \mbox{where $a_1(p),\ldots,a_n(p)$ are the eigenvalues of $\dot{R}^L(p)$}}.\]
Fix $1\leq j_1<j_2<\cdots<j_q\leq n$. We have
\begin{equation} \label{e-s4-VI}
\begin{split}
&\exp(-t\sum^q_{i=1}a_{j_i}(p))\prod^n_{j=1}\frac{a_j(p)}{1-e^{-ta_j(p)}}
=\prod_{j_i\in\iota(p)}\frac{e^{-ta_{j_i}(p)}a_{j_i}(p)}{1-e^{-ta_{j_i}(p)}}\prod_{j_i\notin\iota(p)}\frac{e^{-ta_{j_i}(p)}a_{j_i}(p)}{1-e^{-ta_{j_i}(p)}}\\
&\quad\times\prod_{j\in\iota(p),j\notin\{j_1,\ldots,j_q\}}\frac{a_j(p)}{1-e^{-ta_j(p)}}
\prod_{j\notin\iota(p),j\notin\{j_1,\ldots,j_q\}}\frac{a_j(p)}{1-e^{-ta_j(p)}}.
\end{split}
\end{equation}
We observe that there is a constant $C>0$ such that
\begin{equation} \label{e-s4-VII}
\begin{split}
&\abs{\frac{x}{1-e^x}}\leq C,\ \ \abs{\frac{xe^x}{1-e^x}}\leq C,\ \ \forall x\in\Real,\ \ \abs{x}\leq1,\\
&\abs{\frac{1}{1-e^x}}\leq C,\ \ \abs{\frac{e^x}{1-e^x}}\leq C,\ \ \forall x\in\Real,\ \ \abs{x}>1.
\end{split}
\end{equation}
From \eqref{e-s4-VII} and \eqref{e-s4-VI}, it is straightforward to see that
\begin{equation} \label{e-s4-VIII}
\exp(-t\sum^q_{i=1}a_{j_i}(p))\prod^n_{j=1}\frac{a_j(p)}{1-e^{-ta_j(p)}}\leq
\prod_{j\in\iota(p)}\frac{C}{t}\prod_{j\notin\iota(p)}C\abs{a_j(p)},
\end{equation}
where $C$ is the constant as in \eqref{e-s4-VII}.

\begin{proof} [The proof of \eqref{s1-e3main}]
Let $\varepsilon>0$. Let $W\Subset M$ be any open set of $x_0$. Take $t>\max\set{C,1}$ large enough so that
\begin{equation} \label{s4-eproof}
(2\pi)^{-n}d\frac{C}{t}\Bigr(1+C\sup\set{\abs{a(x)};\, \mbox{$a(x)$: eigenvalue of $\dot{R}^L(x)$, $x\in W$}}\Bigr)^n<\frac{\varepsilon}{2},
\end{equation}
where $C$ is the constant as in \eqref{e-s4-VII} and $d={\rm dim\,}\Lambda^qT^{*(0,1)}_xM$. Let $U\Subset W$ be a small neighborhood of $x_0$ such that for every point $x\in U$, there is an eigenvalue $a_0(x)$ of $\dot{R}^L(x)$ such that $\abs{ta_0(x)}<1$. From \eqref{e-s4-VIII}, \eqref{e-s4-V} and \eqref{s4-eproof}, we see that
\[{\rm Tr\,}P^{(q)}_{k, k^{-N_0}}(x,x)\leq\frac{1}{1-k^{-\!N_0}}\frac{\varepsilon}{2}k^n+o(k^n),\ \ x\in U\,,\]
thus \eqref{s1-e3main} follows.
\end{proof}

Theorem~\ref{tI-s4} also holds on the case when the forms take values in $L^k\otimes E$, for a given holomorphic vector bundle $E$
over $M$. In this case the right side of \eqref{e-s4-IV} gets multiplied by $\operatorname{rank}(E)$, see Theorem 1.6.1 in~\cite{MM07}. From this observation, \eqref{s1-e3main} remains true with the same proof on the case when the forms take values in $L^k\otimes E$, for a given holomorphic vector bundle $E$ over $M$.




\section{Bergman kernel asymptotics for adjoint semi-positive line bundles}\label{adjoint_semipos}

In this Section we prove Theorem \ref{tmain-semi1}, i.\,e., the asymptotic expansion of the Bergman kernel of $L^k\otimes K_M$, where $L$ is a semi-positive line bundle over a complete K\"ahler manifold $M$ and $K_M$ is the canonical line bundle. The existence of the expansion \eqref{asymcanonical} follows immediately from Theorem \ref{t-semi1}, while the calculation of the coefficients is given at the end of this Section.

We assume that $(M,\Theta)$ is a complete K\"{a}hler manifold.
Let $K_M$ be the canonical line bundle over $M$. Then, $\Omega^{n,q}(M,L^k)=\Omega^{0,q}(M,L^k\otimes K_M)$. 
Let $\Box^{(0)}_{k,K_M}$ be the Gaffney extension of the Kodaira Laplacian acting on
$L^k\otimes K_M$. Then
\[{\rm Ker\,}\Box^{(0)}_{k,K_M}=\cH^0(M,L^k\otimes K_M)=\set{u\in L^2(M,L^k\otimes K_M);\, \ddbar_ku=0}.\]
Set
\[P^{(0)}_{k,K_M}:L^2(M,L^k\otimes K_M)\To\cH^0(M,L^k\otimes K_M)\]
be the orthogonal projection with respect to $(\cdot\,,\cdot)_k$. The goal of this Section is to prove that the kernel of
$P^{(0)}_{k,K_M}$ admits a full asymptotic expansion on the non-degenerate part of $L$.
We recall the following form of the $L^2$-estimates for $\overline\partial$ for semi-positive line bundles. Assume that $(L,h^L)$ is a semi-positive Hermitian line bundle over a complex manifold $M$. 
Let $g\in\Lambda^{n,1}T^*M\otimes L$.
For $x\in M$, we denote by $|g|_{R^{L}}(x)\in[0,\infty]$ the smallest constant such that $\langle g,g'\rangle^{2}(x)\leq|g|_{R^{L}}^{2}(x)\langle\sqrt{-1}R^{L}\wedge(\Theta\wedge)^{*} g',g'\rangle(x)$ for all $g'\in\Lambda^{n,1}T^*M\otimes L$. 
\begin{thm}[{\cite[Th.\,4.1]{De:82}}]\label{T:l2}
Let $(M,\Theta)$ be a complete K\"ahler manifold, $(L,h^L)$ be a semi-positive Hermitian line bundle over $M$.
Then for any form $g\in L_{(0,1)}^2(M,L\otimes K_M)$ satisfying
${\overline\partial}g=0$ and $\int_M|g|_{R^{L}}^{2}(x)\, dv_M(x)<\infty$ there exists $f\in L_{(0,0)}^2(M,L\otimes K_M)$ with $\overline\partial f=g$ and
$$\int_M|f|^2_{h^L}(x)\,dv_M(x)\leq\int_M|g|_{R^{L}}^{2}(x)\,dv_M\,.$$
\end{thm}
Denote by $\gamma(x)$ the smallest eigenvalue of the curvature $\sqrt{-1}R^L_{x}$ with respect to $\Theta_{x}$, for $x\in M$; the function
$\gamma:M\to[0,\infty)$ is continuous. Moreover, $|g|_{R^{L}}^{2}(x)\leq\gamma^{-1}(x)|g|_{h^L}^{2}(x)$,  
for any $x\in M$ and  $g\in\Lambda^{n,1}T^*M\otimes L$ (where $\gamma^{-1}:=\infty$ if $\gamma=0$). Therefore we deduce the following.
\begin{thm} \label{sing-t2-bis}
Let $(M,\Theta)$ be a complete K\"ahler manifold and $(L,h^L)$ be a smooth semi-positive line bundle over $M$. Let $D\Subset M(0)$ be a relatively compact open set.
There exists a constant $C_D>0$ such that for any $k>0$ and any $g\in\Omega^{0,1}_0(D,L^k\otimes K_M)$
satisfying $\ddbar_kg=0$ there exists $f\in\cC^\infty(M,L^k\otimes K_M)$ such that $\ddbar_kf=g$ and
\begin{equation}\label{e:l2_est}
\norm{f}^2\leq\frac{C_D}{k}\norm{g}^2\,.
\end{equation}
\end{thm}
We can actually take $C_{D}=\sup_{D}\gamma^{-1}$. We need 

\begin{lem}\label{l-canonical} 
Let $(M,\Theta)$ be a complete K\"ahler manifold and $(L,h^L)$ be a smooth semi-positive line bundle over $M$. Let $D\Subset M(0)$ be a relatively compact open set. Then $\Box^{(0)}_{k,K_M}$ has $O(k^{-n_0})$ small spectral on $D$.
\end{lem} 

\begin{proof} 
Let $u\in \cC^\infty_0(D,L^k\otimes K_M)$. We consider $\ddbar_ku\in\Omega^{0,1}(D,L^k\otimes K_M)$. From Theorem~\ref{sing-t2-bis}, we know that there exists $f\in\cC^\infty(M,L^k\otimes K_M)$ such that $\ddbar_kf=\ddbar_ku$ and
\begin{equation}\label{e-canonical1}
\norm{f}^2\leq\frac{C_D}{k}\norm{\ddbar_ku}^2,
\end{equation}
where $C_D>0$ is independent of $k$ and $u$.
We notice that $(I-P^{(0)}_{k,K_M})u$ has minimal $L^2$ norm of the set $\set{f\in\cC^\infty(M,L^k\otimes K_M)\bigcap L^2(M,L^k\otimes K_M) ;\, \ddbar_kf=\ddbar_ku}$. From this observation and \eqref{e-canonical1}, we conclude that 
\begin{equation}\label{e-canonical2}
\norm{(I-P^{(0)}_{k,K_M})u}^2\leq\frac{C_D}{k}\norm{\ddbar_ku}^2.
\end{equation}
It is easy to check that 
\[\norm{\ddbar_ku}^2\leq\norm{\Box^{(0)}_{k,K_M}u}\norm{(I-P^{(0)}_{k,K_M})u}.\]
Combining this with \eqref{e-canonical2}, we get $\norm{(I-P^{(0)}_{k,K_M})u}\leq\dfrac{C_D}{k}\norm{\Box^{(0)}_{k,K_M}u}$. 
Thus, $\Box^{(0)}_{k,K_M}$ has $O(k^{-n_0})$ small spectral on $D$. The lemma follows.
\end{proof}
Let $s$ be a local frame of $L$ on an open set $D\Subset M(0)$ and
$\abs{s}_{h^L}^2=e^{-2\phi}$. 
As in \eqref{e:localspectral1}, we consider the localized Bergman projection
\begin{equation} \label{semi6}
\begin{split}
\widehat P^{(0)}_{k,s,K_M}:L^2(D,K_M)\cap\mathscr E'(D,K_M)&\To L^2(D,K_M),\\
u&\mapsto e^{-k\phi}s^{-k}P^{(0)}_{k,K_M}(s^ke^{k\phi}u).
\end{split}
\end{equation}
From Lemma~\ref{l-canonical} and Theorem~\ref{t:localspectral}, we get one of the main results of this work:
\begin{thm} \label{t-semi1}
Let $(M,\Theta)$ be a complete K\"ahler manifold and $(L,h^L)$ be a smooth semi-positive line bundle over $M$. Let $D\Subset M(0)$ be a relatively compact open set
and $s$ be a local frame of $L$ on $D$. Then the localized Bergman projection $\widehat P^{(0)}_{k,s,K_M}$ satisfies
\[\widehat P^{(0)}_{k,s,K_M}\equiv\mathcal{S}_{k}\mod O(k^{-\infty})\]
on $D$, where $\mathcal{S}_{k}:\mathscr E'(D,K_M)\To\cC^\infty_0(D,K_M)$ is a smoothing operator and the distribution kernel $\mathcal{S}_{k}(z,w)\in\cC^\infty(D\times D,K_M\boxtimes K_M)$ of $\mathcal{S}_k$ satisfies
\[\mathcal{S}_{k}(z,w)\equiv e^{ik\Psi(z,w)}b(z,w,k)\mod O(k^{-\infty}),\]
with
\[\begin{split}
&b(z,w,k)\in S^{n}_{{\rm loc\,}}\big(1;D\times D, K_M\boxtimes K_M\big), \\
&b(z,w,k)\sim\sum^\infty_{j=0}b_j(z, w)k^{n-j}\text{ in }S^{n}_{{\rm loc\,}}
\big(1;D\times D, K_M\boxtimes K_M\big), \\
&b_j(z, w)\in\cC^\infty\big(D\times D, K_M\boxtimes K_M\big),\ \ j=0,1,2,\ldots,\\
&b_0(z,z)=(2\pi)^{-n}\det \dot R^L(z)\otimes\Id_{K_M}(z),\ \ \mbox{$\Id_{K_M}$ is the identity map on $K_M$},
\end{split}
\]
and $\Psi(z,w)$ is as in Theorem~\ref{s3-t1-bis} and Theorem~\ref{tII}.
\end{thm} 

From Theorem~\ref{t-semi1}, the existence of the asymptotic expansion \eqref{asymcanonical} for $L^k\otimes K_M$ follows immediately. 

We prove now the formulas \eqref{coeIV} for the coefficients. Le $s$ be a local frame of $L$ on an open set $D\Subset M(0)$.
We take local coordinates $(D,z)\cong(D,x)$
defined in $D$. 
Let $\mathcal{S}_k$ and $\mathcal{S}_k(\,\cdot\,,\cdot)\in\cC^\infty(D\times D,K_M\boxtimes K_M)$ be as in Theorem~\ref{t-semi1}. We may replace $\mathcal{S}_k$ by $\frac{1}{2}(\mathcal{S}_k+\mathcal{S}^*_k)$, where $\mathcal{S}^*_k$ is the formal adjoint of $\mathcal{S}_k$ with respect to $(\,\cdot\,,\cdot)$. Then, 
\begin{equation}\label{cancoeI}
\mathcal{S}^*_k=\mathcal{S}_k.
\end{equation} 
Let $e(z)$ be a local section of $K_M$ so that $\abs{e(z)}^2=(V_\Theta(z))^{-1}$, where $V_\Theta(z)$ is given by \eqref{sa1-e10}. 
Define the smooth kernels $\Td{\mathcal{S}}_k(\cdot,\cdot), \widehat{\mathcal{S}}_k(\cdot,\cdot)\in\cC^\infty(D\times D)$ by
\begin{equation}\label{cancoeII} 
\mathcal{S}_k(z,w)=e(z)\Td{\mathcal{S}}_k(z,w)e^*(w),\ \ 
\widehat{\mathcal{S}}_k(z,w)=\Td{\mathcal{S}}_k(z,w)V_\Theta(w)\,. 
\end{equation}
%
From Theorem~\ref{t-semi1}, we have 
\begin{equation}\label{cancoeIII-I}
\begin{split}
&\widehat{\mathcal{S}}_{k}(z,w)\equiv e^{ik\Psi(z,w)}\widehat b(z,w,k)\mod O(k^{-\infty}),\\
&\widehat b(z,w,k)\in S^{n}_{{\rm loc\,}}\big(1;D\times D\big), \\
&\widehat b(z,w,k)\sim\sum^\infty_{j=0}\widehat b_j(z, w)k^{n-j}\text{ in }S^{n}_{{\rm loc\,}}
\big(1;D\times D\big), \\
&\widehat b_j(z, w)\in\cC^\infty\big(D\times D\big),\ \ j=0,1,2,\ldots,\\
&\widehat b_0(z,z)=(2\pi)^{-n}V_\Theta(z)\det\dot R^L(z).
\end{split}
\end{equation}
Let $(\,,)_{d\lambda}$ be the inner product on $\cC^\infty_0(D)$ given by 
\[(u,v)_{d\lambda}=\int u(z)\ol{v(z)}\,d\lambda(z),\ \ u, v\in\cC^\infty_0(D),\]  
where $d\lambda(z)=2^ndx_1dx_2\cdots dx_{2n}$.
Let $\widehat{\mathcal{S}_k}$ be the continuous operator given by 
\[\begin{split}
\widehat{\mathcal{S}_k}:\cC^\infty_0(D)&\To\cC^\infty_0(D),\\
u&\mapsto\int\widehat{\mathcal{S}_k}(z,w)u(w)d\lambda(w).
\end{split}\]
Let $\widehat{\mathcal{S}_k}^{*,d\lambda}$ be the formal adjoint of $\widehat{\mathcal{S}_k}$ with respect to $(\,,)_{d\lambda}$. 
From \eqref{cancoeI}, \eqref{cancoeII} we can check that 
\begin{equation}\label{cancoeIV}
\widehat{\mathcal{S}_k}^{*,d\lambda}=\widehat{\mathcal{S}_k}.
\end{equation}
Since $\mathcal{S}_k^2\equiv\mathcal{S}_k\mod O(k^{-\infty})$, we can check that 
\begin{equation}\label{cancoeV}
(\widehat{\mathcal{S}_k})^2\equiv\widehat{\mathcal{S}}_k\mod O(k^{-\infty}).
\end{equation}
Moreover, it is obviously that 
\begin{equation}\label{cancoeVI}
\ddbar_s\widehat{\mathcal{S}_k}\equiv0\mod O(k^{-\infty}).
\end{equation} 
We recall that $\ddbar_s=\ddbar+k(\ddbar\phi)\wedge$\,. 

From \eqref{cancoeIV}, \eqref{cancoeV} and \eqref{cancoeVI}, we can repeat the procedure in Section~\ref{s:coeff} and conclude that (see \eqref{coecaluXX} and \eqref{coecaluXXII})
\begin{equation}\label{cancoeIX}
\begin{split}
&\widehat b_1(0,0)=V_\omega(0)\Bigr(-\frac{1}{8\pi}r(0)\Bigr),\\
&\widehat b_2(0,0)=V_\omega(0)\Bigr(\frac{1}{128\pi^2}r^2
+\frac{1}{96\pi^2}\triangle_{\omega}r-
\frac{1}{24\pi^2}\abs{{\rm Ric\,}_{\omega}}^2_{\omega}+\frac{1}{96\pi^2}\abs{R^{TM}_{\omega}}^2_{\omega}\Bigr)(0),
\end{split}
\end{equation} 
where $V_\omega$, $r$, $\triangle_\omega$, ${\rm Ric\,}_{\omega}$ and $R^{TM}_{\omega}$ are as in \eqref{sa1-e10}, \eqref{sa1-e11}, \eqref{sa1-e9}, 
\eqref{sa1-e15} and \eqref{sa1-e13} respectively, and $\langle\,\cdot,\cdot\,\rangle_\omega$, $\abs{\cdot}_\omega$ are as in the discussion after \eqref{coeI} and $\abs{R^{TM}_{\omega}}^2_{\omega}$ is given by \eqref{sa1-e14}.
From \eqref{cancoeII} and \eqref{cancoeIII-I}, we can check that for $b^{(0)}_{1,K_M}(z)$, $b^{(0)}_{2,K_M}(z)$ in \eqref{asymcanonical}, we have 
\[b^{(0)}_{1,K_M}(0)=\frac{1}{V_{\Theta}(0)}\widehat b_1(0,0)\Id_{K_M}(0),\ \ b^{(0)}_{2,K_M}(0)=
\frac{1}{V_{\Theta}(0)}\widehat b_2(0,0)\Id_{K_M}(0).\] 
Combining this with \eqref{cancoeIX} and observing that 
\[\frac{1}{V_\Theta(0)}V_\omega(0)=(2\pi)^{-n}\det\dot R^L(0),\] 
we obtain \eqref{coeIV}.

\begin{rem} \label{r:coincide} 
In~\cite[(4.1.9)]{MM07}, Ma-Marinescu gave a formula for $b^{(0)}_{1}$ in the presence of a twisting vector bundle $E$ (under the assumption that $L$ is positive everywhere). For $E=K_M$ the formula \cite[(4.1.9)]{MM07} reads:
\begin{equation}\label{e:lowermm}
b^{(0)}_{1,K_M}=(2\pi)^{-n}(\det\dot R^L)\frac{1}{8\pi}\Bigr(r-2\triangle_\omega\bigr(\log(\det\dot R^L)\bigr)
+4\sqrt{-1}\Lambda_\omega(R^{K_M})\Bigr)\Id_{K_M},
\end{equation}
where $\Lambda_\omega(R^{K_M})$ is given by $nR^{K_M}\wedge\omega^{n-1}=\Lambda_\omega(R^{K_M})\omega^n$.
Formula \eqref{coeIV} gives  
\begin{equation}\label{e:lowerhm}
b^{(0)}_{1,K_M}=(2\pi)^{-n}\det\dot{R}^L\Bigr(-\frac{1}{8\pi} r\Bigr)\Id_{K_M}.
\end{equation}
We show that the right-hand sides of \eqref{e:lowermm} and \eqref{e:lowerhm} are equal. By defintion, we have 
\begin{equation}\label{e:lowerhm2}
r=\triangle_\omega\log V_\omega,\ \ \det\dot R^L=(2\pi)^n\frac{V_\omega}{V_\Theta},
\end{equation}
where $V_\omega$ and $V_\Theta$ are given by \eqref{sa1-e10}. Using \eqref{e:lowerhm2}, \eqref{e:lowermm} becomes
\begin{equation}\label{e:lowerhm3}
b^{(0)}_{1,K_M}=(2\pi)^{-n}(\det\dot R^L)\frac{1}{8\pi}\Bigr(-r+2\triangle_\omega\log V_\Theta+4\sqrt{-1}\Lambda_\omega(R^{K_M})\Bigr)\Id_{K_M}.
\end{equation}
Moreover, it is straightforward to see that 
\begin{equation}\label{e:lowerhm4}
4\sqrt{-1}\Lambda_\omega(R^{K_M})=-2\triangle_\omega\log V_\Theta.
\end{equation}
Combining \eqref{e:lowerhm4} with \eqref{e:lowerhm3}, we conclude that our claim holds true.
\end{rem}
\section{Singular $L^2$-estimates}\label{s:sing_l2_est}

In Section \ref{S:mibk} we need a singular version of $L^2$ estimates. We assume that $(M,\Theta)$ is a compact Hermitian manifold and $(L,h^L)$ is a holomorphic line bundle over $M$, endowed with a singular Hermitian metric $h^L$.
We solve the $\ol\pr$-equation $\ddbar_kf=g$ for $(0,1)$ forms with values in $L^k$ with a rough $L^2$-estimate, namely  $\norm{f}^2\leq C_D k^N\norm{g}^2$ with $N>0$, instead of the estimate  $\norm{f}^2\leq\frac{C_D}{k}\norm{g}^2$ from \eqref{e:l2_est}. 

For a singular Hermitian metric $h^L$ on $L$ (see e.\,g.\ \cite[Def.\,2.3.1]{MM07}) the local weight with respect to a holomorphic frame $s:D\to L$ is a function $\phi\in L^1_{{\rm loc\,}}(D)$, $D\Subset M$, defined by
\[\abs{s}_{h^L}^2=e^{-2\phi}\in[0,\infty].\]
The curvature current $R^L$ is given locally by $R^L:=2\pr\ddbar\phi$ and does not depend on the choice of local frame $s$, is thus well-defined as a $(1,1)$ current on $M$.

We say that $\sqrt{-1}R^L$ is strictly positive if there exists $\varepsilon>0$ such that $\sqrt{-1}R^L\geq\varepsilon\Theta$, that is, $\sqrt{-1}R^L-\varepsilon\Theta$ is a positive current in the sense of Lelong (see e.\,g.\ \cite[Def.\,B.2.11]{MM07}). 
If $\sqrt{-1}R^L$ is strictly positive then $\phi$ is strictly psh on $D$ (in particular $\phi$ is bounded above on $D$).
The goal of this Section is to prove the following.

\begin{thm} \label{sing-t1} 
Let $(L,h^L)$ be a singular Hermitian holomorphic line bundle over a compact Hermitian manifold $(M,\Theta)$. We assume that $h^L$ is smooth outside a proper analytic set $\Sigma$ and 
\begin{equation} \label{s-sing-e0}
\sqrt{-1}R^L\geq\varepsilon\Theta,\ \ \varepsilon>0.
\end{equation}
Let $D\Subset M\setminus\Sigma$. Then, there exist $k_0>0$, $N>0$ and $C_D>0$, such that
for all $k\geq k_0$, and $g\in\Omega^{0,1}_0(D,L^k)$ with $\ddbar_kg=0$, 
there is $u\in\cC^\infty(M,L^k)$ such that $\ddbar_ku=g$ and
\begin{equation} \label{s-sing-e2}
\norm{u}^2_{h^k,\Theta}\leq k^NC_{D}\norm{g}^2_{h^k,\Theta},
\end{equation}
where $\norm{u}^2_{h^k,\Theta}:=\int_M\abs{u}^2_{h^k}dv_M$, $dv_M:=\frac{\Theta^n}{n!}$, and similarly for $\norm{g}^2_{h^k,\Theta}$.
\end{thm}

\begin{proof}
Let $\Theta_{\epsilon_0}$ be the generalized Poincar\'{e} metric on $M\setminus\Sigma$ (see~\cite[p.\,276]{MM07}).
Let $\mathcal T_{\epsilon_0}:=[(\Theta_{\epsilon_0}\wedge)^*,\pr\Theta_{\epsilon_0}]$ be the Hermitian torsion of $\Theta_{\epsilon_0}$. 
Let $R^{\det}_{\Theta_{\epsilon_0}}$ denote the curvature 
of the holomorphic line bundle $\Lambda^nT^{*(1,0)}M$ induced by $\Theta_{\epsilon_0}$.
By \cite[Lemma\,6.2.1]{MM07} we have 
\begin{equation}\label{poin1}
\begin{split}
&\mbox{$\Theta_{\epsilon_0}$ is a complete Hermitian metric of finite volume on $M\setminus\Sigma$},\\
&\mbox{$\Theta_{\epsilon_0}\geq c_0\Theta$ for some $c_0>0$},\\
&\mbox{$-C\Theta_{\epsilon_0}<\sqrt{-1}R^{\det}_{\Theta_{\epsilon_0}}<C\Theta_{\epsilon_0}$, $\abs{\mathcal T_{\epsilon_0}}_{\Theta_{\epsilon_0}}<C$},
\end{split}
\end{equation}
where $C>0$ is a constant and $\abs{\mathcal T_{\epsilon_0}}_{\Theta_{\epsilon_0}}$ is the norm with respect to $\Theta_{\epsilon_0}$. 
Moreover, by \cite[\S.\,6.2]{MM07} there is a Hermitian metric $h^L_{\epsilon_0}$ of $L$ on $M\setminus\Sigma$ such that $h^L_{\epsilon_0}$ is smooth on $M\setminus\Sigma$ and 
\begin{equation}\label{poin2}
h^L_{\epsilon_0}>h^L,\ \ \sqrt{-1}R^{L}_{\epsilon_0}>c\Theta_{\epsilon_0},
\end{equation}
where $c>0$ is a constant and $R^L_{\epsilon_0}$ is the curvature of $L$ induced by $h^L_{\epsilon_0}$. 

Let $s$ be a local frame of $L$ and define local weights $\phi_{\epsilon_0}$ and $\phi$ for $h^L_{\epsilon_0}$ and $h^L$
by $\abs{s}^2_{h^L_{\epsilon_0}}=e^{-2\phi_{\epsilon_0}}$, $\abs{s}^2_{h^L}=e^{-2\phi}$.
Let $\widehat h^k$ be the Hermitian metric on $L^k$ locally given by 
\[\abs{s}^2_{\widehat h^k}:=\exp({-2(\log k)\phi_{\epsilon_0}-2(k-\log k)\phi}).\] 
Since $h^L_{\epsilon_0}>h^L$, we have $\widehat h^k>h^k$.
Moreover, from \eqref{s-sing-e0} and \eqref{poin2}, we can check that
\begin{equation}\label{poin3}
\sqrt{-1}\widehat R^{L^k}>c(\log k)\Theta_{\epsilon_0},\ \ 
\end{equation}
where $\widehat R^{L^k}$ denotes the curvature of $L^k$ associated to $\widehat h^k$ and $c>0$ is the constant as in \eqref{poin2}. Let $(\,,)_{\widehat h^k,\Theta_{\epsilon_0}}$ denote the $L^2$ inner product on $\Omega^{0,q}_0(M\setminus\Sigma,L^k)$ with respect to $\widehat h^k$ and $\Theta_{\epsilon_0}$ as \eqref{toe2.2}. For $f\in\Omega^{0,q}_0(M\setminus\Sigma,L^k)$, we write $\norm{f}^2_{\widehat h^k,\Theta_{\epsilon_0}}:=(f,f)_{\widehat h^k,\Theta_{\epsilon_0}}$. 
Let $\widehat L^2_{(0,q)}(M\setminus\Sigma,L^k)$ be the completion of $\Omega^{0,q}_0(M\setminus\Sigma,L^k)$ with respect to $\norm{\cdot}_{\widehat h^k,\Theta_{\epsilon_0}}$.
Let 
\[\widehat\Box^{(1)}_k=\ddbar_k\ol{\pr}^{*}_k+\ol{\pr}^{*}_k\ddbar_k:{\rm Dom\,}\widehat\Box^{(1)}_k\subset\widehat L^2_{(0,1)}(M\setminus\Sigma,L^k)\To\widehat L^2_{(0,1)}(M\setminus\Sigma,L^k)\]
be the Gaffney extension of the Kodaira Laplacian with respect to $\widehat h^k$ and $\Theta_{\epsilon_0}$(see \eqref{Gaf1}).
Here $\ol{\pr}^{*}_k$ is the Hilbert space adjoint of $\ddbar_k$ with respect to $(\,,)_{\widehat h^k,\Theta_{\epsilon_0}}$.
From \eqref{poin1} and \eqref{poin3}, we 
can repeat the procedure in \cite[p.\,272--273]{MM07} and conclude that for $k$ large, we have
\begin{equation}\label{poin4}
\norm{g}^2_{\widehat h^k,\Theta_{\epsilon_0}}\leq\frac{1}{c(\log k)}\norm{\widehat\Box^{(1)}_kg}^2_{\widehat h^k,\Theta_{\epsilon_0}},
\end{equation}
for all $g\in\Omega^{0,1}_0(M\setminus\Sigma,L^k)$, where $c>0$ is a positive constant.
From this, we can repeat the method in \cite[p.\,272--273]{MM07} and conclude that $\widehat\Box^{(1)}_k$ has closed range in $\widehat L^2_{(0,1)}(M\setminus\Sigma,L^k)$, ${\rm Ker\,}\widehat\Box^{(1)}_k\bigcap\widehat L^2_{(0,1)}(M\setminus\Sigma,L^k)=\{0\}$ and 
there is a bounded operator $G_k:\widehat L^2_{(0,1)}(M\setminus\Sigma,L^k)\To{\rm Dom\,}\widehat\Box^{(1)}_k$ such that 
$\widehat\Box^{(1)}_kG_k=I$ on $\widehat L^2_{(0,1)}(M\setminus\Sigma,L^k)$, $G_k\widehat\Box^{(1)}_k=I$ on 
${\rm Dom\,}\widehat\Box^{(1)}_k$ and
\begin{equation} \label{poin4.1}
\norm{G_kg}_{\widehat h^k,\Theta_{\epsilon_0}}^2\leq\frac{1}{c(\log k)}\norm{g}_{\widehat h^k,\Theta_{\epsilon_0}}^2
\end{equation}
for $k$ large, for all $g\in\widehat L^2_{(0,1)}(M\setminus\Sigma,L^k)$, where $c>0$ is independent of $g$ and $k$, and 
\begin{align}
&G_k:\Omega^{0,1}(M\setminus\Sigma,L^k)\To\Omega^{0,1}(M\setminus\Sigma,L^k),\label{poin5}\\ 
&g=\widehat\Box^{(1)}_kG_kg=\ddbar_k\ol{\pr}^{\,*}_kG_kg,\ \ 
\mbox{if\ $\ddbar_kg=0$, $g\in\widehat L^2_{(0,1)}(M\setminus\Sigma,L^k)$}.\label{poin6}
\end{align} 
Now, fix $D\Subset M\setminus\Sigma$ and let $g\in\Omega^{0,1}_0(D,L^k)$ with $\ddbar_kg=0$ and set 
\[u=\ol{\pr}^{\,*}_kG_kg\in\Omega^{0,1}(M\setminus\Sigma,L^k)\bigcap\widehat L^2_{(0,0)}(M\setminus\Sigma,L^k).\] 
From \eqref{poin6} and \eqref{poin4.1}, it is not difficult to see that 
\begin{equation}\label{poin7}
\begin{split}
&\ddbar_k u=g\ \ \mbox{on $M\setminus\Sigma$},\\
&\norm{u}^2_{\widehat h^k,\Theta_{\epsilon_0}}\leq\frac{1}{c_1\sqrt{\log k}}\norm{g}^2_{\widehat h^k,\Theta_{\epsilon_0}},
\end{split}
\end{equation}
where $c_1>0$ is a constant independent of $g$ and $k$. Now, let's compare the norms $\norm{\cdot}_{\widehat h^k,\Theta_{\epsilon_0}}$ and $\norm{\cdot}_{h^k,\Theta}$. Let $s$ be a local section of $L$ on $D$ and $\abs{s}^2_{h^L_{\epsilon_0}}=e^{-2\phi_{\epsilon_0}}$, $\abs{s}^2_{h^L}=e^{-2\phi}$. Then,
\[\abs{s}^2_{\widehat h^k}=e^{-2k\phi}e^{2\log k(\phi-\phi_{\epsilon_0})}=\abs{s}^2_{h^k}e^{2\log k(\phi-\phi_{\epsilon_0})}.\]
Thus, on $D$, we have 
\begin{equation}\label{poin8}
\abs{s}^2_{\widehat h^k}<k^N\abs{s}^2_{h^k}, 
\end{equation}
where $N>\sup_{x\in D}\abs{2\phi(x)-2\phi_{\epsilon_0}(x)}$.
Thus, 
\begin{equation}\label{poin9}
\norm{g}^2_{\widehat h^k,\Theta_{\epsilon_0}}<\Td C_Dk^N\norm{g}^2_{h^k,\Theta},
\end{equation} 
where $\Td C_D>0$ is a constant independent of $g$ and $k$.
From $\widehat h^k>h^k$ and the second property in \eqref{poin1}, we have
$\norm{u}^2_{h^k,\Theta}<\Td c\norm{u}^2_{\widehat h^k,\Theta_{\epsilon_0}}$, where $\Td c>0$ is a constant independent of $k$ and $u$. Combining this with \eqref{poin9} and \eqref{poin7}, we obtain 
\begin{equation}\label{poin10}
\norm{u}^2_{h^k,\Theta}\leq C_Dk^N\norm{g}^2_{h^k,\Theta},
\end{equation}
where $C_D>0$ is a constant independent of $k$ and $g$. Note that $h^L$ is bounded away from zero and 
$\Sigma$ has Lebesgue measure zero. From this observation and 
\eqref{poin10}, we see that $u$ is $L^2$ integrable with respect to some smooth metric of $L$ over $M$.
Combining this with Skoda's Lemma (see Lemma~\ref{sing-l2} below), we get 
$\ddbar_ku=g$ on $M$ and $u\in\Omega^{0,1}(M,L^k)$. The theorem follows.
\end{proof}

We recall the following result of Skoda (see Demailly~\cite[Lemma\,7.3,\,Ch.\,VIII]{De:11}).
\begin{lem} \label{sing-l2}
Let $u\in\mathscr D'(M,L^k)$, $g\in\mathscr D'(M,L^k\otimes T^{*(0,1)}M)$. We assume that $u$ and $g$ are $L^2$ integrable with respect to some smooth metric of $L$ and $\Theta$ over $M$. If $\ddbar_ku=g$ on $M\setminus\Sigma$ in the sense of distributions,
then $\ddbar_ku=g$ on $M$ in the sense of distributions. 
\end{lem}

\section{Bergman kernel asymptotics for semi-positive line bundles}\label{s:exp_semipos}

In this Section we prove Theorem \ref{s1-sing-semi-main}. Let $(M,\Theta)$ a compact Hermitian manifold.
Assume that $(L,h^L)\to M$ is a smooth semi-positive line bundle which is positive at some point of $M$.
By Siu's criterion \cite[Th.\,2.2.27]{MM07} (see also Corollary~\ref{s1-c5}) we know that $L$ is big and $M$ is Moishezon. By \cite[Lemma\,2.3.6]{MM07}, $L$ admits a singular Hermitian metric $h^L_{{\rm sing\,}}$, smooth outside a proper analytic set $\Sigma$, and with strictly positive curvature current.

\begin{lem} \label{l1}
With the assumptions and notations above, let $D\Subset M\setminus\Sigma$ be an open set. 
Then, there exist $k_0>0$, $N>0$ and $C_D>0$, such that
for all $k\geq k_0$, and $g\in\Omega^{0,1}_0(D,L^k)$ with $\ddbar_kg=0$, there is $u\in\cC^\infty(M,L^k)$ such that $\ddbar_k u=g$ and
\[\norm{u}^2\leq k^NC_D\norm{g}^2.\]
\end{lem}

\begin{proof}
Let $\phi$ and $\widehat\phi$ denote local weights for $h^L$ and $h^L_{{\rm sing\,}}$ respectively.
Then, $\widehat\phi$ is smooth on $M\setminus\Sigma$ and bounded above.
We may assume that
\[\widehat\phi\leq\phi.\]
Let $\Td h^k$ be the Hermitian metric on $L^k$ induced by the local weight
\[\Td \phi:=(\log k)\widehat\phi+(k-\log k)\phi.\]
We can check that $\Td h^k$ is a strictly positive singular Hermitian metric, smooth
outside a proper analytic set $\Sigma$. 
Let $\norm{\cdot}_{\Td h^k}$ and $\norm{\cdot}_{h^k}$ denote the corresponding $L^2$ norms for sections with respect to 
$\Td h^k$ and $h^k$ respectively.
We can repeat the proof of Theorem~\ref{sing-t1} and conclude that for a given $g\in\Omega^{0,1}_0(D,L^k)$ with $\ddbar_kg=0$, there is $u\in\cC^\infty(M,L^k)$ such that $\ddbar_ku=g$ and
\begin{equation} \label{e1}
\norm{u}^2_{\Td h^k}\leq\frac{1}{c\sqrt{\log k}}\norm{g}^2_{\Td h^k},
\end{equation}
where $c>0$ is independent of $k$ and $g$.
Since $\widehat\phi\leq\phi$, we have
\begin{equation} \label{e2}
\norm{u}_{h^k}\leq\norm{u}_{\Td h^k}.
\end{equation}
On the other hand, we have
\begin{equation} \label{e3}
\begin{split}
\norm{g}^2_{\Td h^k}&=\int_D\abs{g}^2e^{-2(\log k)\widehat\phi-2(k-\log k)\phi}dv_M(x)\\
&\leq(\sup_{x\in D}e^{2(\log k)(\phi(x)-\widehat\phi(x))})\int_D\abs{g}^2e^{-2k\phi}dv_M(x)\\
&\leq k^N\norm{g}^2_{h^k},
\end{split}
\end{equation} 
where $N=\sup_{x\in D}2(\phi(x)-\widehat\phi(x))$. From \eqref{e2} and \eqref{e3}, the lemma follows.
\end{proof}

For a holomorphic line bundle $L$ over a compact Hermitian manifold $(M,\Theta)$ we set
$\operatorname{Herm}(L)=\big\{\text{singular Hermitian metrics on $L$} \big\}$\,,
\begin{multline*} \label{sing-semi-set1}
\mathcal{M}(L)=\big\{h^L\in\operatorname{Herm}(L);\, \text{$h^L$ is smooth outside a proper analytic set},\\ \text{$\sqrt{-1} R^L>\varepsilon\Theta$, $\varepsilon>0$} \big\}\,.
\end{multline*}
By \cite[Lemma\,2.3.6]{MM07}, $\mathcal{M}(L)\neq\emptyset$ under the hypotheses of Theorem~\ref{t1} below.
Set
\begin{equation}\label{sing-semi-set2}
M':=\set{p\in M;\, \mbox{$\exists$ $h^L\in\mathcal{M}(L)$ with $h^L$ smooth near $p$}}.
\end{equation} 

From Lemma~\ref{l1}, we can repeat the proof of Lemma~\ref{l-canonical} with minor changes and conclude the following.
\begin{thm}\label{t1}
Let $(M,\Theta)$ be a compact Hermitian manifold.
Let $(L,h^L)\to M$ be a Hermitian holomorphic line bundle with smooth Hermitian metric $h^L$ having semi-positive curvature and with $M(0)\neq\emptyset$. Let $D\Subset M'\bigcap M(0)$ 
be an open set, where $M'$ is 
given by \eqref{sing-semi-set2}. Then, $\Box^{(0)}_k$ has $O(k^{-n_0})$ small spectral gap on $D$.
\end{thm}
Let $s$ be a local frame of $L$ on an open set $D\Subset M$ and $\abs{s}_{h^L}^2=e^{-2\phi}$. 
We define the \emph{localized Bergman projection} (with respect to $s$) by
\begin{align} \label{1}
\widehat P^{(0)}_{k,s}:L^2(D)\cap\mathscr E'(D)&\To\cC^\infty_0(D),\nonumber \\
u&\To e^{-k\phi}s^{-k}P^{(0)}_k(s^ke^{k\phi}u).
\end{align}
That is, if $P^{(0)}_k(s^ke^{k\phi}u)=s^kv$ on $D$, then
$\widehat P^{(0)}_{k,s}u=e^{-k\phi}v$.

From Theorem~\ref{t1} and Theorem~\ref{t:localspectral}, we get the following result.
\begin{thm} \label{t-sing-semi}
Let $(M,\Theta)$ be a compact Hermitian manifold.
Let $(L,h^L)\to M$ be a Hermitian holomorphic line bundle with smooth Hermitian metric $h^L$ having semi-positive curvature and with $M(0)\neq\emptyset$.
Let $s$ be a local frame of $L$ on an open set $D\Subset M'\bigcap M(0)$. Then the localized Bergman projection $\widehat P^{(0)}_{k,s}$ satisfies
\[\widehat P^{(0)}_{k,s}\equiv\mathcal{S}_{k}\mod O(k^{-\infty})\]
on $D$, where $\mathcal{S}_{k}$ is as in Theorem~\ref{tII}.
\end{thm}
Theorem \ref{t-sing-semi} immediately implies Theorem \ref{s1-sing-semi-main}.
\section{Multiplier ideal Bergman kernel asymptotics. Proof of Theorem \ref{sing-main}}\label{S:mibk}

Let us first recall the notion of multiplier ideal sheaf. Let $M$ be a compact complex manifold and $\varphi\in L^1_{loc}(M,\mathbb R)$. The {\em Nadel multiplier ideal sheaf} $\cI(\varphi)\subset\cO_{M}$ is the ideal subsheaf of germs of holomorphic functions $f\in\cO_{M,x}$ such that $|f|^2e^{-2\varphi}$ is integrable with respect to the Lebesgue measure in local coordinates near $x$ for all $x\in M$.

\par Consider now a singular Hermitian metric $h^L$ on a holomorphic line bundle $L$ over $M$.
If $h^{L}_0$ is a smooth Hermitian metric on $L$ then $h^L=h^L_0e^{-2\varphi}$ for some function $\varphi\in L^1_{loc}(M,\mathbb R)$. The Nadel multiplier ideal sheaf of $h^L$ is defined by $\cI(h^L)=\cI(\varphi)$; the definition does not depend on the choice of $h^L_0$. Put
\begin{equation}\label{pcsing}
\cC^\infty(M,L\otimes\cI(h^L)):=\set{S\in\cC^\infty(M,L);\, \int_M\big\lvert S\big\rvert^2_{h^L}\,dv_M=\int_M\big\lvert S\big\rvert^2_{h^L_0}\,e^{-2\varphi}\,dv_M<\infty},
\end{equation}
where $\abs{\cdot}_{h^L}$ and $\abs{\cdot}_{h^L_0}$ denote the pointwise norms for sections induced by $h^L$ and $h^L_0$ respectively.
With the help of $h^L$ and the volume form $dv_M$ we can define an $L^2$ inner product on $\cC^\infty(M,L\otimes\cI(h^L))$:
\begin{equation}\label{l2sing}
(S,S')=\int_M\langle S,S'\rangle_{h^{L}_0}\,e^{-2\varphi}dv_M\,,\quad S,S'\in\cC^\infty(M,L\otimes\cI(h^L))\,.
\end{equation}

The singular Hermitian metric $h^L$ induces a singular Hermitian
metric $h^k=h^k_0e^{-2k\varphi}$ on $L^k$, $k>0$. We denote by $(\cdot\,,\cdot)_k$ the natural inner products on $\cC^\infty(M,L^k\otimes\cI(h^k))$ defined as in \eqref{l2sing} and by $L^2(M,L^k)$ the completion of $\cC^\infty(M,L^k\otimes\cI(h^k))$ with respect to $(\cdot\,,\cdot)_k$. The space of global sections in the sheaf $\cO(L^k)\otimes\cI(h^k)$ is given by
\begin{equation}\label{l2:mult}
\begin{split}
&H^0(M,L^k\otimes\cI(h^k))\\
&\quad=\Big\{
s\in\cC^\infty(M,L^k);\, \ddbar_ks=0,\,\int_M\big\lvert s\big\rvert^2_{h^k}\,dv_M=\int_M\big\lvert s\big\rvert^2_{h^k_0}\,e^{-2k\varphi}\,dv_M<\infty
\Big\}\,.
\end{split}
\end{equation}
Let
\begin{equation}\label{BK_multiplier}
P^{(0)}_{k,\cI}:L^2(M,L^k)\To H^0(M,L^k\otimes\cI(h^k))
\end{equation}
be the orthogonal projection. 

Now, we assume that $h^L$ is a strictly positive singular Hermitian metric on $L$, smooth outside a
proper analytic set $\Sigma$ of $M$. Let $L^2(M\setminus\Sigma,L^k)$ be the completion of $\cC^\infty_0(M\setminus\Sigma,L^k)$ with respect to $(\,\cdot\,,\cdot\,)_k$. We notice that $\Sigma$ is closed and has Lebesgue measure zero. From this observation, it is straightforward to see that 
\begin{equation}\label{pcsingI}
L^2(M\setminus\Sigma,L^k)=L^2(M,L^k).
\end{equation}
We consider the Gaffney extension $\Box^{(0)}_k$ of the Kodaira Laplacian $\ol{\pr}^*_k\ddbar_k$ on $M\setminus\Sigma$ (see \eqref{Gaf1}), where $\ol{\pr}^*_k$ is the formal adjoint of $\ddbar_k$ with respect to $(\,\cdot\,,\cdot\,)_k$ on $M\setminus\Sigma$. It is easy to see that $\Ker\Box^{(0)}_k=L^2(M\setminus\Sigma,L^k)\cap\Ker\ddbar_k$. The local weights of $h^L$ are strictly psh, so they are bounded above, hence elements in $L^2(M\setminus\Sigma,L^k)$ are locally square integrable with respect to smooth metrics on $M$ and $L$. Since holomorphic sections on $M\setminus\Sigma$ which are locally square integrable extend to holomorphic sections on $M$ (see Lemma~\ref{sing-l2}), we see that
\begin{equation}\label{pcsingII}
\Ker\Box^{(0)}_k=L^2(M,L^k)\cap\Ker\ddbar_k=H^0(M,L^k\otimes\cI(h^k)).
\end{equation}
Let 
\[P^{(0)}_k:L^2(M\setminus\Sigma,L^k)\To{\rm Ker\,}\Box^{(0)}_k\]
be the Bergman projection. From \eqref{pcsingI} and \eqref{pcsingII}, we see that 
\begin{equation}\label{pcsingIII}
P^{(0)}_{k,\cI}=P^{(0)}_k\ \ \mbox{on $L^2(M,L^k)=L^2(M\setminus\Sigma,L^k)$}. 
\end{equation} 

From Theorem~\ref{sing-t1}, we can repeat the proof of Lemma~\ref{l-canonical} and conclude that 

\begin{thm}\label{t-pcsing}
With the notations and assumptions above. Let $D\Subset M\setminus\Sigma$. Then, $\Box^{(0)}_k$ has $O(k^{-n_0})$ small spectral gap on $D$.
\end{thm}

Let $s$ be a local frame of $L$ on an open set $D\Subset M\setminus\Sigma$ and
$\abs{s}_{h^L}^2=e^{-2\phi}$. Then, $\phi$ is smooth on $D$ and $\pr\ddbar\phi$ is positive defined at each point of $D$. 
Let us denote by
\begin{equation} \label{s-sing-e11}
\widehat P^{(0)}_{k,s,\cI}:L^2(D)\cap\mathscr E'(D)\longrightarrow L^2(D)\,,\quad u\longmapsto e^{-k\phi}s^{-k}P^{(0)}_{k,\cI}(s^ke^{k\phi}u).
\end{equation}
the localized (multiplier ideal) Bergman projection. 

From Theorem~\ref{t-pcsing}, Theorem~\ref{t:localspectral} and \eqref{pcsingIII}, we get one of the main results of this work
\begin{thm} \label{sing-t5}
Let $(L,h^L)$ be a singular Hermitian holomorphic line bundle with strictly positive curvature current over a compact Hermitian manifold $(M,\Theta)$. We assume that $h^L$ is smooth outside a proper analytic set
$\Sigma$\,.
Let $s$ be a local frame of $L$ on an open set $D\Subset M\setminus\Sigma$. Then the localized
multiplier ideal Bergman projection $\widehat P^{(0)}_{k,s,\cI}$ (see \eqref{s-sing-e11}) satisfies
\[\widehat P^{(0)}_{k,s,\cI}\equiv\mathcal{S}_{k}\mod O(k^{-\infty})\]
on $D$, where $\mathcal{S}_{k}$ is as in Theorem~\ref{tII}.
\end{thm}

From Theorem~\ref{sing-t5}, we get Theorem~\ref{sing-main}.

\section{Further applications}\label{s:appl}

In this Section we collect further applications of the methods developed here.
In Section \ref{s:small} we show the existence of manifolds and line bundles whose Kodaira-Laplace operator has no $O(k^{-n_0})$ small spectral gap. In Section \ref{s:bouche} we show that under an integral condition (due to Bouche) on the first eigenvalue of the curvature, the asymptotic expansion of the Bergman kernel of a semi-positive line bundle holds.
In Section \ref{s:ample} we apply our results to prove a result of Berman about the Bergman kernel associated to an arbitrary semi-positive Hermitian metric on an ample line bundle. 
In Section \ref{s:exp_forms} we give a local version of the Bergman kernel expansion for $q$-forms.
In Section \ref{s:holo_morse} we obtain precise semiclassical estimates for the dimension of the spectral spaces of the Kodaira Laplacian. Using them one obtains immediately the holomorphic Morse inequalities of Demailly.
Finally, we prove in Section \ref{s:tian} a version of Tian's theorem about the convergence of the induced Fubini-Study metrics in the case of singular metrics on a big line bundle. This implies the equidistribution of the zeros of sections in the high tensor powers twisted with the Nadel ideal sheaves.

\subsection{Existence of ``small'' eigenvalues of the Kodaira Laplacian}\label{s:small}
The hypothesis on the existence of a $O(k^{-n_0})$ small spectral gap was of central importance in our approach. It is interesting to know if there is a compact complex manifold $M$ and a holomorphic line bundle $L$ over $M$ such that
the associated
Kodaira Laplacian does not exhibit such a spectral gap.
We will construct a compact manifold and a holomorphic line bundle $L$ over $M$ such that the associated Kodaira Laplacian $\Box^{(q)}_k$ has non-vanishing eigenvalues of order $O(k^{-\infty})$.
\begin{thm} \label{s1-mainex}
Let $0\leq q\leq n$, $q\in\mathbb N_0$. There exists a compact complex manifold $M$ of dimension $n$ and a holomorphic line bundle $L$ over $M$ such that for
\[\lambda_k:={\rm inf\,}\set{\lambda;\, \lambda:\ \mbox{non-zero eigenvalues of $\Box^{(q)}_k$}},\]
we have for every $N>0$
\[\lim_{k\To\infty}k^N\lambda_k=0\,.\]
\end{thm}  
Let $S$ be a compact Riemann surface with a smooth Hermitian metric. Let $(L_0,h^{L_0})$ be a holomorphic line bundle over $S$. We assume that $\sqrt{-1}{R}^{L_0}$ is positive. It is not difficult to see that $L_0$ admits another smooth Hermitian fiber
metric $\Td h^{L_0}$ such that the associated curvature form $\sqrt{-1}{\Td R}^{L_0}$ is positive on $S_+\subset S$, negative on $S_-\subset S$ and
degenerate on $S_0\subset S$, where $S=S_+\bigcup S_-\bigcup S_0$, $S_+, S_-$ contain non-empty open subsets of $S$. 

Let $M_1$ be a compact complex manifold of dimension $n-1$ with a smooth Hermitian metric 
and let $(L_1,h^{L_1})$ be a
holomorphic line bundle over $M_1$. We assume that $\sqrt{-1}{R}^{L_1}$ is non-degenerate of constant signature $(n_-,n_+)$,
$n_-+n_+=n-1$, at each point of $M_1$. Put
\[M:=M_1\times S,\ \ L:=L_1\otimes L_0.\]
Then, $M$ is a compact complex manifold of dimension $n$ and $L$ is a holomorphic line bundle over $M$. The Hermitian metrics
on $M_1$ and $S$ induce a Hermitian metric $\langle\,\cdot\,,\cdot\,\rangle$ on $M$. 
Consider the metric $h^L=h^{L_0}\otimes h^{L_1}$ on $L$; then the associated curvature
$\sqrt{-1}{R}^L$ is non-degenerate of constant signature $(n_-,n_++1)$ at each point of $M$. Similarly, setting $\Td h^L=\Td h^{L_0}\otimes h^{L_1}$, the associated curvature $\sqrt{-1}\Td R^L$ is non-degenerate of constant
signature $(n_-,n_++1)$ on $M_+\subset M$, non-degenerate of constant signature $(n_-+1,n_+)$ on $M_-\subset M$ and
degenerate on $M_0\subset M$, where $M=M_-\bigcup M_+\bigcup M_0$, $M_-, M_+$ contain non-empty open subsets of $M$.
First, we need

\begin{lem} \label{sE-l1}
Under the notations above let $q=n_-$\,. Then
\[{\rm dim\,}\cH^q(M,L^k)=(-1)^q\frac{k^n}{n!}\Big(\int_{M_+}\big(\tfrac{\sqrt{-1}}{2\pi}{\Td R}^L\big)^n+
\int_{M_-} \big(\tfrac{\sqrt{-1}}{2\pi}{\Td R}^L\big)^n\Big)
+o(k^n)\,,\quad k\to\infty\,.\]
\end{lem}

\begin{proof}
Note that $L$ admits a smooth Hermitian fiber metric such that the induced curvature is non-degenerate of constant signature
$(n_-,n_++1)$ at each point of $M$. From this observation and Andreotti-Grauert vanishing theorem, we know that if $k$ large, then
\begin{equation} \label{sE-e1}
\cH^j(M,L^k)=0\ \ \mbox{if $j\neq n_-$}.
\end{equation}
From the Riemann-Roch-Hirzebruch theorem (see e.\,g.\ \cite[(4.1.10)]{MM07}), we see that
\begin{equation} \label{sE-e2}
\sum^n_{j=0}(-1)^j{\rm dim\,}\cH^j(M,L^k)=\frac{k^n}{n!}\int_M c_1(L)^n+O(k^{n-1}),
\end{equation}
where $c_1(L)$ is the first Chern class. Combining \eqref{sE-e2} with \eqref{sE-e1}, we have for $k$ large enough 
\begin{equation} \label{sE-e3}
{\rm dim\,}\cH^q(M,L^k)=(-1)^q\frac{k^n}{n!}\int_M c_1(L)^n+O(k^{n-1}).
\end{equation}
But $\tfrac{\sqrt{-1}}{2\pi}{\Td R}^L$ represents the Chern class so
\[\int_M {c_1(L)^n}=\int_M\big(\tfrac{\sqrt{-1}}{2\pi}{\Td R}^L\big)^n.\]
The lemma follows from \eqref{sE-e3}.
\end{proof}

The Hermitian fiber metric $\Td h^L$ induces a Hermitian fiber metric $\Td h^k$ on the $k$-th tensor power of $L$.
As before, let $\Box^{(q)}_k$ be the Kodaira Laplacian with values in $L^k$ associated to $\Td h^k$. 

\begin{thm} \label{sE-t1}
Under the notations above let $q=n_-$. Then, for any $N>2n$, we have
\[\lim_{k\To\infty}k^N\lambda_k=0.\]
\end{thm}

\begin{proof}
Fix $N_0>2n$. From Corollary~\ref{s1-c4} below and Lemma~\ref{sE-l1}, we know that
\[\begin{split}
{\rm dim\,}\cE^q_{k^{-\!N_0}}(M,L^k)&=(-1)^q\frac{k^n}{n!}\int_{M_+}\big(\tfrac{\sqrt{-1}}{2\pi}{\Td R}^L\big)^n+o(k^n)\\
&>(-1)^q\frac{k^n}{n!}\Bigr(\int_{M_+}\big(\tfrac{\sqrt{-1}}{2\pi}{\Td R}^L\big)^n+
\int_{M_-}\big(\tfrac{\sqrt{-1}}{2\pi}{\Td R}^L\big)^n\Bigr)+o(k^n)\\
&>{\rm dim\,}\cH^q(M,L^k)+o(k^n).
\end{split}
\]
Thus, for $k$ large, we have
\[{\rm dim\,}\cE^q_{0<\lambda\leq k^{-\!N_0}}(M,L^k)>0,\]
where $\cE^q_{0<\lambda\leq k^{-\!N_0}}(M,L^k)$ denotes the spectral space spanned by the eigenforms of $\Box^{(q)}_k$ whose
eigenvalues are bounded by $k^{-N_0}$ and $>0$. We notice that since $M$ is compact, $\Box^{(q)}_k$ has a discrete
spectrum, each eigenvalues occurs with finite multiplicity. Thus, $\lambda_k\leq k^{-N_0}$ for $k$ large. The theorem follows.
\end{proof}

From Theorem~\ref{sE-t1}, we get Theorem~\ref{s1-mainex}.

\subsection{Bouche integral condition}\label{s:bouche}

Let $(L,h^L)$ be a semi-positive holomorphic line bundle over a compact Hermitian manifold $(M,\Theta)$ of dimension $n$. 
Let $0\leq\lambda_1(x)\leq\lambda_2(x)\leq\cdots$ be the eigenvalues of $\dot{R}^L(x)$.
We say that $(L,h^L)$ satisfies the Bouche integral condition~\cite{Bou95} if
\begin{equation}\label{e:bouche_int}
\int_M\lambda_1^{-6n}<\infty\,.
\end{equation}

If $(L,h^L)$ satisfies \eqref{e:bouche_int} then Bouche~\cite{Bou95} proved that
\[\inf\set{\lambda\in\operatorname{Spec}(\Box^{(q)}_k);\, \lambda\neq0}\geq k^{\frac{10n+1}{12n+1}},\]
for $k$ large. From this and Theorem~\ref{s1-main2}, we deduce

\begin{cor} \label{s1-c1}
Let $(L,h^L)$ be a semi-positive holomorphic line bundle over a compact Hermitian manifold $(M,\Theta)$ of dimension $n$. 
If $(L,h^L)$ satisfies \eqref{e:bouche_int} then 
\[P^{(0)}_k(x)\sim\sum^\infty_{j=0}k^{n-j}b^{(0)}_j(x)\ \ \mbox{locally uniformly on $M(0)$},\]
where $b^{(0)}_j(x)\in\cC^\infty(M(0))$, $j=0,1,2,\ldots$\,, are as in \eqref{s1-e2main}.
\end{cor}
\subsection{Asymptotics for arbitrary semi-positive metrics on ample line bundles}\label{s:ample}
We consider now the Bergman kernel of a metric with semi-positive curvature on an ample line bundle and recover the following result of Berman \cite{Be07}.
\begin{cor} \label{s1-c2}
Let $L$ be an ample line bundle over a compact projective manifold $M$ of dimension $n$.
We endow $M$ with a Hermitian metric $\Theta$ and $L$ with a Hermitian metric $h^L$ with semi-positive curvature. Then the Bergman kernel function associated to these metric data admits an asymptotic expansion
\[P^{(0)}_k(x)\sim\sum^\infty_{j=0}k^{n-j}b^{(0)}_j(x)\ \ \mbox{locally uniformly on $M(0)$},\]
where $b^{(0)}_j(x)\in\cC^\infty(M(0))$, $j=0,1,2,\ldots$\,, are as in \eqref{s1-e2main}.
\end{cor}
\begin{proof}
By a result due to Donnelly~\cite{Don03} there exist $C>0$ and $k_0\in\N$ such that for all $k\geq k_0$
\[\inf\set{\lambda\in\operatorname{Spec}(\Box^{(0)}_k);\, \lambda\neq0} \geq C\,.\]
In particular, $\Box^{(0)}_k$ has $O(k^{-n_0})$ small spectral gap on any open set $D\subset M(0)$. By applying Theorem \ref{s1-main2}
we immediately deduce the result.
\end{proof}

\subsection{Expansion for Bergman kernel on forms}\label{s:exp_forms} 
Let $(L,h^L)$ be a holomorphic line bundle over a compact Hermitian manifold $(M,\Theta)$ of dimension $n$. 
Given $q\in\mathbb N_0$, $0\leq q\leq n$, $\dot{R}^L$ is said to satisfy condition $Z(q)$ at $p\in M$ if $\dot{R}^L(p)$
has at least $n+1-q$ positive eigenvalues or at least $q+1$ negative eigenvalues. If $\dot{R}^L(p)$ is non-degenerate
of constant signature $(n_-, n_+)$, then $Z(q)$ holds at $p$ if and only if $q\neq n_-$.
It is well-known that if $Z(q-1)$ and $Z(q+1)$ hold
at each point of $M$, then $\Box^{(q)}_k$ 
has a ``large" spectral gap, i.e.\ there exists a constant $C>0$ such that for all $k$
we have 
\begin{equation}\label{large_spec_gap}
\inf\set{\lambda\in\operatorname{Spec}(\Box^{(q)}_k);\, \lambda\neq0}\geq Ck\,.
\end{equation}
This fact essentially follows from the $L^2$ method for $\overline\partial$ of H\"{o}rmander
(see H\"{o}rmander~\cite{Hor90} for the classical case and Sj\"{o}strand~\cite[Appendix]{Sjo95} for the semi-classical case). From this and Theorem~\ref{s1-main2}, we deduce the following local version of the results due to Catlin~\cite{Cat97}, Zelditch~\cite{Zel98}, Dai-Liu-Ma~\cite{DLM04a} (for $q=0$) and Berman-Sj\"{o}strand~\cite{BS05}, Ma-Marinescu~\cite{MM06} (for
$q>0$):
\begin{cor} \label{s1-c3} 
Let $(L,h^L)$ be a holomorphic line bundle over a compact Hermitian manifold $(M,\Theta)$ of dimension $n$. 
Given $q\in\mathbb N_0$, $0\leq q\leq n$. We assume that $Z(q-1)$ and $Z(q+1)$ hold at each point of $M$. If $\dot{R}^L$ is non-degenerate of constant signature
$(n_-,n_+)$ on an open set $D\subset M$, where $q=n_-$, then we have
\[P^{(q)}_k(x)\sim\sum^\infty_{j=0}k^{n-j}b^{(q)}_j(x)\ \ \mbox{locally uniformly on $D$},\]
where $b^{(q)}_j(x)\in\cC^\infty(D,\End(\Lambda^qT^{*(0,1)}M))$, $j=0,1,2,\ldots$\,, are as in \eqref{s1-e2main}.
\end{cor}
Let us illustrate Corollary \ref{s1-c3} in the case $q=0$: if at each point the curvature $R^L$
has either only positive eigenvalues or at least two negative eigenvalues, then the Bergman kernel
of the sections of $L^k$ has an asymptotic expansion on $M(0)$ as $k\to\infty$. 

\subsection{Holomorphic Morse inequalities}\label{s:holo_morse}

Let $(L,h^L)$ be a holomorphic line bundle over a compact Hermitian manifold $(M,\Theta)$ of dimension $n$. 
Since $M$ is compact, $\Box^{(q)}_k$ has a discrete
spectrum, each eigenvalues occurs with finite multiplicity.
From \eqref{s0-e4main}, \eqref{s0-e5main} and the Lebesgue dominated convergence theorem, we deduce the following.
\begin{cor} \label{s1-c4}
Let $(L,h^L)$ be a holomorphic line bundle over a compact Hermitian manifold $(M,\Theta)$ of dimension $n$. 
Given $q\in\mathbb N_0$, $0\leq q\leq n$. If $N_0>2n$, then
\[{\rm dim\,}\cE^q_{k^{-\!N_0}}(M,L^k)=k^n(2\pi)^{-n}\int_{M(q)}\abs{{\rm det\,}\dot{R}^L(x)}dv_M(x)+o(k^n).\]
\end{cor}
Fix $N_0\geq1$. Let $\cE^q_{0<\lambda\leq k^{-\!N_0}}(M, L^k)$ denote the spectral space spanned by the eigenforms of $\Box^{(q)}_k$ whose eigenvalues are bounded by $k^{-N_0}$ and $>0$. Since the operator $\ddbar_k\oplus\ddbar^*_k$ maps $\cE^q_{0<\lambda\leq k^{-\!N_0}}(M,L^k)$ injectively into $\cE^{q+1}_{0<\lambda\leq k^{-\!N_0}}(M,L^k)\oplus \cE^{q-1}_{0<\lambda\leq k^{-\!N_0}}(M,L^k)$. Thus,
\[{\rm dim\,}\cE^q_{0<\lambda\leq k^{-\!N_0}}(M,L^k)\leq{\rm dim\,}\cE^{q+1}_{0<\lambda\leq k^{-\!N_0}}(M,L^k)+
{\rm dim\,}\cE^{q-1}_{0<\lambda\leq k^{-\!N_0}}(M,L^k).\]
From this observation and Corollary~\ref{s1-c4}, we deduce:
\begin{cor} \label{s1-c5}
Let $(L,h^L)$ be a holomorphic line bundle over a compact Hermitian manifold $(M,\Theta)$ of dimension $n$. 
Given $q\in\mathbb N_0$, $0\leq q\leq n$. If $N_0>2n$, then
\[\begin{split}&{\rm dim\,}\cH^q(M,L^k)+{\rm dim\,}\cE^{q-1}_{0<\lambda\leq k^{-\!N_0}}(M,L^k)+
{\rm dim\,}\cE^{q+1}_{0<\lambda\leq k^{-\!N_0}}(M,L^k)\\
&\quad\geq
k^n(2\pi)^{-n}\int_{M(q)}\abs{{\rm det\,}\dot{R}^L(x)}dv_M(x)+o(k^n).\end{split}\]
In particular, we have
\begin{equation} \label{s1-tIweyl}
\begin{split}
&{\rm dim\,}\cH^q(M,L^k)\\
&\geq k^n(2\pi)^{-n}\Bigr(\int_{M(q)}\abs{{\rm det\,}\dot{R}^L(x)}dv_M(x)-
\int_{M(q-1)}\abs{{\rm det\,}\dot{R}^L(x)}dv_M(x)\\
&\quad-\int_{M(q+1)}\abs{{\rm det\,}\dot{R}^L(x)}dv_M(x)\Bigr)+o(k^n).
\end{split}
\end{equation} 
Hence, if $M(q-1)=\emptyset$, $M(q+1)=\emptyset$, then 
\begin{equation} \label{s1-tIweyl2}
{\rm dim\,}\cH^q(M,L^k)
=k^n(2\pi)^{-n}\Bigr(\int_{M(q)}\abs{{\rm det\,}\dot{R}^L(x)}dv_M(x)\Bigl)+o(k^n).
\end{equation} 
\end{cor}
By Corollary~\ref{s1-c4} and a straightforward application of the linear algebra result from Demailly~\cite[Lemma\,4.2]{De:85} or~\cite[Lemma\,3.2.12]{MM07} to the complex $(\cE^\bullet_{k^{-\!N_0}}(M,L^k),\ol\pr_k)$, we obtain the following fundamental result due to  Demailly's \cite[Th.\,0.1]{De:85}. 
We refer the reader to \cite[Ch.\,1--3]{MM07} for a thorough discussion of the holomorphic Morse inequalities.
\begin{cor}[strong holomorphic Morse inequalities] \label{s1-c5.1}
Let $(L,h^L)$ be a holomorphic line bundle over a compact Hermitian manifold $(M,\Theta)$ of dimension $n$. 
Then for any $q\in\{0,1,\ldots,n\}$ we have for $k\to\infty$
\begin{equation*}
\sum_{j=0}^q(-1)^{q-j}\dim\cH^j(M,L^k)\leq k^n(2\pi)^{-n}\sum^q_{j=0}(-1)^{q-j}\int_{M(j)}\abs{{\rm det\,}\dot{R}^L(x)}dv_M(x)+o(k^n)\,.
\end{equation*}
\end{cor}
Let us also give an example of a quite general holomorphic Morse inequalities on arbitrary complete K\"ahler manifolds.
\begin{cor}
Let $(M,\Theta)$ be a complete K\"ahler manifold and $(L,h^L)$ be a semi-positive Hermitian holomorphic line bundle on $M$. Then
\begin{equation}\label{micpl}
\liminf_{k\to\infty}k^{-n}\dim\cH^0(M,L^k\otimes K_M)\geq \frac{1}{n!}\int_{M}\big(\tfrac{\sqrt{-1}}{2\pi} R^L\big)^n\,.
\end{equation}  
\end{cor}
\begin{proof}
Let $\{S^k_j\}_{j=1}^{d_k}$ be an orthonormal basis of $\cH^0(M,L^k\otimes K_M)$, $d_k\in\N\cup\{\infty\}$. 
Then the Bergman kernel function is given by
${\rm Tr\,}P_{k,K_M}(x)=\sum_{j=1}^{d_k}|S^k_j(x)|^2$, $x\in X$,
where $|\cdot|$ denotes the pointwise norm in the metric ${h^k\otimes h^{K_M}}$. By integrating this relation we obtain 
\[
\dim\cH^0(M,L^k\otimes K_M)=d_k=\int_M {\rm Tr\,}P_{k,K_M}(x)\,dv_M\,.
\]
By \eqref{asymcanonical} we know that the sequence $k^{-n}{\rm Tr\,}P_{k,K_M}(x)$ converges pointwise on $M(0)$ to ${\rm Tr\,}b^{(0)}_{0,K_M}$ as $k\to\infty$. By Fatou's lemma we obtain
\begin{equation*}
\begin{split}
\liminf_{k\to\infty}k^{-n}\int_M {\rm Tr\,}P_{k,K_M}(x)\,dv_M&\geq\liminf_{k\to\infty}\int_{M(0)} k^{-n}{\rm Tr\,}P_{k,K_M}(x)\,dv_M\\
\geq\int_{M(0)} {\rm Tr\,}b^{(0)}_{0,K_M}(x)\,dv_M
&=\frac{1}{n!}\int_{M(0)}\big(\tfrac{\sqrt{-1}}{2\pi} R^L\big)^n=\frac{1}{n!}\int_{M}\big(\tfrac{\sqrt{-1}}{2\pi} R^L\big)^n\,.
\end{split}
\end{equation*}
Hence \eqref{micpl} follows. 
\end{proof}
Let us close with an amusing by-product of Theorem \ref{s1-main1}. Let $(L,h^L)$ be a holomorphic line bundle over a compact Hermitian manifold $(M,\Theta)$ of dimension $n$. Assume that $\dot{R}^L$ is non-degenerate of constant signature
$(n_-,n_+)$ at each point of $M$. From Theorem~\ref{s1-main1}, we see 
that if $q\neq n_-$, then $P^{(q)}_k(x)=O(k^{-N})$, for every $N\geq0$. Thus,
\[
{\rm dim\,}\cH^q(M,L^k)=O(k^{-N}),\ \ \forall N\geq0.
\]
Since ${\rm dim\,}\cH^q(M,L^k)$ is an integer, we obtain the Andreotti-Grauert coarse vanishing theorem
(see \cite[Th.\,1.5]{MM06}, \cite[Rem.\,8.2.6]{MM07}):
\begin{equation}\label{ag}
{\rm dim\,}\cH^q(M,L^k)=0\,,\quad\text{for $k$ large enough}.
\end{equation}
This proof uses just estimates of the spectral spaces. The original proof of Andreotti-Grauert was based on cohomology finiteness theorems for the disc bundle $L^*$. Ph. Griffiths gave a proof using the Bochner-Kodaira-Nakano formula. For a proof using Lichnerowicz formula and a comparison of methods, see \cite[Th.\,1.5]{MM06}, \cite[Rem.\,1.6]{MM06}.
Note that the above proof of \eqref{ag} provides a positive answer to a question of Bouche \cite{Bou99} whether one could get vanishing theorems by just using (heat or Bergman) kernel methods. 

\subsection{Tian's theorem and equidistribution of zeros}\label{s:tian}
Given a positive line bundle $L$ on a compact manifold $M$ one can consider the Kodaira embeddings $\Phi_k:M\to\mathbb{P}(H^0(M,L^k)^*)$ for large $k$, where $H^0(M,L^k)=\set{u\in\cC^\infty(M,L^k);\, \ddbar_ku=0}$. Denote by $\omega_{FS}$ the Fubini-Study metric on $\mathbb{P}(H^0(M,L^k)^*)$. Tian  
\cite[Th.\,A]{Tian} proved that $\frac1k\Phi_k^*(\omega_{FS})$ converges to the curvature $\frac{\sqrt{-1}}{2\pi} R^L$ as $k\to\infty$ in the $\cC^2$-topology. This answered a conjecture of S.~T.~Yau~\cite{Yau87}. Ruan \cite{Ruan98} proved the convergence in the $\cC^\infty$-topology and  
improved the estimate of the convergence speed. Both papers use the peak section method,  
based on $L^2$--estimates for ${\overline\partial}$. 
A proof of the convergence in the $\cC^0$-topology using the heat kernel  
appeared in Bouche \cite{Bou90}. Catlin \cite{Cat97} and Zelditch \cite{Zel98} deduced the convergence from the  
asymptotic expansion of the Bergman kernel. 

We will consider here a compact Hermitian manifold 
$(M,\Theta)$ and a big line bundle $L\to M$. Let $h^L$ be a strictly positive singular Hermitian metric on $L$, smooth outside a
proper analytic set $\Sigma$ of $M$. 
We endow $H^0(M,L^k\otimes\cI(h^k))$ with the $L^2$ scalar product \eqref{l2sing} induced by $h^L$ and $dv_M=\Theta^n/n!$. 
Consider the Kodaira map
\begin{equation}
\begin{split}
&\Phi_k:M\setminus B_k\to\mathbb{P}\big(H^0(M,L^k\otimes\cI(h^k))^*\big)\,,\\
&x\longmapsto\big\{s\in H^0(M,L^k\otimes\cI(h^k));s(x)=0\big\}\,,
\end{split}
\end{equation}
where $B_k$ is the base locus of $H^0(M,L^k\otimes\cI(h^k))$.
To the Hermitian structure \eqref{l2sing} corresponds a Fubini-Study metric $\omega_{FS}$ on 
$\mathbb{P}\big(H^0(M,L^k\otimes\cI(h^k))^*\big)$, defined as the curvature of the hyperplane line bundle (see e.\,g.\ \cite[(5.1.3)]{MM07}). The induced Fubini-Study metric is the metric 
$\frac{1}{k}\Phi_k^*(\omega_{FS})$ on $M\setminus B_k$. 
\begin{thm}\label{tian_big}  Let $(M,\Theta)$ be a compact Hermitian manifold and let $L\to M$ be a big line bundle. Let $h^L$ be a strictly positive singular Hermitian metric on $L$, smooth outside a
proper analytic set $\Sigma$ of $M$. 
Then for any compact set $K\subset M\setminus\Sigma$, there exists $k_0$ such that for $k\geq k_0$ the base locus $B_k$ of $H^0(M,L^k\otimes\cI(h^k))$ is disjoint of $K$.  
Moreover, for any $\ell\in\N$, there exists $C_{\ell,K}>0$ independent of $k$ such that  
for $k\geq k_0$ the following holds
\begin{equation}\label{sz1,1} 
\Big|\frac{1}{k}\Phi_k^*(\omega_{FS})-\frac{\sqrt{-1}}{2\pi} R^L\Big|_{\cC^\ell(K)}
\leqslant\frac{C_{\ell,K}}{k}\,\cdot 
\end{equation} 
\end{thm}
\begin{proof}
Let $\{S^k_j\}_{j=1}^{m_k}$ be an orthonormal basis of $H^0(M,L^k\otimes\cI(h^k))$. Then the multiplier Bergman kernel function \eqref{sing-e0-2} is given by
\[
P^{(0)}_{k,\cI}(x)=\sum_{j=1}^{m_k}|S^k_j(x)|_{h^k}^2,\;\;\;x\in M\setminus\Sigma\,.
\]
Let $K\subset M\setminus\Sigma$ be a compact set.
The expansion \eqref{s1-e3mainterz} yields $P^{(0)}_{k,\cI}(x)=b^{(0)}_0(x)k^n+o(k^n)$, as $k\to\infty$, uniformly on $K$. Since $\inf_{x\in K}b^{(0)}_0(x)>0$, there exists $k_0$ such that for all $k\geq k_0$ we have $\inf_{x\in K}P^{(0)}_{k,\cI}(x)>0$. Hence $K\cap B_k=\emptyset$, for all $k\geq k_0$.

For a local holomorphic frames $e_L$ 
of $L$ over an open set $U\subset M$, we set 
$S^k_j=f^k_je^{\otimes k}_L$, where $f^k_j\in\cO(U)$.
The choice of the basis $\{S^k_j\}_{j=1}^{m_k}$ induces an isometric identification 
$\mathbb{P}\big(H^0(M,L^k\otimes\cI(h^k))^*\big)\cong\mathbb{P}^{m_k-1}$ and in terms of this identification $\Phi_k$ has the form 
\[\Phi_k:M\setminus B_k\to\mathbb{P}^{m_k-1}\,,\;\;  \Phi_k(x)=[f_1^k(x),\ldots,f_{m_k}^k(x)]\,,
\]
hence 
\[\Phi_k^*(\omega_{FS})=\frac{\sqrt{-1}}{2\pi}\,\partial\overline\partial\log
\Big(\sum_{j=1}^{m_k}|f^k_j(x)|^2\Big)\,\;\;\text{on $U\setminus B_k$}\,,\]
thus
\begin{equation}\label{ell00,1}
\frac{1}{k}\Phi_k^*(\omega_{FS})-\frac{\sqrt{-1}}{2\pi}R^L=
-\frac{\sqrt{-1}}{2\pi k}\,\overline\partial\partial\log P^{(0)}_{k,\cI}(x)\,,\;\;\text{on $M\setminus B_k$}\,.
\end{equation}
The expansion \eqref{s1-e3mainterz} shows that $\overline\partial\partial\log P^{(0)}_{k,\cI}(x)=O(1)$, for $k\to\infty$ in the $\cC^{\ell}$-topology, since $\log P^{(0)}_{k,\cI}(x)=\log k^n+\log(b^{(0)}_{0}(x)+O(\frac1k))$ in the $\cC^{\ell+2}$-topology. Hence \eqref{sz1,1} is a consequence of \eqref{ell00,1}.
\end{proof}

\par An important application of the convergence of the Fubini-Study currents is the study of the asymptotic distribution of zeros of random holomorphic sections. After the pioneering work of Nonnenmacher-Voros
\cite{NoVo:98}, general methods were developed by Shiffman-Zelditch \cite{ShZ99} and Dinh-Sibony \cite{DS06b} to describe the asymptotic distribution of zeros of random holomorphic sections of a positive line bundle over a projective manifold endowed with a smooth positively curved metric. The paper \cite{DS06b} gives moreover very good convergence speed
and applies to general measures (e.\,g.\ equidistribution of complex zeros of homogeneous polynomials with real coefficients). Some important technical tools for higher dimension used in the previous works were introduced
by Forn{\ae}ss-Sibony \cite{FS95}. 
For the non-compact setting and the case of singular Hermitian metrics see \cite{CM11,CM12,CM13,DMS11}.

Using the results of the present paper we can further generalize some results from \cite{CM11}.

We define positive $(1,1)$ currents $\gamma_k$ on $M$, called Fubini-Study currents, by 
\begin{equation}\label{e:gammap}
\gamma_k\mid_{_U}=\frac{\sqrt{-1}}{2\pi}\,\partial\overline\partial\log
\Big(\sum_{j=1}^{m_k}|f^k_j(x)|^2\Big)\,.
\end{equation}
Then
\begin{equation}\label{ell00,2}
\frac{1}{k}\gamma_k-\frac{\sqrt{-1}}{2\pi}R^L=
-\frac{\sqrt{-1}}{2\pi k}\,\overline\partial\partial\log P^{(0)}_{k,\cI}(x)\,,\;\;\text{on $M\setminus\Sigma$}\,.
\end{equation}
This shows that the definition \eqref{e:gammap} of $\gamma_k$ is independent of the choice of holomorphic frame $e_L$ and basis $\{S^k_j\}_{j=1}^{m_k}$.

Let $\lambda_k$ be the normalized surface measure on the unit sphere ${\mathcal S}^k$ of $H^0(M,L^k\otimes\cI(h^k))$, defined in the natural way by using a fixed orthonormal basis.
Consider the probability space ${\mathcal S}_\infty=\prod_{p=1}^\infty{\mathcal S}^p$ endowed with the probability measure $\lambda_\infty=\prod_{p=1}^\infty\lambda_p$. 
Denote by $[S=0]$ the current of integration (with multiplicities) over the analytic hypersurface $\{S=0\}$ determined by a nontrivial section $S\in H^0(M,L^k\otimes\cI(h^k))$.
\begin{cor}\label{equi_big}
Let $(M,\Theta)$ be a compact Hermitian manifold and let $L\to M$ be a big line bundle. Then we have in the weak sense of currents on $M$
\begin{gather*}
\lim_{k\to\infty}\frac{1}{k}\gamma_k=\frac{\sqrt{-1}}{2\pi}R^L\,,\\
\lim_{k\to\infty}\,\frac{1}{k}\,[\sigma_k=0]=\frac{\sqrt{-1}}{2\pi}R^L\,,\;for\;\lambda_\infty\text{-a.e. sequence }\{\sigma_k\}_{k\geq1}\in{\mathcal S}_\infty\,.
\end{gather*}
\end{cor}
\begin{proof}
Let us observe that $H^0(M,L^k\otimes\cI(h^k))=H^0_{(2)}(M\setminus\Sigma,L^k,h^k,dv_M)$ and $h^L$ and $dv_M$ satisfy the conditions (A)-(C) of \cite{CM11}. Due to the asymptotic expansion \eqref{s1-e3mainterz} on $M\setminus\Sigma$ the conclusion follows from Theorems 1.1 and 4.3 from \cite{CM11}. 
\end{proof}
Corollary \ref{equi_big} generalizes \cite[Th.\,6.5]{CM11}, where the result was obtained under the hypothesis that $M$ is K\"ahler.

\par The results of this paper allow also to extend Tian's convergence theorem to the situation considered in Theorem \ref{tmain-semi1}.
Let $\{S^k_j\}^{d_k}_{j=1}$, $d_k\in\N\cup\{\infty\}$, be an orthonormal basis of $H^0_{(2)}(X,L^k\otimes K_M)$. We define the Fubini-Study currents $\gamma_k$ on $M$ in analogy to \eqref{e:gammap} 
as follows. Let $U$ be an open set and let $e_L$ be a local holomorphic frame for $L$ on $U$. Set
\begin{equation}\label{e:gammap1}
\gamma_k\mid_{_U}=\frac{\sqrt{-1}}{2\pi}\,\partial\overline\partial\log
\Big(\sum_{j=1}^{d_k}|f^k_j(x)|^2\Big)\,.
\end{equation}
The currents $\gamma_k$ don't depend on the choice of the local frame $e_L$ and are globally defined $(1,1)$ currents (see \cite[Lemma\,3.2 (ii)]{CM11}, \cite{MM07}, \cite{MM08a}). 
If $d_k<\infty$, then $\gamma_k=\Phi_k^*(\omega_{FS})$, where 
$\Phi_k:M\setminus B_k\to\mathbb{P}\big(H^0_{(2)}(X,L^k\otimes K_M)^*\big)$, $x\longmapsto\big\{s\in H^0_{(2)}(X,L^k\otimes K_M);s(x)=0\big\}$ is the Kodaira map.
We have moreover
\begin{equation}\label{ell00,3}
\frac{1}{k}\gamma_k-\frac{\sqrt{-1}}{2\pi}R^L=
-\frac{\sqrt{-1}}{2\pi k}\,\overline\partial\partial\log P_{k,K_M}(x)\,,\;\;\text{on $M$}\,.
\end{equation}
Hence Theorem \ref{tmain-semi1} implies immediately the following.
\begin{thm}\label{tian_semipos}
Let $(M,\Theta)$ be a complete K\"ahler manifold and $(L,h^L)$ be a holomorphic semi-positive line bundle over $M$, with smooth Hermitian metric $h^L$. Let $M(0)$ be the set where $(L,h^L)$ be a positive.
Then for any compact set $K\subset M(0)$, there exists $k_0$ such that for $k\geq k_0$ the base locus $B_k$ of $H^0(M,L^k\otimes K_M)$ is disjoint of $K$.  
Moreover, for any $\ell\in\N$, there exists $C_{\ell,K}>0$ independent of $k$ such that  
for $k\geq k_0$ the following holds
\begin{equation}\label{sz1,2} 
\Big|\frac{1}{k}\gamma_k-\frac{\sqrt{-1}}{2\pi} R^L\Big|_{\cC^\ell(K)}
\leqslant\frac{C_{\ell,K}}{k}\,\cdot 
\end{equation} 
\end{thm}
Theorem \ref{tian_semipos} can be used as above to prove the analogue of Corollary \ref{equi_big}.
Note that in \cite[Th.\,3.1\,(ii)]{CM13} the equidistribution of sections of adjoint bundles was actually obtained in the presence of singular Hermitian metrics.

\bigskip

\noindent
{\small\emph{
\textbf{Acknowledgments.} The methods of microlocal analysis used in this work are marked by the influence of Professor Johannes Sj\"{o}strand. The first-named author in particular wishes to express his hearty thanks for discussions on similar subjects and for giving us the idea of the
proof of Theorem~\ref{s1-mainex}.
}}

\end{document}